\documentclass[12pt,twoside]{amsart}
\usepackage{amssymb,amsmath,amsthm, amscd, enumerate, mathrsfs}
\usepackage{tikz-cd}
\usepackage{graphicx, hhline}
\usepackage[all]{xy}
\usepackage{color}
\usepackage{hyperref}
\usepackage{fancyhdr}
\usepackage{bbm}
\usepackage{comment}
\usepackage[top=30truemm,bottom=30truemm,left=30truemm,right=30truemm]{geometry}

\usepackage{todonotes}
\usepackage{mathrsfs}

\hypersetup{colorlinks=true}

\title[CM line bundles]{On positivity of CM line bundles on the moduli space of klt good minimal models with $\kappa=1$}
\author{Masafumi Hattori}
\date{\today}

\address{Masafumi Hattori \\ School of 
Mathematical Sciences, University of Nottingham, Nottingham, England}
\email{masafumi.hattori@nottingham.ac.uk}

\newtheorem{thm}{Theorem}[section]

\newtheorem{lem}[thm]{Lemma}
\newtheorem{dele}[thm]{Definition-Lemma}
\newtheorem{cor}[thm]{Corollary}
\newtheorem{prop}[thm]{Proposition}
\newtheorem{conj}[thm]{Conjecture}
\newtheorem{cond}[thm]{Condition}

\theoremstyle{definition}
\newtheorem{de}[thm]{Definition}
\newtheorem{ex}[thm]{Example}
\newtheorem{rem}[thm]{Remark}
\newtheorem{stup}[thm]{Setup}

\newtheorem*{ack}{Acknowledgments}

\newtheorem{setup}[thm]{Setup}

\newtheorem*{claim*}{Claim}
\newcommand{\hattori}[1]{\textsf{\scriptsize\color{purple} Hattori: {#1}}}
\begin{document}

\maketitle
\begin{abstract}
We study the positivity of CM line bundles on the coarse moduli space of Kawamata log terminal (klt) good minimal models with Kodaira dimension one. We prove that the seminormalization of the moduli space is quasi-projective under a mild assumption on the general fibers of good minimal models. Moreover, we show that the CM line bundle becomes ample after normalization. A key new ingredient is the construction of a moduli space of numerical equivalence classes, which is an extension of the work of Viehweg and allows us to bypass the failure of quasi-finiteness in the approach of the previous work by Hashizume and the author. 

We also establish the projectivity of the moduli space of $\epsilon$-stable quotients, which is introduced by Toda, to a projective space, which plays a central role in our method.
This particular situation is encompassed by our general framework of K-moduli of quasimaps.
\end{abstract}

\setcounter{tocdepth}{3}
\tableofcontents

\section{Introduction}
This paper is a sequel of the paper \cite{HH2}. We work over an algebraically closed field $\mathbbm{k}$ of characteristic zero unless otherwise stated.

\subsection{Background of K-moduli conjecture and positivity of the CM line bundle}

K-stability is an algebro-geometric notion introduced by Tian \cite{T} and Donaldson \cite{Dn2}, and the existence of a constant scalar curvature K\"ahler metric is expected to be equivalent to K-polystability, as Yau--Tian--Donaldson conjecture.

On the other hand, K-stability is also closely related to moduli theory. Odaka \cite{O2} proposed the following conjecture.
\begin{conj}[K-moduli conjecture]\label{conj--K-moduli}
A set of K-polystable polarized varieties forms a separated scheme. 
Furthermore, each connected component is a quasi-projective scheme with the ample CM line bundle.
\end{conj}

Here, the CM line bundle (see Subsection \ref{subsec--CM}) is a $\mathbb{Q}$-line bundle canonically associated with any scheme parameterizing polarized varieties, which was observed by Fujiki--Schumacher \cite{FS}, independently observed and named by Tian \cite{T} and algebro-geometrically studied by \cite{FR06,PT}.

One of the motivations for Conjecture \ref{conj--K-moduli} came from \cite{FS}, where they constructed a moduli space parameterizing polarized smooth manifolds with a unique cscK metric and showed that this admits a positive Weil--Petersson current representing the first Chern class of the CM line bundle.
Dervan--Naumann \cite{DN} weakened the uniqueness assumption on cscK metrics later.

Conjecture \ref{conj--K-moduli} has been settled in only a few cases, including canonically polarized varieties and log Fano pairs; the ampleness parts were established by \cite{PX} and \cite{XZ} (see \cite{kollar-moduli,Xu} and Subsection \ref{sec--rel--works}).
However, including the intermediate Kodaira dimension case, especially $\kappa=1$, Conjecture \ref{conj--K-moduli} has been a long-standing open problem.

Hashizume and the author \cite{HH} constructed the K-moduli space of {\it klt Calabi--Yau fibrations over curves}, i.e.~contractions $f\colon X\to C$ such that $K_X\sim_{\mathbb{Q},f}0$, including {\it Kawamata log terminal (klt) good minimal models with $\kappa=1$} (cf.~Definition \ref{de--gmm}). Indeed, any such model admits the Iitaka fibration, which coincides with $f$ as above in the case $\kappa=1$.
The authors also addressed the positivity of the CM line bundle when the fibrations have non-nef canonical divisor in \cite{Hat23,HH2}.
We only dealt with the projectivity of proper subspaces in \cite{Hat23} and the approach of \cite{HH2} relies on the K-moduli of quasimaps (see Subsection \ref{sec--intro--qmaps}).

\subsection{Main result}\label{subsec--main--result}

In this paper, we show the positivity of the CM line bundle over the K-moduli of good minimal models with $\kappa=1$ in the following sense.

\begin{thm}\label{thm--main--intro--rough}
Let $\mathscr{M}_{\kappa=1}$ be a connected component of the K-moduli space parameterizing polarized klt good minimal models $f\colon (X,A)\to C$ with $\kappa=1$.
Take the open subspace $\mathscr{M}\subset \mathscr{M}_{\kappa=1}$ such that $f$ belongs to $\mathscr{M}$ if and only if the generic fiber of $f$ lies in the normal locus of the moduli space of polarized Calabi--Yau varieties.
Let $\Lambda_{\mathrm{CM},t}$ be the CM line bundle with respect to the polarization $A+tK_X$ (cf.~Definition \ref{de--CM--level}).

Then, the following hold.
\begin{enumerate}
    \item The seminormalization $\mathscr{M}^{\mathrm{sn}}$ of $\mathscr{M}$ is quasi-projective.
    \item Let $\nu\colon \mathscr{M}^\nu\to\mathscr{M}$ be the normalization.
    Then $\nu^*\Lambda_{\mathrm{CM,t}}$ is ample for any sufficiently large $t>0$.
\end{enumerate}
\end{thm}

We will restate this theorem in Theorem \ref{thm--intro--main} in a precise form.
The seminormalization preserves the underlying topological space but this result does not imply the quasi-projectivity of the original moduli space (see Remark \ref{rem--koex} below).
We also confirm the positivity of the CM line bundle after normalization.
Moreover, we succeeded in eliminating the assumption that the general fiber of $f$ is Abelian or primitive symplectic from \cite{HH2} by applying the recent proof of the $\mathbf{B}$-semiampleness conjecture \cite[Theorem 1.5]{BFMT} (see Remark \ref{rem--ass}).

The principal obstacles are twofold. First, non‑properness prevents a direct use of the Nakai–Moishezon criterion for algebraic spaces \cite{kollar-moduli-stable-surface-proj,FM3}. Second, the morphism $\beta$ from $\mathscr{M}^{\mathrm{sn}}$ to the K-moduli of quasimaps $\mathscr{M}_{\mathrm{qmaps}}$ constructed in \cite{HH2} fails to be quasi‑finite in the $\kappa=1$ case \cite[Remark 4.15]{HH2}, essentially because the K‑moduli distinguishes {\it linear equivalence classes} while quasimaps only see {\it numerical equivalence}.
By the latter obstacle, we cannot directly deduce the positivity of the CM line bundle from the positivity of the CM line bundle on $\mathscr{M}_{\mathrm{qmaps}}$ via $\beta$.

Our key idea is to introduce a new {\it moduli space $\mathscr{N}$ of numerical equivalence classes} as a quotient of $\mathscr{M}$ in the spirit of \cite{viehweg91}, 
which allows us to circumvent the failure of quasi‑finiteness and to recover the positivity of the CM line bundle via $\mathscr{M}_{\mathrm{qmaps}}$.
Using $\mathscr{N}$, we construct the following commutative diagram:
\begin{equation}\label{diagram}
\begin{tikzcd}
\mathscr{M}^{\mathrm{sn}} \arrow[rd,"\beta"]\arrow[d,"\xi"] & \\
\mathscr{N}^{\mathrm{sn}} \arrow[r,"\alpha"]& \mathscr{M}_{\mathrm{qmaps}}.
\end{tikzcd}
\end{equation}
See Theorems \ref{thm--intro--const--N} and \ref{thm--technical-core} later for more details of the construction of $\mathscr{N}$ and this diagram, respectively.
In this diagram, the key points are the projectivity of $\xi$ and the quasi-finiteness of $\alpha$.
These enable us to check that $\mathscr{M}^{\mathrm{sn}}$ is quasi-projective and to compare $\Lambda_{\mathrm{CM},t}$ with the positivity of ample line bundles on $\mathscr{M}_{\mathrm{qmaps}}$.
More precisely, we will show that $\lim_{t\to\infty}\nu^*\Lambda_{\mathrm{CM},t}$ coincides with the pullback of an ample line bundle on $\mathscr{M}_{\mathrm{qmaps}}$ (see Theorem \ref{thm--intro--corresp}) and $\Lambda_{\mathrm{CM},t}$ is $\xi$-ample for any sufficiently large $t>0$ (see Theorem \ref{thm--intro--const--N}).
These observations show Theorem \ref{thm--main--intro--rough}.

Our approach extends beyond \cite{viehweg91} by working with $\mathbb{Q}$-Cartier Weil polarizations and by allowing klt singularities, and by constructing $\mathscr{N}$ as the coarse moduli of a Deligne–Mumford stack; this clarifies the modular meaning of $\mathscr{N}$, which enables us to construct $\alpha$. 
We will explain these ideas more precisely in Subsection \ref{sec--intro--technicalcore}.

The following observation, which is obtained by the above diagram, is also important.

\begin{rem}   
In the context of K\"ahler geometry, \cite{FS,DN} constructed the moduli space, which only distinguishes numerical classes of the polarization. 
Their moduli space corresponds to our $\mathscr{N}$.
Furthermore, they claimed the existence of a positive Weil--Petersson current and a CM line bundle on their moduli space.
However, the $\xi$-ampleness of $\Lambda_{\mathrm{CM},t}$ shows that $\Lambda_{\mathrm{CM},t}$ is not a pullback of any line bundle on $\mathscr{N}$ and thus we cannot define a CM line bundle on $\mathscr{N}$ that is compatible with the CM line bundle on $\mathscr{M}$. 
See Example \ref{ex--cm} for a simpler example illustrating this incompatibility.

Dervan informed the author that this subtlety can be resolved by modifying the gauge group in the usual “scalar curvature as a moment map” picture so that the universal family carries a natural line bundle. For details, see \cite{Dervan} on this correction.
\end{rem}

\subsection{Main result in a precise form}

Before restating Theorem \ref{thm--main--intro--rough}, we explain the moduli space constructed by \cite{HH} and its objects.
Fix $d,v,w\in\mathbb{Q}_{>0}$ and $u\in\mathbb{Q}_{\ne0}$.
Let $f\colon (X,A)\to C$ be a contraction of projective klt varieties such that
\begin{itemize}
    \item $C$ is a curve,
    \item $A$ is an ample line bundle, 
    \item $\mathrm{dim}\,X=d$, 
    \item there exists a $\mathbb{Q}$-line bundle such that $K_X\sim_{\mathbb{Q}}f^*L$ and $\mathrm{deg}_CL=u$, 
    \item $\mathrm{vol}(A|_F)=v$ for a general fiber $F$ of $f$, and
    \item $\mathrm{vol}(A)=w$.
\end{itemize}
When $u>0$, the above $f$ is a klt good minimal model with $\kappa=1$.
According to \cite[Theorem 1.4]{HH}, fixing a sufficiently divisible $r\in\mathbb{N}$ as \cite[Lemma 3.1]{HH}, there exists a separated algebraic space $M_{d,v,u,r,w}$ of finite type parameterizing uniformly adiabatically K-stable Calabi-Yau fibrations $f\colon (X,A)\to C$ with the above conditions.  
The space arises as the coarse moduli space of a Deligne--Mumford stack defined as in Definition \ref{defn--HH-moduli}.

we now restate Theorem \ref{thm--main--intro--rough} in a precise form.

\begin{thm}\label{thm--intro--main}
Fix an arbitrary sufficiently divisible $k\in\mathbb{Z}_{>0}$.
Let $V_1$ be the normal locus of the moduli space of polarized klt Calabi--Yau varieties $(Y,A_Y)$ of dimension $d-1$ and $\mathrm{vol}(A_Y)=k^{d-1}v$.    
Let $M_{d,v,u,w,V_1}\subset M_{d,v,u,r,w}$ be the locally closed reduced subspace such that $f\colon (X,A)\to C$ belongs to $M_{d,v,u,w,V_1}$ if and only if there exists a general fiber $F$ of $f$ such that $(F,kA|_F)$ belongs to $V_1$.

Then, the seminormalization $(M_{d,v,u,w,V_1})^{\mathrm{sn}}$ of $M_{d,v,u,w,V_1}$ is quasi-projective.
Furthermore, $\nu^*\Lambda_{\mathrm{CM},t}$ is ample for any sufficiently large $t>0$, where $\nu\colon(M_{d,v,u,w,V_1})^\nu\to M_{d,v,u,w,V_1}$ is the normalization. 
\end{thm}


Theorem \ref{thm--intro--main} immediately follows from Section \ref{sec--finish} and is implied by the stronger Theorem \ref{thm--ampleness-CM}. 
Note that $M^{W,I}_{d,v,u,w,p_a,V_1}$ appearing in Theorem \ref{thm--ampleness-CM} is an enlargement of $(M_{d,v,u,r,w,V_1})_{\mathrm{red}}$ up to quasi-finite morphisms (see the discussion after Definition \ref{de--WI}).
See Section \ref{sec--finish} for more details.
For the necessity of assumption on $V_1$, see Remark \ref{rem--ass} later.


\subsection{Quasimaps}\label{sec--intro--qmaps}


Before stating our result on quasimaps, we recall fundamental notions.
If $X\subset \mathbb{P}^N$ is a closed subscheme with $\mathrm{Cone}(X)\subset\mathbb{A}^{N+1}$ the affine cone,
a {\it quasimap} $q\colon C\to [\mathrm{Cone}(X)/\mathbb{G}_m]$ is defined as a morphism from a nodal curve to the quotient stack with some condition in this paper (see Definition \ref{de--qmaps}).
This object can be regarded as a curve with a linear system defining a rational map to $X$ and was discovered by Givental \cite{givental} and also studied in \cite{DP2}.
Now, we obtain an associated line bundle $\mathscr{L}$ on $C$ corresponding to the linear system.
Fix a weight $w\in\mathbb{Q}\cap (0,1]$.
Then, we can define the K-stability of $q$ with respect to $w$ as Definition \ref{de--K-st--qmaps}.
When $K_{C}+w\mathscr{L}$ is ample and $X=\mathbb{P}^N$, we observe that our quasimap is nothing but the $w$-stable quotient with the target $\mathbb{P}^N$, which was introduced by Toda \cite{Toda}.

We now state the existence of the K-moduli of quasimaps as a projective scheme. 

\begin{thm}[{Proposition \ref{lem--quasimap--stack}, Theorem \ref{thm--qm--cm--positive}}]\label{thm--intro--qmaps}
Fix $m\in\mathbb{Z}_{>0}$, $w\in\mathbb{Q}\cap[0,1]$ and $v\in\mathbb{Q}_{>0}$.
Let $\iota\colon X\hookrightarrow\mathbb{P}^N$ be a closed immersion of a normal variety.
Let $\mathcal{M}^{\mathrm{Kss,qm}}_{m,w,v,\iota}$ be a stack such that for any scheme $S$, the collection of objects $\mathcal{M}^{\mathrm{Kss,qm}}_{m,w,v,\iota}(S)$ is 
$$\left\{
 \vcenter{
 \xymatrix@C=12pt{
\mathcal{C}\ar[rr]^-{q}\ar[dr]_{\pi}&& [\mathrm{Cone}(X)_S/\mathbb{G}_{m,S}] \ar[dl]\\
&S
}
}
\;\middle|
\begin{array}{rl}
(i)&\text{$q$ is a family of K-stable quasimaps}\\
&\text{of weight $w$,}\\
(ii)&\text{$K_{\mathcal{C}/S}+w\mathscr{L}$ is $\pi$-ample, where $\mathscr{L}$ is}\\
&\text{the line bundle associated with $q$,}\\
(iii)&\text{$\mathrm{deg}_{\mathcal{C}_{\bar{s}}}(K_{\mathcal{C}_{\bar{s}}})+wm=v$ and $\mathrm{deg}(\mathscr{L}_{\bar{s}})=m$}\\
&\text{for any geometric point $\bar{s}\in S$}
\end{array}\right\},$$
and we set its arrows to be isomorphisms of families of quasimaps.  

Then $\mathcal{M}^{\mathrm{Kss,qm}}_{m,w,v,\iota}$ is a proper Deligne--Mumford stack of finite type with a projective coarse moduli space $M^{\mathrm{Kps,qm}}_{m,w,v,\iota}$.
Furthermore, $M^{\mathrm{Kps,qm}}_{m,w,v,\iota}$ admits an ample $\mathbb{Q}$-line bundle $\Lambda_{\mathrm{CM}}^{\mathrm{qmaps}}$.
\end{thm}
A key novelty of our result is that the coarse moduli space $M^{\mathrm{Kps,qm}}_{m,w,v,\iota}$ is projective even when $X=\mathbb{P}^N$.
Although the case where $X=\mathbb{P}^N$ was treated in \cite{Toda}, the projectivity of the resulting moduli space was not addressed there.

\subsection{Technical core of this paper}\label{sec--intro--technicalcore}

In this subsection, we explain the construction of the diagram \eqref{diagram} in more detail.
 First, we construct $\beta\colon\mathscr{M}^{\mathrm{sn}}\to\mathscr{M}_{\mathrm{qmaps}}$ following the same strategy as in \cite[Theorem 4.4]{HH2}.

 \begin{thm}\label{thm--intro--corresp}
 Notations as in Theorem \ref{thm--intro--main}.
There exist a closed immersion $\iota\colon\overline{V}^{\mathrm{BB}}\hookrightarrow\mathbb{P}^L$ of a normal variety $\overline{V}^{\mathrm{BB}}$ that contains $V_1$ as an open subset for some $L$, and a morphism $\overline{\beta_{v,w}}\colon(M_{d,v,u,w,p_a,V_1})^{\mathrm{sn}}\to M^{\mathrm{Kps,qm}}_{k(u-2p_a+2),\frac{1}{k},u,\iota}$, where $M_{d,v,u,r,w,p_a,V_1}$ is the largest open subset of $M_{d,v,u,r,w,V_1}$ such that the base curve $C$ of $f$ belonging to $M_{d,v,u,r,w,p_a,V_1}$ satisfies that $g(C)=p_a$ and $(M_{d,v,u,w,p_a,V_1})^{\mathrm{sn}}$ is the seminormalization of $M_{d,v,u,r,w,p_a,V_1}$.   

Furthermore, $\lim_{t\to\infty}\nu^*\Lambda_{\mathrm{CM,t}}$ exists as $\nu^*\overline{\beta_{v,w}}^*\Lambda^{\mathrm{qmaps}}_{\mathrm{CM}}$ for some $b\in\mathbb{Q}_{>0}$.
 \end{thm}

Theorem \ref{thm--intro--corresp} follows immediately from Proposition \ref{prop--approx--CM}, Theorems \ref{thm--canonical--bundle--formula-sin} and \ref{thm--previous-hh}. 
The construction of $\overline{V}^{\mathrm{BB}}$ is given in Theorem \ref{thm--canonical--bundle--formula-sin}.

As we remarked earlier, $\overline{\beta_{v,w}}$ is not quasi-finite in the case $u>0$ (\cite[Remark 4.15]{HH2}) and we cannot deduce Theorem \ref{thm--intro--main} directly from Theorem \ref{thm--intro--corresp}.
This is because quasimaps detect only numerical equivalence.
To overcome this issue, we introduce $\mathscr{N}$ a moduli space that only distinguishes the numerical classes in the spirit of \cite{viehweg91}.
Our construction extends Viehweg’s original approach by allowing klt singularities and
$\mathbb{Q}$-Cartier Weil polarizations, and by constructing $\mathscr{N}$ as the
coarse moduli space of a Deligne–Mumford stack.
This gives $\mathscr{N}$ a clear modular interpretation and enables the construction of $\alpha$.


Before stating the construction of $\mathscr{N}$, we explain the moduli of varieties with $\mathbb{Q}$-Cartier polarization briefly. 
For some sufficiently divisible $I$ such that $I|l$, we can take a closed algebraic subspace $M^{W,I}_{d,v,u,r,w}\subset M_{d,I^{d-1\cdot }v,u,r,I^d\cdot w}$ (cf.~
Definition \ref{de--Num--consruction}), which consists of isomorphic classes of $f\colon (X,A)\to C$ in $M_{d,I^{d-1\cdot }v,u,r,I^d\cdot w}$ with a $\mathbb{Q}$-Cartier Weil divisor $D$ such that $ID=A$.
Note that there is a natural quasi-finite morphism $\gamma\colon(M_{d,v,u,r,w})_{\mathrm{red}}\to(M^{W,I}_{d,v,u,r,w})_{\mathrm{red}}$.
We can regard $M^{W,I}_{d,v,u,r,w}$ as a space parameterizing varieties polarized by ample $\mathbb{Q}$-Cartier divisorial sheaves $\mathcal{O}_X(D)$.

Then, we construct the desired $\mathscr{N}$ as a quotient of $M^{W,I}_{d,v,u,r,w}$.

\begin{thm}\label{thm--intro--const--N}
Let $\mathcal{M}^{W,I}_{d,v,u,r,w}$ be the moduli stack, which has $M^{W,I}_{d,v,u,r,w}$ as the coarse moduli space.
Then there exist a Deligne--Mumford stack $\mathcal{N}^{W,I}_{d,v,u,r,w}$ of finite type with the coarse moduli space $N^{W,I}_{d,v,u,r,w}$ and a proper smooth surjective morphism
    $$\xi\colon \mathcal{M}^{W,I}_{d,v,u,r,w}\to \mathcal{N}^{W,I}_{d,v,u,r,w}$$
    such that the set of $\mathbbm{k}$-valued points $N^{W,I}_{d,v,u,r,w}(\mathbbm{k})$ is exhibited as $M^{W,I}_{d,v,u,r,w}(\mathbbm{k})/\equiv$, where $(X,A)\equiv(Y,B)$ if and only if there exists an isomorphism $\varphi\colon X\to Y$ such that $\varphi^*B-A$ is numerically trivial.

    Furthermore, $\Lambda_{\mathrm{CM},t}$ is $\overline{\xi}$-ample for any sufficiently large $t>0$, where $\overline{\xi}\colon M^{W,I}_{d,v,u,r,w}\to N^{W,I}_{d,v,u,r,w}$ is the induced morphism.
\end{thm}

Theorem \ref{thm--intro--const--N} follows from Corollary \ref{cor--N--closedpoint}, Proposition \ref{prop--unnec--u<0}, and Theorems \ref{thm:N--DMstack} and \ref{thm--rel--ample}.
See Definition \ref{de--Num--consruction} for more details.

Theorem \ref{thm--intro--const--N} extends the main result of \cite{viehweg91} to the case of moduli of varieties polarized by $\mathbb{Q}$-Cartier Weil divisors.
Using $N^{W,I}_{d,v,u,r,w}$, we can relate $M_{d,v,u,r,w}^{\mathrm{sn}}$ to the moduli of quasimaps, despite the lack of quasi-finiteness of $\overline{\beta_{v,w}}$.

Now, we can define  $M^{W,I}_{d,v,u,w,p_a,V_1}\subset M^{W,I}_{d,v,u,r,w}$ as a locally closed reduced subspace consisting of those objects whose general fiber lies in $V_1$ and whose base curve has the arithmetic genus $p_a$.
Then the following holds.

\begin{thm}\label{thm--technical-core}
There exists a locally closed subspace $N^{W,I}_{d,v,u,w,p_a,V_1}\subset N^{W,I}_{d,v,u,r,w}$ such that $$N^{W,I}_{d,v,u,w,p_a,V_1}\times_{N^{W,I}_{d,v,u,r,w}}M^{W,I}_{d,v,u,r,w}\cong M^{W,I}_{d,v,u,w,p_a,V_1}.$$    
Furthermore, there exists a quasi-finite morphism $$\overline{\alpha}\colon  (N^{W,I}_{d,v,u,w,p_a,V_1})^{\mathrm{sn}}\to M^{\mathrm{Kps,qm}}_{k(u-2p_a+2),\frac{1}{k},u,\iota}$$ such that $\overline{\alpha}\circ\overline{\xi}\circ\gamma^{\mathrm{sn}}=\overline{\beta}_{v,w}$, where $(N^{W,I}_{d,v,u,w,p_a,V_1})^{\mathrm{sn}}$ is the seminormalization and $\gamma^{\mathrm{sn}}\colon M_{d,v,u,w}^{\mathrm{sn}}\to(M^{W,I}_{d,v,u,r,w})^{\mathrm{sn}}$ is a natural quasi-finite morphism.

In particular, $(N^{W,I}_{d,v,u,w,p_a,V_1})^{\mathrm{sn}}$ and $(M_{d,v,u,w,V_1})^{\mathrm{sn}}$ are quasi-projective.
\end{thm}

Theorem \ref{thm--technical-core} follows from Lemma \ref{lem--N--V_1}, Proposition \ref{prop-q-proj--N} and Theorem \ref{thm--ampleness-CM}.
Theorem \ref{thm--technical-core} is the technical core to deduce Theorem \ref{thm--intro--main} of this paper because we can apply the framework of \cite{HH2} with the quasi-finiteness of $\overline{\alpha}$ instead of $\overline{\beta}_{v,w}$ due to this theorem.

\begin{rem}\label{rem--koex}

Note that we cannot deduce the quasi-projectivity or the ampleness of the CM line bundle of the original $M^{W,I}_{d,v,u,w,p_a,V_1}$ before taking the normalization or the seminormalization by the method of this paper for the same reason as \cite[Remark 4.14]{HH2}.  
See also \cite[Proposition 2]{Koex}, which is an example for a non-ample line bundle on a non-proper scheme with pullback to the normalization ample. 
\end{rem}


In the next remark, we explain how to eliminate the assumption on general fibers of $F$ and the necessity of the assumption on $V_1$ in Theorem \ref{thm--intro--main}.

\begin{rem}\label{rem--ass}
In \cite{HH2}, we concentrated only on the case where $F$ is an Abelian variety or holomorphic symplectic variety due to the Baily--Borel compacitification of the moduli space of polarized klt Calabi--Yau varieties established in these cases.
In this paper, due to the contribution of \cite{BFMT}, where they settle the $\mathbf{B}$-semiampleness conjecture, we can construct a compactification $\overline{V}^{\mathrm{BB}}$ of the normalization of the component as the ample model of the Hodge line bundle on the moduli space constructed by \cite{KX,Bir} (see Theorem \ref{thm--canonical--bundle--formula-sin}).
This compactification has good properties similar to the Baily--Borel compactification that enable us to deduce Theorem \ref{thm--intro--corresp}.


 The assumption on $V_1$ in Theorem \ref{thm--intro--main} comes from the construction of $\overline{V}^{\mathrm{BB}}$. 
 Since $V$ is not normal in general as Remark \ref{rem-bfmt} and \cite{Gro}, we cannot embed $V$ but $V_1$ into $\overline{V}^{\mathrm{BB}}$.
 Furthermore, it is still challenging to compactify the moduli space of polarized klt Calabi--Yau varieties such that the Hodge line bundle can be extended to an ample $\mathbb{Q}$-line bundle before normalization.
\end{rem}

\subsection{Related works}\label{sec--rel--works}

For the canonically polarized case, Patakfalvi--Xu \cite{PX} established the positivity of the CM line bundle using the results of Koll\'ar \cite{kollar-moduli-stable-surface-proj}, Fujino \cite{fujino-semi-positivity}, and Kov\'acs--Patakfalvi \cite{KP}.
For the log Fano case, after the work of Codogni--Patakfalvi \cite{CP,CP2} and Posva \cite{P} on positivity of the CM line bundle on proper algebraic subspaces in the K-stable locus, Xu--Zhuang \cite{XZ} finally proved the entire positivity of the CM line bundle on the K-moduli space of log Fano pairs assuming the properness of the K-moduli space, which was later shown by \cite{LXZ}.
Note that their techniques rely on the Nakai--Moishezon criterion for proper algebraic spaces \cite{kollar-moduli-stable-surface-proj,FM3}.

 Viehweg \cite{viehweg95} also dealt with the quasi-projectivity of the non-proper coarse moduli space parameterizing good minimal models with an arbitrary Kodaira dimension and only canonical singularities in \cite[Theorem 8.23]{viehweg95}.
 The work \cite{viehweg91} also only dealt with the moduli of varieties with canonical singularities.
Note that the ample line bundle, which he constructed, is different from the CM line bundle, in general.

\subsection{Structure of the paper}
In Section \ref{sec2}, we provide the reader with fundamental notions and facts on birational geometry, K-stability, klt--trivial fibrations, CM line bundle, and moduli stacks.

Sections \ref{sec--3}, \ref{sec-4}, and \ref{sec-5} are devoted to show Theorem \ref{thm--intro--const--N}.
Firstly, we will introduce the notion of flat $\mathbb{Q}$-Cartier divisorial sheaves (see Definition \ref{de--Q-Cartier--divisorial}) and construct the moduli space $\mathbf{W}^{\mathbb{Q}}\mathbf{Pic}_{X/S}$ of them as the Picard scheme $\mathbf{Pic}_{X/S}$ for line bundles in Section \ref{sec--3} in the case where $f\colon (X,\Delta)\to S$ is locally stable (cf.~Definition \ref{de--locally-stable}) and any geometric fiber is klt.
Secondly, we generalize the main result of \cite[Theorem 1.1]{Hat23} in Section \ref{sec-4}, which is a key step in discussing the relative ampleness of $\Lambda_{\mathrm{CM},t}$ in Theorem \ref{thm--intro--const--N}.
Thirdly, we show Theorem \ref{thm--intro--const--N} in Section \ref{sec-5}.
We construct ${N}^{W,I}_{d,v,u,r,w}$ as the coarse moduli space of a Deligne--Mumford stack $\mathcal{N}^{W,I}_{d,v,u,r,w}$ that is a quotient of the stack $\mathcal{M}^{W,I}_{d,v,u,r,w}$ with the coarse moduli space $M^{W,I}_{d,v,u,r,w}$ by a certain Abelian scheme. 

We deduce Theorem \ref{thm--intro--qmaps} in Section \ref{sec--quasimap} using the positivity of the CM line bundle on the KSBA moduli space \cite{PX}, the results on the degree of CM line bundles \cite{WX}, and the discussion in \cite{HH2}.

 Section \ref{sec-7}  is devoted to construct $\overline{V}^{\mathrm{BB}}$ and find $k$ as in Theorems \ref{thm--intro--corresp} and \ref{thm--technical-core}. 
We first discuss how to generalize a key step \cite[Theorem 2.50]{HH2} to the general klt polarized Calabi--Yau varieties.
We will construct $\overline{V}^{\mathrm{BB}}$ as we explained in Remark \ref{rem--ass}.

Finally, we show Theorems \ref{thm--intro--corresp}, \ref{thm--technical-core} and \ref{thm--intro--main} in Section \ref{sec--finish}.
First, Theorem \ref{thm--intro--corresp} follows from the same discussion as \cite{HH2} using the results of Sections \ref{sec--quasimap} and \ref{sec-7}.
Combining the results of Section \ref{sec-5} with Theorem \ref{thm--intro--corresp}, we show Proposition \ref{prop-q-proj--N}, which is the essential point of Theorem \ref{thm--technical-core}.
Finally, we deduce Theorem \ref{thm--ampleness-CM}, which is a generalization of Theorem \ref{thm--intro--main} and states the quasi-projectivity of $(M^{W,I}_{d,v,u,w,p_a,V_1})^{\mathrm{sn}}$ and the ampleness of the CM line bundle $\Lambda_{\mathrm{CM},t}$ on the normalization $(M^{W,I}_{d,v,u,w,p_a,V_1})^{\nu}$, applying Theorem \ref{thm--intro--const--N} and Proposition \ref{prop-q-proj--N}.

\begin{ack}
The author thanks Hamid Abban, Harold Blum, Ruadha\'i Dervan, Kenta Hashizume, Eiji Inoue, Kentaro Inoue, Xiaowei Jiang, Yuki Koto and Yuji Odaka for fruitful discussions and comments.
The author is supported by Royal Society International Collaboration Award ICA\textbackslash1\textbackslash231019.
\end{ack}

\section{Preliminaries}\label{sec2}

\subsection*{Notations and conventions}
\begin{enumerate}[(i)]
\item If we say that $X$ is a scheme, then we assume $X$ to be a locally Noetherian scheme over $\mathbbm{k}$.
If $X$ is further of finite type over $\mathbbm{k}$, separated, irreducible and reduced, then we say that $X$ is a variety.
For any point $x\in X$, let $\kappa(x)$ denote the residue field of the local ring $\mathcal{O}_{X,x}$.
That is, if we set $\mathfrak{m}_x$ as the maximal ideal of $\mathcal{O}_{X,x}$, $\kappa(x):=\mathcal{O}_{X,x}/\mathfrak{m}_x$.
\item
Let $X$ be a scheme.
We denote 
\[
X(S):=\mathrm{Hom}(S,X)
\]
and call this the set of all $S$-valued points of $X$.
If $S=\mathrm{Spec}\,R$ for some ring, then we will write $X(R)$.
If $S=\mathrm{Spec}\,\Omega$, where $\Omega$ is an algebraically closed field over $\mathbbm{k}$, then we call elements of $X(S)$ {\it geometric points} of $X$.
If the image of $\mathrm{Spec}\,\Omega\to X$ is $x\in X$, we denote this by $\bar{x}\in X$. 
\item Let $S$ be a Noetherian scheme. Take a finitely many locally closed subset $S_i$ of $S$ such that $S_i\cap S_j=\emptyset$ for any $i\ne j$.
Then we say that the natural morphism $\coprod S_i\to S$ is a {\it partial locally closed decomposition of} $S$.
If $S=\bigcup_{j}S_i$ as a set, then we say that the above morphism is a {\it locally closed decomposition of} $S$.
\item Let $f\colon X\to S$ be a projective flat morphism of schemes with a coherent sheaf $\mathscr{F}$ and an $f$-ample line bundle $A$ over $X$.
We say that {\it $\mathscr{F}$ has the Hilbert polynomial $P$ with respect to $A$} if $\chi(X_s,\mathscr{F}_s\otimes A_s^{\otimes m})=P(m)$ for any $m\in\mathbb{Z}$ and $s\in S$.
\item We say that a projective morphism of schemes $f\colon X\to Y$ is a {\it contraction} if $f_*\mathcal{O}_X\cong \mathcal{O}_Y$.
\item Let $X$ be a projective variety with $\mathbb{Q}$-Cartier $\mathbb{Q}$-divisors $D_1$ and $D_2$.
We say that $D_1$ and $D_2$ are {\it numerically equivalent} if $\mathrm{deg}_CD_1=\mathrm{deg}_CD_2$ for any irreducible curve in $X$.
If further $D_2=0$, then $D_1$ is called {\it numerically trivial}.
  \item For any two morphisms of Artin stacks $f\colon \mathscr{X}\to\mathscr{S}$ and $g\colon \mathscr{Y}\to \mathscr{S}$, we write $\mathscr{X}\times_{f,\mathscr{S},g}\mathscr{Y}$ the fiber product induced by $f$ and $g$.
  If there is no fear of confusion, we will write $\mathscr{X}\times_{f,\mathscr{S}}\mathscr{Y}$, $\mathscr{X}\times_{\mathscr{S},g}\mathscr{Y}$, or $\mathscr{X}\times_{\mathscr{S}}\mathscr{Y}$.
  If $\mathscr{S}=\mathrm{Spec}\,\mathbbm{k}$, then we will also omit $\mathscr{S}$.
    \item \label{Notations--(vi)} Let $f\colon \mathscr{X}\to \mathscr{S}$ be a morphism of Artin stacks.
    Let $g\colon T\to \mathscr{S}$ be a morphism from a scheme.
     Then we set $\mathscr{X}_T:=\mathscr{X}\times_{\mathscr{S}}T$ and $f_T\colon \mathscr{X}_T\to T$ as the morphism induced by $f$ and $g$.
    If there exists a morphism $\phi\colon \mathscr{X}\to \mathscr{Y}$ of Artin stacks over $\mathscr{S}$, then we write $\phi_T\colon \mathscr{X}_T\to\mathscr{Y}_T$ as the induced morphism. 
     Let $h\colon \mathscr{X}_T\to \mathscr{X}$ be the canonical morphism, which is denoted by $g_{\mathscr{X}}$ in this paper.
     If $L$ is a $\mathbb{Q}$-line bundle on $\mathscr{X}$, we set $L_T:=h^*L$.
     Suppose that $\mathscr{S}$ is a scheme.
     If $\mathscr{S}=\mathrm{Spec}\,R$ for some ring, we will write the above notions as $\mathscr{X}_R$ and so on.
     For any $s\in \mathscr{S}$, we denote $\mathscr{X}_s:=\mathscr{X}_T$ and $L_s=L_T$, where $T=\mathrm{Spec}\,\kappa(s)$.
     If $T=\mathrm{Spec}\,\overline{\kappa(s)}$, then we denote $\mathscr{X}_{\bar{s}}:=\mathscr{X}_T$ and $L_{\bar{s}}=L_T$.

     For details and the definition of Artin stacks and algebraic spaces, refer to \cite{Ols,HH,HH2}.
     Note that we can take the normalization of an Artin stack of finite type over a field. 
     See \cite[e.g.~Tag 0GMK]{Stacks} for details.
\end{enumerate}

\subsection{Birational geometry and K-stability}
We first recall the fundamental concepts of birational geometry and K-stability.

Let $X$ be a quasi projective variety with a $\mathbb{Q}$-divisor $\Delta$.
Suppose that $K_X+\Delta$ is $\mathbb{Q}$-Cartier.
Then, we say that $(X,\Delta)$ is a {\it sublog pair}.
For any prime divisor $F$ over $X$, we set the log discrepancy $A_{X,\Delta}(F)$ of $(X,\Delta)$ with respect to $F$ as
\[
A_{X,\Delta}(F):=1+\mathrm{ord}_F(K_Y-\pi^*(K_X+\Delta)),
\]
where $\pi\colon Y\to X$ is a log resolution of $(X,\Delta)$ such that $F$ is a prime divisor on $Y$.
If $A_{X,\Delta}(F)\ge 0$ (resp.~$A_{X,\Delta}(F)>0$) for any prime divisor $F$ over $X$, we say that $(X,\Delta)$ is {\it sublc} (resp.~{\it subklt}).
If $\Delta$ is further effective, we say that $(X,\Delta)$ is {\it log canonical} (resp.~{\it Kawamara log terminal}). We simply call this {\it lc} (resp.~{\it klt}).
Furthermore, we say that $X$ is {\it of lc (resp.~klt) type} if $(X,\Delta)$ is lc (resp.~klt) for some $\mathbb{Q}$-divisor $\Delta$.

\begin{de}[Good minimal models]\label{de--gmm}
    Let $(X,\Delta)$ be a projective klt pair.
    If $K_X+\Delta$ is semiample, then we say that $(X,\Delta)$ is a {\it good minimal model}.
    If there exists a contraction $f\colon X\to C$ to a smooth proper curve such that $f^*L\sim_{\mathbb{Q}}K_X+\Delta$ for some ample $\mathbb{Q}$-line bundle $L$ on $C$, then we say that $(X,\Delta)$ is a {\it good minimal model of Kodaira dimension one}.
    For simplicity, we call this a {\it good minimal model with $\kappa=1$}.
    See \cite{KM,BCHM} for more details of minimal models and the minimal model program.
\end{de}

\begin{de}[Slc schemes]
We say that a reduced scheme $Y$ of finite type over $\mathbbm{k}$ is a deminormal scheme if $Y$ satisfies Serre's $S_2$-condition and $Y$ has at worst nodal singularities at any codimension one point.

Let $Y$ be a quasi-projective deminormal scheme with a $\mathbb{Q}$-divisor $\Delta$ such that $ \mathrm{Supp}(\Delta)$ does not contain any irreducible closed subset $F$ with $F\subset\mathrm{Sing}(Y)$ and $\mathrm{codim}_Y(F)=1$.
If $K_Y+\Delta$ is $\mathbb{Q}$-Cartier, then we say that $(Y,\Delta)$ is a {\it deminormal sublog pair}.
If $\Delta$ is further effective, then $(Y,\Delta)$ is a {\it deminormal (log) pair}.

Let $\nu\colon X\to Y$ be the normalization morphism.
Then we set the {\it conductor divisor} as the reduced divisor $D_X\subset X$ defined as the ideal sheaf $\mathcal{H}om_{\mathcal{O}_Y}(\nu_*\mathcal{O}_X,\mathcal{O}_Y)$.
In this case, we say that a deminormal sublog pair $(Y,\Delta)$ is {\it subslc} if $K_Y+\Delta$ is $\mathbb{Q}$-Cartier and $(X,\nu_*^{-1}\Delta+D_X)$ is sublc.
If $\Delta$ is effective, then we say that $(Y,\Delta)$ is {\it semi log canonical (slc)}. See \cite[Section 5]{kollar-mmp} for more details.
We say that $Y$ is {\it of slc type} if there exists a $\mathbb{Q}$-divisor such that $(Y,\Delta)$ is slc.

\end{de}

Let $(X,\Delta)$ be a projective deminormal pair. 
We say that $(X,\Delta;L)$ is a {\it polarized log pair} if $L$ is an ample $\mathbb{Q}$-line bundle.
If $\Delta=0$, let $(X,L)$ denote $(X,0;L)$ for simplicity.

\begin{de}[Test configuration]\label{de--tc}
Let $(X,\Delta;L)$ be a polarized deminormal pair of dimension $n$.
We say that a pair $(\mathcal{X},\mathcal{L})$ is a {\it (semiample) test configuration} for polarized deminormal pair $(X,\Delta)$ if the following are satisfied.
\begin{enumerate}
\item There exists a flat proper morphism $\pi\colon \mathcal{X}\to\mathbb{A}^1$,
\item $\mathbb{G}_m$ acts on $\mathcal{X}$ equivariantly over $\mathbb{A}^1$, where $\mathbb{G}_m$ acts on $\mathbb{A}^1$ by multiplication,
\item $\mathcal{L}$ is a $\pi$-semiample $\mathbb{Q}$-line bundle such that $m\mathcal{L}$ admits a $\mathbb{G}_m$-linearization \cite[\S1.3]{GIT} for some positive integer $m$,
\item $(\mathcal{X}_1,\mathcal{L}_1)\cong(X,L)$ for $1\in \mathbb{A}^1$.
\end{enumerate}
Let $X_{\mathbb{A}^1}$ (resp.~$X_{\mathbb{P}^1}$) denote $X\times \mathbb{A}^1$ (resp.~$X\times\mathbb{P}^1$) with the trivial $\mathbb{G}_m$-action.
We say that $(\mathcal{X},\mathcal{L})$ is {\it trivial} if $\mathcal{X}$ is $\mathbb{G}_m$-equivariantly isomorphic to $X_{\mathbb{A}^1}$.  
Let $\nu\colon\tilde{X}\to X$ be the normalization morphism and $D_{\tilde{X}}$ the conductor divisor.
Let $\mu\colon\widetilde{\mathcal{X}}\to \mathcal{X}$ be the normalization.
We note that $(\widetilde{\mathcal{X}},\mu^*\mathcal{L})$ is a semiample test configuration for $(\tilde{X},\nu^*L)$.
We set the {\it Donaldson--Futaki invariant} as
\[
\mathrm{DF}_{\nu_*^{-1}\Delta+D_{\tilde{X}}}(\widetilde{\mathcal{X}},\mu^*\mathcal{L}):=\frac{1}{L^n}\left((K_{\widetilde{\mathcal{X}}^{\mathrm{c}}/\mathbb{P}^1}+\mathcal{D})\cdot(\mu^*\mathcal{L}^{\mathrm{c}})^n-\frac{n(K_X+\Delta)\cdot L^n}{(n+1)L^n}(\mu^*\mathcal{L}^{\mathrm{c}})^{n+1} \right),
\]
where $(\widetilde{\mathcal{X}}^{\mathrm{c}},\mu^*\mathcal{L}^{\mathrm{c}})$ is a normal scheme projective and flat over $\mathbb{P}^1$ obtained by gluing $(\widetilde{\mathcal{X}},\mu^*\mathcal{L})\to\mathbb{A}^1$ and $(X_{\mathbb{A}^1},L\times\mathbb{A}^1)\to \mathbb{P}^1\setminus\{0\}$, and $\mathcal{D}$ is the closure of $(\nu_*^{-1}\Delta+D_{\tilde{X}})\times\mathbb{G}_m$.

Take a $\mathbb{G}_m$-equivariant projective birational morphism $\sigma\colon \mathcal{Y}\to \widetilde{\mathcal{X}}^c$ such that there exists a $\mathbb{G}_m$-equivariant birational morphism $\rho\colon \mathcal{Y}\to X_{\mathbb{P}^1}$.
For any $\mathbb{Q}$-line bundle $T$ over $X$, we set the {\it non--Archimedean $\mathrm{J}^T$-functional} as
\begin{equation*}
        \mathcal{J}^{T,\mathrm{NA}}(\mathcal{X},\mathcal{L}):=\frac{1}{L^n}\left(\rho^*(L\times\mathbb{P}^1)\cdot\sigma^*(\mu^*\mathcal{L}^c)^n-\frac{nT\cdot L^{n-1}}{(n+1)L^n}(\mu^*\mathcal{L}^c)^{n+1} \right).
    \end{equation*}
    For more details, see \cite{BHJ,Hat2,Xu}, for example.
\end{de}

\begin{de}[K-stability and J-stability]\label{de--k-st}
Let $(X,\Delta;L)$ be a polarized deminormal pair.
We say that $(X,\Delta;L)$ is {\it uniformly K-stable} if there exists $\epsilon>0$ such that
\[
\mathrm{DF}_{\nu_*^{-1}\Delta+D_{\tilde{X}}}(\widetilde{\mathcal{X}},\mu^*\mathcal{L})\ge \epsilon \mathcal{J}^{L,\mathrm{NA}}(\mathcal{X},\mathcal{L}) 
\]
for any semiample test configuration $(\mathcal{X},\mathcal{L})$ for $(X,L)$, where $(\widetilde{\mathcal{X}},\mu^*\mathcal{L})$ is defined as Definition \ref{de--tc}.

Let $H$ be an ample $\mathbb{Q}$-line bundle on $X$.
We say that $(X,L)$ is {\it uniformly $\mathrm{J}^H$-stable} if there exists $\epsilon>0$ such that 
\[
\mathcal{J}^{H,\mathrm{NA}}(\mathcal{X},\mathcal{L})\ge \epsilon \mathcal{J}^{L,\mathrm{NA}}(\mathcal{X},\mathcal{L}) 
\]
for any semiample test configuration $(\mathcal{X},\mathcal{L})$ for $(X,L)$.
\end{de}

\begin{rem}
We note that the uniform K-stability in Definition \ref{de--k-st} is equivalent to the uniform K-stability in the original sense (cf.~\cite{De2,BHJ}).  
\end{rem}

Next, we will introduce the following notion, which detects the K-stability of Fano varieties by \cite{BlJ}.

\begin{de}[Delta invariant]
 Let $(X,\Delta;L)$ be a polarized klt log pair.  
 For any prime divisor $F$ over $X$, we set
 \[
 S_L(F):=\frac{1}{L^n}\int^\infty_0\mathrm{vol}(\pi^*L-xF)dx,
 \]
 where $\pi\colon Y\to X$ is a projective birational morphism such that $F$ is a Cartier divisor on $Y$.
 We note that the above invariant is independent of the choice of $\pi$.
 Then, we set 
 \[
 \delta_{(X,\Delta)}(L):=\inf_{F}\frac{A_{X,\Delta}(F)}{S_L(F)},
 \]
 where $F$ runs over all prime divisors over $X$.
 Note that $\delta_{(X,\Delta)}(L)>0$ by \cite[Theorem A]{BlJ}. See \cite{BlJ} for more details.
\end{de}

\begin{rem}
The delta invariant $\delta_{(X,\Delta)}(L)$ is firstly introduced by Fujita--Odaka \cite{FO} and Blum--Jonsson \cite{BlJ} showed that their delta invariant can be defined as above (see \cite[Theorem A]{BlJ}).
\end{rem}

Combining the uniform J-stability and the delta invariant, we introduce the following subclass of the uniform K-stability.

\begin{de}[Special K-stability {\cite[Definition 3.21]{CM}}]\label{de--special--k-st}
    Let $(X,\Delta;L)$ be a polarized klt log pair.
    We say that $(X,\Delta;L)$ is {\it specially K-stable} if $\delta_{(X,\Delta)}(L)L+K_X+\Delta$ is ample and $(X,L)$ is uniformly J$^{\delta_{(X,\Delta)}(L)L+K_X+\Delta}$-stable. 
\end{de}

\begin{thm}[{\cite[Corollary 3.23]{CM}}]
    Let $(X,\Delta;L)$ be a polarized klt log pair.
    If $(X,\Delta;L)$ is specially K-stable, then $(X,\Delta;L)$ is uniformly K-stable.
\end{thm}

\begin{de}[Klt--trivial fibrations]\label{de--klt--triv}
    Let $f\colon X\to Y$ be a contraction of quasi-projective normal varieties and $\Delta$ a $\mathbb{Q}$-divisor on $X$ such that $(X,\Delta)$ is sublc.
    We set $B_Y$ {\it the discriminant $\mathbb{Q}$-divisor} as follows.
    Let $\eta$ be an arbitrary point of $Y$ of codimension one and $F_{\eta}$ the Zariski closure of $\eta$.
    Then, we set $$\mathrm{mult}_{F_{\eta}}(B_Y):=1-\sup\{t\in \mathbb{Q}|(X,\Delta+tf^*F_{\eta})\text{ is lc around $f^{-1}(\eta)$}\}.$$
    Note that $B_Y$ is a well defined $\mathbb{Q}$-divisor since $F_{\eta}$ is Cartier around $\eta$.

Suppose further that $K_X+\Delta\sim_{\mathbb{Q},f}0$ and
$\mathrm{rank}\,f_*\mathcal{O}_X(\lceil \mathbf{A}^*(X,\Delta)\rceil)=1$,
    where the coherent sheaf $\mathcal{O}_X(\lceil \mathbf{A}^*(X,\Delta)\rceil)$ is defined as in \cite[Lemma 3.22]{fujino-bpf}. 
    In this case, we say that $f$ is a {\it sublc-trivial fibration}.
    Let $F$ be the geometric generic fiber of $f$ and let $b\in\mathbb{Z}_{>0}$ the smallest positive integer such that $h^0(F,\mathcal{O}_F(b(K_F+\Delta|_F)))\ne0$.
Then, we set a nonzero rational function $\varphi\in K(X)$ and a $\mathbb{Q}$-divisor $M_Y$ such that
\[
K_{X}+\Delta+\frac{1}{b}(\varphi)=f^*(K_Y+B_Y+M_Y),
\]
where $K(X)$ is the function field of $X$ and $(\varphi)$ is the principal divisor defined by $\varphi$.
We call $M_Y$ the {\it moduli $\mathbb{Q}$-divisor} on $Y$.
For details, see \cite[Definition 2.7]{HH2}.

For any birational morphism $h\colon Y'\to Y$, let $f'\colon X'\to Y'$ be a contraction of quasi-projective normal varieties and $\Delta'$ a $\mathbb{Q}$-divisor on $X'$ such that
\begin{itemize}
\item $X'$ and $X$ are birational,
\item $(X',\Delta')$ and $(X,\Delta)$ are log crepant.
\end{itemize}
Then, we can see that $f'$ is a sublc-trivial fibration.
We can define the moduli $\mathbb{Q}$-divisor $M_{Y'}$ and the discriminant $\mathbb{Q}$-divisor $B_{Y'}$ in the same way.
Then $h_*M_{Y'}=M_Y$, $h_*B_{Y'}=B_Y$, and 
\[
K_{Y'}+B_{Y'}+M_{Y'}=h^*(K_Y+B_Y+M_Y).
\]
They are independent of the choice of $(X',\Delta')$.
We say that if $M_{Y'}$ is nef and $h'^*M_{Y'}\sim M_{Y''}$ for any birational morphism $h'\colon Y''\to Y'$, then we say that $h\colon Y'\to Y$ is an {\it Ambro model}.
We note that for any sublc-trivial fibration, there exists an Ambro model by \cite{A,FG,fujino-slc-trivial}. See also \cite[Definition 2.7]{HH2}.

We say that if $f\colon(X,\Delta)\to S$ is a sublc-trivial fibration and $(X,\Delta)$ is klt (resp.~lc), then we say that $f$ is a {\it klt--trivial fibration} (resp.~{\it lc--trivial fibration}).
\end{de}

\begin{de}\label{de--unif--ad--kst}
    Let $f\colon (X,\Delta)\to C$ be a klt--trivial fibration over a projective smooth curve.
    Let $A$ be an $f$-ample $\mathbb{Q}$-line bundle on $X$.
    Then, we say that $f\colon (X,\Delta;A)\to C$ is a {\it polarized klt--trivial fibration}.
    If $\Delta=0$, we simply write it as $f\colon (X,A)\to C$.
    Let $L$ be an ample $\mathbb{Q}$-Cartier $\mathbb{Q}$-divisor on $C$.
    If there exist $\delta>0$ and $\epsilon_0>0$ such that 
    \[
    \mathrm{DF}_\Delta(\mathcal{X},\mathcal{M})\ge \delta \mathcal{J}^{\epsilon A+f^*L,\mathrm{NA}}(\mathcal{X},\mathcal{M})
    \]
    for any normal semiample test configuration $(\mathcal{X},\mathcal{M})$ for $(X,\epsilon A+f^*L)$ for any $\epsilon\in\mathbb{Q}\cap(0,\epsilon_0)$, then we say that $f$ is {\it uniformly adiabatically K-stable}.
    It is easy to see that this condition does not depend on the $\mathbb{Q}$-linear equivalence classes of $L$ and $A$ but on the numerical equivalence classes.
\end{de}

Then, we obtain the following, which has already been known in the case $\mathbbm{k}=\mathbb{C}$ by \cite[Theorem 1.1]{Hat}.

\begin{thm}\label{thm--kst--u>0--CY-fib}
    Let $f\colon (X,\Delta;A)\to C$ be a polarized klt--trivial fibration.
Then $f$ is uniformly adiabatically K-stable if and only if one of the following holds.
\begin{enumerate}
    \item $K_X+\Delta$ is nef.
    \item Suppose that $K_X+\Delta$ is not nef and let $B_C$ and $M_C$ be the discriminant and the moduli $\mathbb{Q}$-divisors, respectively.
    Then, $\mathrm{mult}_p(B)<\frac{\mathrm{deg}_C(M_C+B_C)}{2}$ for any $p\in C$.
\end{enumerate}
Furthermore, $(X,\Delta;\epsilon A+f^*L)$ is specially K-stable for any sufficiently small $\epsilon>0$ in both cases, where $L$ is an ample line bundle on $C$ such that $\mathrm{deg}_CL=1$.
\end{thm}

\begin{proof}
If $K_X+\Delta$ is numerically trivial, $f$ is uniformly adiabatically K-stable for any base field $\mathbbm{k}$ by the proof of \cite[Corollary 9.4]{BHJ}.
   Therefore, we may assume that $K_X+\Delta$ is not numerically trivial.
   We may further assume that $A$ is an ample line bundle.
   Fix $\mathrm{dim}\,X=d$, $n\in\mathbb{Z}_{>0}$ such that $n\Delta$ is a $\mathbb{Z}$-divisor, $A^{d-1}\cdot L=v$, $\mathrm{deg}_C(K_C+B_C+M_C)=u$, where $B_C$ and $M_C$ are the discriminant and the moduli $\mathbb{Q}$-divisors, respectively, and $\mathrm{vol}(A)=w$.
   Then, there exist constants $\epsilon_0$ and $\delta_0$ such that $(X,\Delta;\epsilon A+f^*L)$ is uniformly $\mathrm{J}^{(\delta_{(X,\Delta)}(\epsilon A+f^*L)-\delta_0)(\epsilon A+f^*L)+K_X+\Delta}$-stable for any $\epsilon\in\mathbb{Q}\cap(0,\epsilon_0)$ if $\mathbbm{k}=\mathbb{C}$ by \cite[Proposition 6.1]{HH} and the proof of \cite[Theorems 6.2 and 6.4]{HH}.  
   
   From now on, we work over an arbitrary algebraically closed field $\mathbbm{k}$ of characteristic zero and fix $d,n,u,v$, and $w$ as above for $f\colon(X,\Delta;A)\to C$, where we assume that $A$ is an ample line bundle.
   Fix also $\delta_0$ and $\epsilon_0$.
First, consider an extension of algebraically closed fields $\mathbbm{k}\subset \mathbbm{k}'$.
For any ample $\mathbb{Q}$-line bundle $H$ on $X$, $\delta_{(X,\Delta)}(H)=\delta_{(X_{\mathbbm{k}'},\Delta_{\mathbbm{k}'})}(H_{\mathbbm{k}'})$ by \cite[Proposition 4.15]{CP} applied to the case where $T=\mathrm{Spec}\,\mathbbm{k}$.
Moreover, for any test configuration $(\mathcal{X},\mathcal{H})$ for $(X,H)$ and $\mathbb{Q}$-line bundle $M$,
\[
\mathcal{J}^{M,\mathrm{NA}}(\mathcal{X},\mathcal{H})=\mathcal{J}^{M_{\mathbbm{k}'},\mathrm{NA}}(\mathcal{X}_{\mathbbm{k}'},\mathcal{H}_{\mathbbm{k}'}).
\]
If $\mathbbm{k}'=\mathbb{C}$, by what we have shown in the first paragraph, we see that $(X,\Delta;\epsilon A+f^*L)$ is uniformly $\mathrm{J}^{(\delta_{(X,\Delta)}(\epsilon A+f^*L)-\delta_0)(\epsilon A+f^*L)+K_X+\Delta}$-stable for any $\epsilon\in\mathbb{Q}\cap(0,\epsilon_0)$, because the intersection numbers, the dimension, and the coefficients of $\Delta$ are invariant under the base change of the fields.

Otherwise, take any test configuration $(\mathcal{X},\mathcal{M})$ for $(X,\Delta;\epsilon A+f^*L)$ for any $\epsilon\in\mathbb{Q}\cap(0,\epsilon_0)$.
Then, there exists a finitely generated field extension $\mathbb{Q}\subset \mathbb{K}$ such that $\mathbb{K}\subset \mathbbm{k}$ with a family of polarized klt--trivial fibrations over curves $f_{\mathbb{K}}\colon(X_{\mathbb{K}},\Delta_{\mathbb{K}};A_{\mathbb{K}})\to C_{\mathbb{K}}$ whose base change to $\mathbbm{k}$ coincides with the original $f$ by using \cite[Theorem 4.76]{kollar-moduli}.
Furthermore, we may assume that there exists a test configuration $(\mathcal{X}_{\mathbb{K}},\mathcal{M}_{\mathbb{K}})$ for $(X_{\mathbb{K}},\Delta_{\mathbb{K}};\epsilon A_{\mathbb{K}}+f_{\mathbb{K}}^*L_{\mathbb{K}})$ whose base change to $\mathbbm{k}$ coincides with the original $(\mathcal{X},\mathcal{M})$.
It is easy to check that there exists a homomorphism $\mathbb{K}\to\mathbb{C}$ since $\mathbb{C}$ is uncountable and fix it.
Then the algebraic closure $\mathbb{K}\subset \overline{\mathbb{K}}$ satisfies that $\overline{\mathbb{K}}\subset \mathbb{C}$ and $\overline{\mathbb{K}}\subset \mathbbm{k}$.
Replace $\mathbb{K}$ with $\overline{\mathbb{K}}$.
Then $(X_{\mathbb{K}},\Delta_{\mathbb{K}};\epsilon A_{\mathbb{K}}+f_{\mathbb{K}}^*L_{\mathbb{K}})$ is uniformly J$^{(\delta_{(X_{\mathbb{K}},\Delta_{\mathbb{K}})}(\epsilon A_{\mathbb{K}}+f_{\mathbb{K}}^*L_{\mathbb{K}})-\delta_0)(\epsilon A_{\mathbb{K}}+f_{\mathbb{K}}^*L_{\mathbb{K}})+K_{X_{\mathbb{K}}}+\Delta_{\mathbb{K}}}$-stable by what we have shown in the second paragraph of this proof.
Therefore, 
\[
\mathcal{J}^{(\delta_{(X_{\mathbb{K}},\Delta_{\mathbb{K}})}(\epsilon A_{\mathbb{K}}+f_{\mathbb{K}}^*L_{\mathbb{K}})-\delta_0)(\epsilon A_{\mathbb{K}}+f_{\mathbb{K}}^*L_{\mathbb{K}})+K_{X_{\mathbb{K}}}+\Delta_{\mathbb{K}},\mathrm{NA}}(\mathcal{X}_{\mathbb{K}},\mathcal{M}_{\mathbb{K}})>0.
\]
This shows the uniform $\mathrm{J}^{(\delta_{(X,\Delta)}(\epsilon A+f^*L)-\delta_0)(\epsilon A+f^*L)+K_X+\Delta}$-stability over $\mathbbm{k}$.
We complete the proof.
\end{proof}

Due to Theorem \ref{thm--kst--u>0--CY-fib}, we can freely change the algebraically closed field when we deal with uniform adiabatic K-stability.

\subsection{$\mathbb{Q}$-Gorenstein family}
Throughout this subsection, we use the following terminologies.
Let $f\colon X\to S$ be a flat and projective morphism.
    Fix $n\in\mathbb{Z}_{>0}$.
    Suppose that all geometric fibers are connected and deminormal of dimension $n$.
    
First, we recall the notion of the universal hull introduced by Koll\'ar.
\begin{de}[Universal hull {\cite[Definition 9.16]{kollar-moduli}}]\label{de--univ--hull}
Let $\mathscr{F}$ be a coherent sheaf on $X$.
Suppose that there exists an open subset $U$ such that $\mathrm{codim}_{X_{\bar{s}}}(X_{\bar{s}}\setminus U)\ge2$ for any geometric point $\bar{s}\in S$ and $\mathscr{F}|_U$ is locally free.
Let $\iota\colon U\to X$ be the canonical open immersion.
We note that if $S$ is a spectrum of a field, then we can set $\mathscr{F}^{[**]}:=\iota_*\mathscr{F}|_U$ as a coherent sheaf.
If there exist a coherent sheaf $\mathscr{G}$ on $X$ flat over $S$ and an $\mathcal{O}_{X}$-linear homomorphism $\varphi\colon \mathscr{F}\to \mathscr{G}$ such that $\varphi_s$ is isomorphic to $\mathscr{F}_s\to \mathscr{F}_s^{[**]}$, then we say that $\mathscr{G}$ is {\it the universal hull} of $\mathscr{F}$ and written as $\mathscr{F}^{[**]}$.
Note that if such a sheaf $\mathscr{G}$ as above exists, then $\mathscr{G}=\iota_*\mathscr{F}|_U$. See \cite[Section 9]{kollar-moduli} for more details (cf.~\cite[Section 2.4]{HH}).
In this paper, we write the universal hull of $\mathscr{F}^{\otimes m}\otimes\mathscr{G}$ as $\mathscr{F}^{[m]}[\otimes]\mathscr{G}^{[**]}$ for any $m\in\mathbb{Z}$.
\end{de}
    
    Note that we say that an effective Cartier divisor $D$ is a {\it relative Cartier divisor} if $D$ is flat over $S$ (cf.~\cite[Lemma 9.3.4]{FGA}).
\begin{de}[Relative Mumford divisor {\cite[Definition 4.68]{kollar-moduli}}]\label{de--rel--mum--div}
    We say that a closed subscheme $M$ of $X$ is an {\it effective relative Mumford divisor} if there exists an open subset $U\subset X$ satisfying the following:
    \begin{itemize}
\item for any geometric point $\bar{s}\in S$, $\mathrm{codim}_{X_{\bar{s}}}(X_{\bar{s}}\setminus U)\ge2$,
\item $M$ is the Zariski closure of $M|_U$, where $M|_U$ is flat over $S$, and
\item $X_{\bar{s}}$ is smooth at any generic point of $M_{\bar{s}}$ for any geometric point $\bar{s}\in S$.
    \end{itemize}
    We say that $\Delta$ is a {\it relative Mumford $\mathbb{Q}$-divisor} if $\Delta=\sum_{i=1}^la_iM_i$, where $M_i$ is an effective relative Mumford divisor for each $i$ and some $a_i\in\mathbb{Q}$.
    If $a_i$ is an integer for each $i$, we call $\Delta$ a {\it relative Mumford $\mathbb{Z}$-divisor}.
If $a_i$ is positive for each $i$, we say that $\Delta$ is {\it effective}.
    We say that a relative Mumford $\mathbb{Q}$-divisor $\Delta$ is {\it $\mathbb{Q}$-Cartier} if $m\Delta$ is Cartier for some $m\in\mathbb{Z}_{>0}$.
\end{de}

Let $\Delta$ be a relative Mumford $\mathbb{Q}$-divisor.
We note that for any morphism $g\colon T\to S$, we can set $g_X^*\Delta$ as the Zariski closure of $g_X|_{U_T}^*\Delta|_U$, where we take $U$ as in Definition \ref{de--rel--mum--div} and $g_X\colon X_T\to X$ is the canonical morphism.
Note that $g_X^*\Delta$ is also a relative Mumford $\mathbb{Q}$-divisor and we write it by $\Delta_T$.

\begin{de}[cf.~{\cite[Definition-Theorem 3.1]{kollar-moduli}}]\label{de--locally-stable}
Let $\Delta$ be an effective relative Mumford $\mathbb{Q}$-divisor on $X$.
Let $U$ be an open subset such that $(K_X+\Delta)|_U$ is $\mathbb{Q}$-Cartier and $\mathrm{codim}_{X_{s}}(X_s\setminus U)\ge2$ for any $s\in S$.
If there exists a universal hull $\omega_{X/S}^{[m]}(m\Delta)$ of a coherent sheaf on $X$ that is an extension of $\mathcal{O}_U(m(K_X+\Delta)|_U)$ for some $m\in\mathbb{Z}_{>0}$ and $\omega_{X/S}^{[m]}(m\Delta)$ is Cartier, then we say that $(X,\Delta)\to S$ is {\it log $\mathbb{Q}$-Gorenstein}.
If further $\Delta=0$, then we say that $f$ is  {\it $\mathbb{Q}$-Gorenstein}.
We denote $\omega_{X/S}(\Delta)$ as $K_{X/S}+\Delta$.

We say that a log $\mathbb{Q}$-Gorenstein family $f\colon (X,\Delta)\to S$ is called {\it locally stable} if $(X_s,\Delta_s)$ is slc for any geometric point $\bar{s}\in S$. 

If $f\colon (X,\Delta)\to S$ is a locally stable family and $K_{X/S}+\Delta$ is $f$-ample, then we say that $f$ is a {\it family of stable log pairs} or {\it stable family}.
\end{de}

\subsection{CM line bundle}\label{subsec--CM}

Let $f\colon X\to S$ be a projective and flat morphism of schemes.
Suppose that all geometric fibers are deminormal, connected, and of dimension $n$, and there exists an effective relative Mumford $\mathbb{Q}$-divisor $\Delta$ such that $(X,\Delta)\to S$ is log $\mathbb{Q}$-Gorenstein.
Let $L$ be a relatively ample $\mathbb{Q}$-line bundle on $X$ over $S$.
Then, we can define by \cite[Proposition 3.3]{Hat23} the {\it CM line bundle} as \cite[Definition 3.2]{Hat23}.
We mainly deal with case $\Delta=0$ and explain the definition here.

By \cite[Theorem 4]{KnMu}, we see that there exist unique line bundles $\lambda_0,\ldots,\lambda_{n+1}$ on $S$ such that
\[
\mathrm{det}(f_*\mathcal{O}_X(mL))=\bigotimes_{i=0}^{n+1}\lambda_i^{\otimes\binom{m}{i}}
\]
for any sufficiently large and divisible $m\in\mathbb{Z}_{>0}$.
On the other hand, we expand 
\[
\chi(X_s,\mathcal{O}_{X_s}(mL_s))=a_0m^n+a_{1}m^{n-1}+O(m^{n-2}).
\]
We set $\mu:=\frac{2a_1}{(n+1)!a_0^2}$. Then, we set the CM line bundle as the following $\mathbb{Q}$-line bundle
\[
\lambda_{\mathrm{CM},f}:=\lambda_{n+1}^{\otimes(\mu+\frac{1}{(n-1)!a_0}))}\otimes \lambda_{n}^{\otimes-\frac{2}{n!a_0}}.
\]
If $S$ is a proper smooth curve, then we note that
\[
\mathrm{deg}_S\lambda_{\mathrm{CM},f}=\mu L^{n+1}+\frac{1}{n!a_0}K_{X/S}\cdot L^n.
\]

We note that for any morphism $g\colon T\to S$ from a scheme,
\[
g^*\lambda_{\mathrm{CM},f}=\lambda_{\mathrm{CM},f_T}.
\]

\subsection{Representable functors}

Let $f\colon X\to S$ be a proper morphism and fix $A$ a $f$-very ample line bundle on $X$.
Consider a coherent sheaf $\mathscr{F}$ on $X$.
Set the following functor $\mathfrak{Quot}_{X/S}^{\mathscr{F},P}\colon \mathsf{Sch}_{S}^{\mathrm{op}}\to \mathsf{Sets}$ such that for any $S$-scheme $T$,
\[
\mathfrak{Quot}_{X/S}^{\mathscr{F},P}(T):=\left\{
 q\colon\mathscr{F}_T\to\mathscr{G}
\;\middle|
\begin{array}{l}
\text{$q$ is a surjective morphism to a coherent}\\
\text{sheaf $\mathscr{G}$ on $X_T$ flat over $T$ such that}\\
\text{$P(l)=\chi(\mathscr{G}_t\otimes A_t^{\otimes l})$ for any $t\in T$ and $l\in\mathbb{Z}$}
\end{array}\right\}/\sim,
\]
where $\sim$ is the equivalence relation such that for any two surjective morphisms $q\colon\mathscr{F}_T\to\mathscr{G}$ and $q'\colon\mathscr{F}_T\to\mathscr{G}'$ as above, $q\sim q'$ if and only if there exists an isomorphism $r\colon \mathscr{G}\to\mathscr{G}'$ such that $q'=r\circ q$.
If $\mathscr{F}=\mathcal{O}_X$, then let $\mathfrak{Hilb}_{X/S}^{P}$ denote $\mathfrak{Quot}_{X/S}^{\mathscr{F},P}$.

We note that $\mathfrak{Quot}_{X/S}^{\mathscr{F},P}$ is represented by a scheme $\mathbf{Quot}_{X/S}^{\mathscr{F},P}$ projective over $S$. On the other hand, $\mathfrak{Hilb}_{X/S}^{P}$ is represented by a scheme $\mathbf{Hilb}_{X/S}^{P}$ projective over $S$.
Note that the group scheme $\mathrm{Aut}^P_{X/S}$ of automorphisms $\varphi$ such that $X$ has the Hilbert polynomial $P$ with respect to $A\otimes\varphi^*A$ is represented by a locally closed scheme of $\mathbf{Hilb}_{X\times_SX/S}^{P}$.
Set $\mathrm{Aut}_{X/S}:=\coprod_P\mathrm{Aut}^P_{X/S}$.
When $S=\mathbbm{k}$, then the identity component of $\mathrm{Aut}_{X/S}$ is an algebraic group and we can set $\mathrm{dim}\,\mathrm{Aut}_{X/S}$ as the dimension of the identity component. 
See \cite[Section 5]{FGA} for details.

We recall the following notations in \cite[\S9]{FGA}.

\begin{de}\label{de--pic}
    Let $f\colon X\to S$ be a flat and projective morphism with only geometrically integral fibers.
Then, we set the following functor 
\[
\mathfrak{Pic}_{X/S}(T):=\{\text{line bundles on $X_T$}\}/\sim_T
\]
for any $S$-scheme $T$.
By \cite[Theorem 9.4.8]{FGA}, we see that the \'etale sheafification of $\mathfrak{Pic}_{X/S}$ is represented by a separated $S$-scheme $\mathbf{Pic}_{X/S}$.
On the other hand, $\mathbf{Pic}_{X/S}$ contains the torsion component $\mathbf{Pic}^\tau_{X/S}$ as an open and closed group subscheme by \cite[Theorem 9.6.19]{FGA}. 
Note that if $S$ is Noetherian, then $\mathbf{Pic}^\tau_{X/S}$ is a quasi-projective scheme over $S$.
For any $S$-scheme $T$, we will denote the multiplication of $\lambda$ and $\lambda'\in\mathbf{Pic}_{X/S}(T)$, and the inverse element of $\lambda$ as $\lambda\otimes\lambda'$ and $\lambda^{\otimes-1}$, respectively.
\end{de}

Let $\mathscr{F}$ and $\mathscr{E}$ be coherent sheaves on $X$. Suppose that $\mathscr{E}$ is flat over $S$.
Then, we set the following functor $\mathfrak{Hom}(\mathscr{F},\mathscr{E})\colon \mathsf{Sch}_S^{\mathrm{op}}\to \mathsf{Sets}$ as for any scheme $T$ over $S$,
\[
\mathfrak{Hom}(\mathscr{F},\mathscr{E})(T):=\{\varphi\colon \mathscr{F}_T\to \mathscr{E}_T|\varphi \text{ is $\mathcal{O}_{X_T}$-linear}\}.
\]
By \cite[Theorem 5.8]{FGA}, we see that $\mathfrak{Hom}(\mathscr{F},\mathscr{E})$ is represented by an affine scheme $\mathbf{Hom}_S(\mathscr{F},\mathscr{E})$ over $S$.

If $S=\mathrm{Spec}\,R$ for some ring, we will write the above notations as $\mathbf{Hilb}_{X/R}$ for example.
If $R=\mathbbm{k}$, then we will write $\mathbf{Pic}_{X}$ for example.

\subsection{Moduli stacks and spaces of polarized varieties}

In this subsection, we discuss some moduli stacks and their basic properties. 

\begin{dele}[{\cite[Lemma 2.15]{HH}}]\label{defn--pol}
Let $\mathfrak{Pol}$ be a category fibered in groupoids over $\mathsf{Sch}_{\mathbbm{k}}$ such that for any scheme $S$, the collection $\mathfrak{Pol}(S)$ of objects is
\[
\left\{
\begin{array}{l}
f\colon(\mathcal{X},\mathscr{A})\to S
\end{array}
\;\middle|
\begin{array}{rl}
\bullet\!\!&\text{$f$ is a surjective proper flat morphism of schemes whose}\\
&\text{fibers are geometrically normal and connected,}\\
\bullet\!\!&\text{$\mathscr{A}\in \mathbf{Pic}_{\mathcal{X}/S}(S)$ such that there exists an \'{e}tale covering}\\
&\text{$S'\to S$ by which the pullback of $\mathscr{A}$ to $\mathcal{X}\times_SS'$ is}\\
&\text{represented by a relatively ample line bundle over $S'$}
\end{array}\right\},
\]
and an arrow $(g,\alpha)\colon(f\colon(\mathcal{X},\mathscr{A})\to S)\to (f'\colon(\mathcal{X}',\mathscr{A}')\to S')$ is defined in the way that $\alpha\colon S\to S'$ is a morphism and $g\colon\mathcal{X}\to \mathcal{X}'\times_{S'}S$ is an isomorphism such that $g^*\alpha_{\mathcal{X}'}^*\mathscr{A}'=\mathscr{A}$ as elements of $\mathbf{Pic}_{\mathcal{X}'/S'}(S)$. 
Then, $\mathfrak{Pol}$ is indeed a stack over $\mathbbm{k}$.
\end{dele}

We will express moduli stacks of certain polarized varieties as substacks of $\mathfrak{Pol}$.
In this paper, we mainly use Deligne--Mumford stacks.
Note that a separated Deligne--Mumford stack of finite type over $\mathbbm{k}$ has a coarse moduli space by \cite{KeM}.
Moreover, we say that $f\colon(\mathcal{X},\mathscr{A})\to S\in\mathfrak{Pol}(S)$ is a {\it polarized family}.

\subsection{K-moduli of klt--trivial fibrations over curves}\label{subsec--k-moduli-cy-fib}

Fix $d\in\mathbb{Z}_{>0}$, $u\in\mathbb{Q}\setminus\{0\}$, and $v\in\mathbb{Q}_{>0}$.
Then we saw the boundedness of the following set
$$\mathfrak{Z}_{d, v,u}:=\left\{
\begin{array}{l}
f\colon (X,\Delta=0;A) \to C
\end{array}
\;\middle|
\begin{array}{rl}
(i)&\text{$f$ is a uniformly adiabatically}\\
&\text{K-stable polarized klt-trivial} \\
&\text{fibration over a curve $C$,}\\
(ii)&\text{${\rm dim}X=d$,}\\
(iii)&\text{$K_X\equiv uf^*H$ for some line bundle} \\
&\text{$H$ on $C$ such that $\mathrm{deg}\,H=1$,} \\
(iv)&\text{$A$ is an $f$-ample $\mathbb{Q}$-Cartier Weil}\\
&\text{divisor on $X$ such that $(H\cdot A^{d-1})=v$.}
\end{array}\right\}$$
in the sense of \cite[Theorem 1.5]{HH}.
For any $w\in\mathbb{Z}_{>0}$, we can consider the following set:
$$\mathfrak{Z}_{d, v,u,w}:=\left\{
\begin{array}{l}
f\colon (X,\Delta=0;A) \to C
\end{array}
\;\middle|
\begin{array}{rl}
(i)&\text{$f$ is a uniformly adiabatically}\\
&\text{K-stable polarized klt-trivial} \\
&\text{fibration over a curve $C$,}\\
(ii)&\text{${\rm dim}X=d$,}\\
(iii)&\text{$K_X\equiv uf^*H$ for some line bundle} \\
&\text{$H$ on $C$ such that $\mathrm{deg}\,H=1$,} \\
(iv)&\text{$A$ is an ample line bundle on}\\
&\text{$X$ such that $(H\cdot A^{d-1})=v$, and}\\
&\text{$\mathrm{vol}(A)=w$.}
\end{array}\right\}.$$
By \cite[Theorem 1.5]{HH}, we see that the set of $X$ of $f\in \mathfrak{Z}_{d, v,u,w}$ is bounded for any $w$ and for some $w$, the forgetful map $\mathfrak{Z}_{d, v,u,w}\to \mathfrak{Z}_{d, v,u}$ is surjective. 
Fix such $w$.
Next, fix a sufficiently divisible $r\in\mathbb{Z}_{>0}$ as \cite[Lemma 3.1]{HH}.
We recall the following moduli stack $\mathcal{M}_{d,v,u,r,w}$ in \cite{HH}.

\begin{de}\label{defn--HH-moduli}
 For any scheme $S$, $\mathcal{M}_{d,v,u,r,w}(S)$ is the following collection:
 $$\left\{
 \vcenter{
 \xymatrix@C=12pt{
(\mathcal{X},\mathscr{A})\ar[rr]^-{f}\ar[dr]_{\pi_{\mathcal{X}}}&& \mathcal{C} \ar[dl]\\
&S
}
}
\;\middle|
\begin{array}{rl}
(i)&\text{$\pi_{\mathcal{X}}$ is a flat projective morphism and $\mathcal{X}$ is a scheme,}\\
(ii)&\text{$\mathscr{A}\in\mathbf{Pic}_{\mathcal{X}/S}(S)$ such that $\mathscr{A}$ is $\pi_{\mathcal{X}}$-ample and}\\
&\text{$f_{\bar{s}}\colon (\mathcal{X}_{\bar{s}},0;\mathscr{A}_{\bar{s}})\to \mathcal{C}_{\bar{s}}$ belongs to $\mathfrak{Z}_{d, v,u,w}$ for any}\\ 
&\text{geometric point $\bar{s}\in S$,}\\
(iii)&\text{$\omega_{\mathcal{X}/S}^{[r]}$ exists as a line bundle,}\\
(iv)&\text{$\pi_{\mathcal{X}*}\omega_{\mathcal{X}/S}^{[lr]}$ is locally free and it generates}\\
&\text{$H^0(\mathcal{X}_{s}, \mathcal{O}_{\mathcal{X}_{s}}(lrK_{\mathcal{X}_{s}}))$ for any point $s\in S$ and any}\\
&\text{$l\in\mathbb{Z}_{>0}$,}\\
(v)&\text{$f$ is the ample model of $\omega_{\mathcal{X}/S}^{[r]}$ over $S$ and}\\
&\text{$(\mathcal{X}_{\overline{s}},0;\mathscr{A}_{\overline{s}})\to \mathcal{C}_{\overline{s}} \in \mathfrak{Z}_{d, v,u,w}$ for any geometric}\\
&\text{point $\overline{s}\in S$}
\end{array}\right\}.$$
Here, we define an isomorphism $\alpha\colon (f \colon (\mathcal{X},\mathscr{A})\to\mathcal{C})\to (f' \colon (\mathcal{X}',\mathscr{A}')\to\mathcal{C}')$ of any two objects of $\mathcal{M}_{d,v,u,r,w}(S)$ to be an $S$-isomorphism $\alpha\colon\mathcal{X}\to\mathcal{X}'$ such that $\alpha^*\mathscr{A}'=\mathscr{A}$ as elements of $\mathbf{Pic}_{\mathcal{X}/S}(S)$.
Note that the ample model is in the sense of \cite[Notation and convention (9)]{HH}.
\end{de}

By Theorem \ref{thm--kst--u>0--CY-fib}, we note that the above moduli stack coincides with the moduli stack we constructed in \cite[Theorem 5.1]{HH}. Thus, we have the following.
\begin{thm}[{\cite[Theorem 1.3]{HH}}]\label{thm--adiabatic--k-moduli}
    $\mathcal{M}_{d,v,u,r,w}$ constructed as above is a separated Deligne--Mumford stack of finite type over $\mathbbm{k}$.
    In particular, $\mathcal{M}_{d,v,u,r,w}$ has a coarse moduli space by \cite{KeM}.
\end{thm}

As noted in \cite[Remark 5.7]{HH}, we see that the reduced structure of $\mathcal{M}_{d,v,u,r,w}$ is independent of the choice of sufficiently large $r$. 
We can define the following closed substack of $\mathcal{M}_{d,v,u,r,w}$ independent from the choice of $r$.

\begin{de}\label{de--new--stack}
We set a substack $\mathcal{M}_{d,v,u,w}$ of $\mathfrak{Pol}$ as follows.
 For any scheme $S$, $\mathcal{M}_{d,v,u,w}(S)$ is the following collection:
 $$\left\{
\pi_{\mathcal{X}} \colon (\mathcal{X},\mathscr{A})\to S
\;\middle|
\begin{array}{rl}
(i)&\text{$\pi_{\mathcal{X}}$ is a flat projective morphism and $\mathcal{X}$ is a scheme,}\\
(ii)&\text{$\mathscr{A}\in\mathbf{Pic}_{\mathcal{X}/S}(S)$ such that $\mathscr{A}$ is $\pi_{\mathcal{X}}$-ample and}\\
&\text{$f_{\bar{s}}\colon (\mathcal{X}_{\bar{s}},0;\mathscr{A}_{\bar{s}})\to C$ belongs to $\mathfrak{Z}_{d, v,u,w}$ for any}\\ 
&\text{geometric point $\bar{s}\in S$ and for some proper}\\
&\text{smooth curve,}\\
(iii)&\text{$\omega_{\mathcal{X}/S}^{[l]}$ exists for any $l\in\mathbb{Z}$,}\\
(iv)&\text{For some $l\in\mathbb{Z}_{>0}$, $\pi_{\mathcal{X}*}\omega_{\mathcal{X}/S}^{[l]}$ is locally free and it}\\
&\text{generates $H^0(\mathcal{X}_{s}, \mathcal{O}_{\mathcal{X}_{s}}(lK_{\mathcal{X}_{s}}))$ for any point $s\in S$ and}\\
&\text{$\omega_{\mathcal{X}/S}^{[l]}$ is a line bundle.}
\end{array}\right\}.$$
\end{de}

We put the following lemma.
\begin{lem}\label{lem--dubois}
Let $f\colon (X,\Delta)\to S$ be a projective locally stable family over a locally Noetherian scheme such that $X_{\bar{s}}$ is connected for any geometric point $\bar{s}\in S$. 
Suppose that there exist an open subset $U\subset X$ and a coherent sheaf $\mathscr{G}$ such that
\begin{enumerate}
    \item $\omega_{X/S}|_U$ and $\mathscr{G}|_U$ are line bundles and for any $s\in S$, $\mathrm{codim}_{X_s}(X_s\setminus U)\ge2$,
    \item the universal hull $\omega_{X/S}^{[l_1]}[\otimes]\mathscr{G}^{[l_2]}$ of $\omega_{X/S}^{\otimes l_1}\otimes\mathscr{G}^{\otimes l_2}$ exists for any $l_1\in\{0,1\}$, $l_2\in\mathbb{Z}$ and there exists $r\in\mathbb{Z}_{>0}$ such that $\mathscr{G}^{[r]}$ is an $f$-globally generated line bundle.
\end{enumerate}

Then, for any $j\in\mathbb{Z}_{>0}$, $i\in\mathbb{Z}_{\ge0}$ and $s\in S$, $R^if_*\mathscr{G}^{[-j]}$ and $R^if_*(\omega_{X/S}^{[1]}[\otimes]\mathscr{G}^{[j]})$ are locally free and
\begin{align*}
R^if_*\mathscr{G}^{[-j]}\otimes\kappa(s)&\to H^i(X_s,\mathscr{G}_s^{[-j]}), \text{ and}\\
R^if_*(\omega_{X/S}^{[1]}[\otimes]\mathscr{G}^{[j]})\otimes\kappa(s)&\to H^i(X_s,\omega_{X_s}^{[1]}[\otimes]\mathscr{G}_s^{[j]})
\end{align*}
is bijective.
\end{lem}

\begin{proof}
    To show the assertion, we may assume that $S$ is affine and we can take for any $m\in\mathbb{Z}_{>0}$ a general section $h\in H^0(X,\mathscr{G}^{[mr]})$ such that $\mathrm{div}(h)$ is flat over $S$.
    Using $h$, we obtain a $\mu_{mr}$-cover $\pi\colon Y:=\mathbf{Spec}_X(\oplus_{j=0}^{mr-1}\mathscr{G}^{[-j]})\to X$.
    We see that $Y\to S$ has only geometric fibers of slc type.
    Applying \cite[Theorem 2.62]{kollar-moduli} to $f\circ \pi$, we see that $R^i(f\circ \pi)_*\mathcal{O}_Y\cong\oplus_{j=0}^{mr-1}R^if_*\mathscr{G}^{[-j]}$ is locally free and $\oplus_{j=0}^{mr-1}R^if_*\mathscr{G}^{[-j]}\to \oplus_{j=0}^{mr-1}H^i(X_s,\mathscr{G}_s^{[-j]})$ is surjective.
    Thus, we obtain the assertion for $R^if_*\mathscr{G}^{[-j]}$. 

    From now on, we deal with $R^if_*(\omega_{X/S}^{[1]}[\otimes]\mathscr{G}^{[j]})$.
    Suppose that $f$ is of relative dimension $n$.
    We deduce the assertion by induction on $n$.
    Note that the case where $n=1$ immediately follows from \cite[Corollary 2.69]{kollar-moduli}.
    Therefore, we may assume that $n\ge2$.
    By \cite[2.68.3]{kollar-moduli}, there exists an isomorphism
    \[
    f_*\mathcal{H}om_{\mathcal{O}_X}(F,\omega_{X/S})\cong \mathcal{H}om_{\mathcal{O}_S}(R^nf_*F,\mathcal{O}_S)
    \]
    for any $F$ that is functorial on $F$.
    Here, we note that $\mathcal{H}om_{\mathcal{O}_X}(\mathscr{G}^{[j]},\omega_{X/S})\cong\omega_{X/S}^{[1]}[\otimes]\mathscr{G}^{[-j]}$ for any $j\in\mathbb{Z}$.
    Indeed, let $\iota\colon U\hookrightarrow X$ be the inclusion.
    Then, it is easy to see that 
    \[
    \mathcal{H}om_{\mathcal{O}_X}(\mathscr{G}^{[j]},\omega_{X/S})=\iota_*\mathcal{H}om_{\mathcal{O}_U}(\mathscr{G}^{[j]}|_U,\omega_{X/S}|_U)
    \]
    by the same argument of \cite[Corollary 10.8]{kollar-moduli}.
    Since $\mathcal{H}om_{\mathcal{O}_U}(\mathscr{G}^{[j]}|_U,\omega_{X/S}|_U)\cong\omega_{X/S}|_U\otimes\mathscr{G}|_U^{\otimes j}$, we see that $\mathcal{H}om_{\mathcal{O}_X}(\mathscr{G}^{[j]},\omega_{X/S})\cong\omega_{X/S}^{[1]}[\otimes]\mathscr{G}^{[-j]}$ by Definition \ref{de--univ--hull}.
    Since $R^nf_*\mathscr{G}^{[-j]}$ is locally free and $R^nf_*\mathscr{G}^{[-j]}\otimes\kappa(s)\cong H^n(X_s,\mathscr{G}_s^{[-j]})$ for any $j\ge0$ and $s\in S$, we see that $f_*(\omega_{X/S}^{[1]}[\otimes]\mathscr{G}^{[j]})$ is also locally free and $f_*(\omega_{X/S}^{[1]}[\otimes]\mathscr{G}^{[j]})\otimes \kappa(s)\cong H^0(X_s,\omega_{X_s}^{[1]}[\otimes]\mathscr{G}^{[j]})$ for any $j\in\mathbb{Z}_{>0}$ and $s\in S$.
    Take an $f$-ample line bundle $M$ such that $H^i(X_s,\omega_{X_s}[\otimes]\mathscr{G}_s^{[j]}\otimes M_s)=0$ for any $s\in S$ and $i,j>0$. 
   By \cite[Corollaries 10.12 and 10.19]{kollar-moduli}, locally on $S$, we can take a relative Cartier divisor $H\sim M$ such that $f|_H\colon H\to S$ satisfies that \begin{enumerate}
    \item $\omega_{H/S}|_{U_H}$ and $(\mathscr{G}|_H)|_{U_H}$ are line bundles and for any $s\in S$, $\mathrm{codim}_{X_s}(X_s\setminus U)\ge2$,
    \item $f|_H\colon (H,\Delta|_H)\to S$ is locally stable and of dimension $n-1$,
    \item the universal hull $\omega_{H/S}^{[l_1]}[\otimes]\mathscr{G}|_H^{[l_2]}$ of $\omega_{H/S}^{\otimes l_1}\otimes\mathscr{G}|_H^{\otimes l_2}$ exists for any $l_1\in\{0,1\}$, $l_2\in\mathbb{Z}$ and there exists $r\in\mathbb{Z}_{>0}$ such that $\mathscr{G}|_H^{[r]}$ is an $f$-globally generated line bundle,
\end{enumerate} 
where $U_H:=U\times_XH$.
By \cite[III, Theorem 12.11]{Ha} and the exact sequence
\[
0\to \omega_{X/S}^{[1]}[\otimes]\mathscr{G}^{[j]}\to \omega_{X/S}^{[1]}[\otimes]\mathscr{G}^{[j]}\otimes M\to \omega_{H/S}^{[1]}[\otimes]\mathscr{G}|_H^{[j]}\to0,
\]
we see that $R^{i-1}(f|_H)_*(\omega_{H/S}^{[1]}[\otimes]\mathscr{G}|_H^{[j]})\cong R^if_*(\omega_{X/S}^{[1]}[\otimes]\mathscr{G}^{[j]})$ for any $i\ge2$.
By induction hypothesis, it suffices to show the assertion for $R^1f_*(\omega_{X/S}^{[1]}[\otimes]\mathscr{G}^{[j]})$.
By the above exact sequence, we obtain the following exact sequence:
\[
0\to f_*\omega_{X/S}^{[1]}[\otimes]\mathscr{G}^{[j]}\to f_*\omega_{X/S}^{[1]}[\otimes]\mathscr{G}^{[j]}\otimes M\to f_*\omega_{H/S}^{[1]}[\otimes]\mathscr{G}|_H^{[j]}\to R^1f_*(\omega_{X/S}^{[1]}[\otimes]\mathscr{G}^{[j]})\to0.
\]
In the above exact sequence, if we replace $S$ with $S_d:=\mathrm{Spec}(\mathcal{O}_{S,s}/\mathfrak{m}^d_s)$ for any $d\ge0$, then $H^1(X_d,\omega_{X_d/S_d}^{[1]}[\otimes]\mathscr{G}|_{X_d}^{[j]})$ is free over $A_d:=\mathcal{O}_{S,s}/\mathfrak{m}^d_s$ by \cite[2.70]{kollar-moduli}, where $\mathfrak{m}_s$ is the maximal ideal of $\mathcal{O}_{S,s}$ and $X_d:=X\times_SS_d$.
Moreover, it is easy to check that $H^1(X_d,\omega_{X_d/S_d}^{[1]}[\otimes]\mathscr{G}|_{X_d}^{[j]})\otimes_{A_d}\kappa(s)\cong H^1(X_s,\omega_{X_s}^{[1]}[\otimes]\mathscr{G}_s^{[j]})$.
This implies that 
\[
\varprojlim H^1(X_d,\omega_{X_d/S_d}^{[1]}[\otimes]\mathscr{G}|_{X_d}^{[j]})\to H^1(X_s,\omega_{X_s}^{[1]}[\otimes]\mathscr{G}_s^{[j]})
\]
is surjective by \cite[III, Theorem 12.11]{Ha}.
By \cite[III, Theorem 11.1]{Ha}, $$\widehat{R^1f_*(\omega_{X/S}^{[1]}[\otimes]\mathscr{G}^{[j]})}\cong \varprojlim H^1(X_d,\omega_{X_d/S_d}^{[1]}[\otimes]\mathscr{G}|_{X_d}^{[ j]}),$$ where $\widehat{R^1f_*(\omega_{X/S}^{[1]}[\otimes]\mathscr{G}^{[j]})}$ is the completion of $R^1f_*(\omega_{X/S}^{[1]}[\otimes]\mathscr{G}^{[j]})$ in the $\mathfrak{m}_s$-adic topology.
Therefore, 
\[
R^1f_*(\omega_{X/S}^{[1]}[\otimes]\mathscr{G}^{[l]})\to H^1(X_s,\omega_{X_s}^{[1]}[\otimes]\mathscr{G}_s^{[l]})
\]
is surjective.
By \cite[III, Theorem 12.11]{Ha}, we see that $R^1f_*(\omega_{X/S}^{[1]}[\otimes]\mathscr{G}^{[l]})$ is locally free.
We complete the proof.
\end{proof}

This lemma shows the following.

\begin{prop}\label{prop--univ--closed--r}
$\mathcal{M}_{d,v,u,w}$ is isomorphic to a closed substack of $\mathcal{M}_{d,v,u,r,w}$ that contains $(\mathcal{M}_{d,v,u,r,w})_{\mathrm{red}}$ for any sufficiently divisible $r\in\mathbb{Z}_{>0}$.
In particular, $\mathcal{M}_{d,v,u,w}$ is a separated Deligne--Mumford stack that admits a coarse moduli space $M_{d,v,u,w}$.
\end{prop}

\begin{proof}
For $d,u,v$, we set $r$ as \cite[Lemma 3.1]{HH}.   
 For any object $\pi_{\mathcal{X}}\colon \mathcal{X}\to S\in \mathcal{M}_{d,v,u,w}(S)$, the condition (iv) of Definition \ref{de--new--stack} implies the $\pi_{\mathcal{X}}$-semiampleness of $\omega_{\mathcal{X}/S}$.
 By Lemma \ref{lem--dubois}, $\pi_{\mathcal{X}*}\omega_{\mathcal{X}/S}^{[l]}$ is locally free and generates $H^0(\mathcal{X}_{s}, \mathcal{O}_{\mathcal{X}_{s}}(lK_{\mathcal{X}_{s}}))$ for any point $s\in S$ and $l\in\mathbb{Z}$.
 By the proof of \cite[Theorem 5.1]{HH}, it is easy to check that $\pi_{\mathcal{X}}$ satisfies condition (v) of Definition \ref{defn--HH-moduli}.
 On the other hand, it is easy to see that $(\mathcal{M}_{d,v,u,r,w})_{\mathrm{red}}$ satisfies condition (iii) of Definition \ref{de--new--stack} by \cite[Corollary 5.25]{KM} and \cite[Theorem 9.40]{kollar-moduli}.
 Set a substack $\mathcal{M}^{\mathrm{K}}_{d,v,u,r,w}$ of $\mathcal{M}_{d,v,u,r,w}$ such that $\mathcal{M}^{\mathrm{K}}_{d,v,u,r,w}(S)$ is the following collection:
 $$\left\{
 \vcenter{
 \xymatrix@C=12pt{
(\mathcal{X},\mathscr{A})\ar[rr]^-{f}\ar[dr]_{\pi_{\mathcal{X}}}&& \mathcal{C} \ar[dl]\\
&S
}
}
\;\middle|
\begin{array}{rl}
\text{$\pi_{\mathcal{X}}$ satisfies condition (iii) of Definition \ref{de--new--stack}.}
\end{array}\right\}.$$
Then we see that $\mathcal{M}^{\mathrm{K}}_{d,v,u,r,w}$ is a closed substack of $\mathcal{M}_{d,v,u,r,w}$ that contains $(\mathcal{M}_{d,v,u,r,w})_{\mathrm{red}}$ by \cite[Theorem 9.40]{kollar-moduli}.
It is easy to see that there exists a natural morphism $\mathcal{M}^{\mathrm{K}}_{d,v,u,r,w}\to\mathcal{M}_{d,v,u,w}$.
It is also easy to check that this defines a category equivalence because of the choice of $r$.
We complete the proof.
\end{proof}

We denote the reduced structure by $(\mathcal{M}_{d,v,u,w})_{\mathrm{red}}$, which is also the reduced structure of $\mathcal{M}_{d,v,u,r,w}$ by Proposition \ref{prop--univ--closed--r}. 

We note that for any $w'\ge w$, we can consider a morphism $$\gamma\colon \mathcal{M}_{d,v,u,w'}\to \mathcal{M}_{d,v,u,w'+druv}$$ by sending $(\mathcal{X},\mathscr{A})\to \mathcal{C}\in \mathcal{M}_{d,v,u,w'}(S)$ to $(\mathcal{X},\mathscr{A}\otimes\omega_{\mathcal{X}/S}^{[r]})\to \mathcal{C}\in \mathcal{M}_{d,v,u,w'+druv}(S)$ for the choice of $r$ as \cite[Lemma 3.1]{HH}.
By \cite[Theorem 1.5]{HH}, we see that $\gamma$ is an isomorphism of Deligne--Mumford stacks.
This means that it is essential to consider $\mathcal{M}_{d,v,u,r,w}$.

We put the following lemma, which we make use of in Section \ref{sec--finish}.

\begin{lem}\label{lem--HH2--de--yatta}
Suppose that $u>0$.
There exists $l_0\in\mathbb{Z}_{>0}$ depending only on $d,u,v,w$ that satisfies the following.
Let $S$ be an arbitrary reduced scheme and $f\colon (\mathcal{X},\mathscr{A})\to \mathcal{C}$ an object of $\mathcal{M}_{d,v,u,w}(S)$.  
Then $f$ is flat, $f_*\mathcal{O}_{\mathcal{X}}\cong\mathcal{O}_{\mathcal{C}}$, and $f_*\omega_{\mathcal{X}/S}^{[l_0]}$ is a line bundle such that $f^*f_*\omega_{\mathcal{X}/S}^{[l_0]}\cong \omega_{\mathcal{X}/S}^{[l_0]}$.
\end{lem}

\begin{proof}
    This follows from the same argument as \cite[Lemmas 2.60 and 2.61]{HH2}.
\end{proof}

From now on, we discuss the CM line bundle on $(M_{d,v,u,w})_{\mathrm{red}}$.

\begin{de}\label{de--CM--level}
Fix $w\in\mathbb{Z}_{>0}$.
We note that $\mathcal{M}_{d,v,u,r,w}$ can be obtained as a quotient stack $[H/G]$, where $H$ is a quasi-projective scheme and $G$ is a smooth semisimple linear algebraic group (cf.~\cite[Proof of Theorem 5.1]{HH}).
Let $H\to \mathcal{M}_{d,v,u,r,w}$ be the canonical morphism and let $(\mathscr{U},\mathscr{A})\to H$ be the family corresponding to this morphism.
If we choose a suitable $w$, then we see that $(\mathscr{U}_{\bar{s}},\mathscr{A}_{\bar{s}})$ is specially K-stable (cf.~\cite[Corollary 1.7]{HH}).
Let $t\in\mathbb{Q}_{>0}$ and consider $(\mathscr{U},m(\mathscr{A}+tK_{\mathscr{U}/H}))$, where $m\in\mathbb{Z}_{>0}$ such that $mt\in r\cdot\mathbb{Z}$.
Then, we set the CM line bundle on $H$ as
\[
\lambda_{\mathrm{CM},(\mathscr{U},m(\mathscr{A}+tK_{\mathscr{U}/H}))}.
\]
We note that $\lambda_{\mathrm{CM},(\mathscr{U},m(\mathscr{A}+tK_{\mathscr{U}/H}))}$ is independent of the choice of $m$.
By \cite[Subsection 2.5.1]{HH2}, we see that $\lambda_{\mathrm{CM},(\mathscr{U},m(\mathscr{A}+tK_{\mathscr{U}/H}))}^{\otimes m'}$ admits a natural $G$-linearlization for any sufficiently divisible $m'\in\mathbb{Z}_{>0}$.
Therefore, we obtain that there exists a $\mathbb{Q}$-line bundle $\lambda_{\mathrm{CM},t}$ on $\mathcal{M}_{d,v,u,r,w}$ whose pullback to $H$ coincides with $\lambda_{\mathrm{CM},(\mathscr{U},m(\mathscr{A}+tK_{\mathscr{U}/H}))}$.
Let $\Lambda_{\mathrm{CM},t}$ denote the $\mathbb{Q}$-line bundle on $M_{d,v,u,r,w}$ whose pullback to $\mathcal{M}_{d,v,u,r,w}$ coincides with $\lambda_{\mathrm{CM},t}$.
They are uniquely determined up to $\mathbb{Q}$-linear equivalence.
Let $\nu\colon\mathcal{M}^\nu_{d,v,u,w}\to  \mathcal{M}_{d,v,u,r,w}$ be the normalization.
We note that the normalization is independent of the choice of $r$ by Proposition \ref{prop--univ--closed--r} and we omit $r$.
Let $M^\nu_{d,v,u,w}$ be the coarse moduli space of $\mathcal{M}^\nu_{d,v,u,w}$.
\end{de}
\begin{dele}[{\cite[Proposition 2.62, Definition 2.64]{HH2}}]
There exists a $\mathbb{Q}$-line bundle $\lambda_{\mathrm{CM},\infty}$ such that $\lim_{t\to\infty}\nu^*\lambda_{\mathrm{CM},t}=\lambda_{\mathrm{CM},\infty}$.
We write $\Lambda_{\mathrm{CM},\infty}$ on $M^\nu_{d,v,u,w}$ as the $\mathbb{Q}$-line bundle whose pullback to $\mathcal{M}^\nu_{d,v,u,w}$ coincides with $\lambda_{\mathrm{CM},\infty}$, which is uniquely determined up to $\mathbb{Q}$-linear equivalence.
\end{dele}

We note that $\nu^*\lambda_{\mathrm{CM},t}$ converges in $\mathrm{Pic}_{\mathbb{Q}}(\mathcal{M}^\nu_{d,v,u,w})$ in the same way as the proof of \cite[Proposition 2.62]{HH2}.

\subsection{KSBA moduli}
Fix $d\in\mathbb{Z}_{>0}$, $v\in\mathbb{Q}_{>0}$ and  $c\in\mathbb{Q}_{>0}$.
We set the following moduli stack $\mathcal{M}^{\mathrm{KSBA}}_{d,v,c}$ as follows: for any scheme $S$, the collection of objects $\mathcal{M}^{\mathrm{KSBA}}_{d,v,c}(S)$ is
$$\left\{
\begin{array}{l}
f\colon (\mathcal{X},\Delta) \to S
\end{array}
\;\middle|
\begin{array}{rl}
(i)&\text{$\Delta=c\mathcal{D}$, where $\mathcal{D}$ is a K-flat relative Mumford divisor,}\\
(ii)&\text{$f$ is stable,} \\
(iii)&\text{$\mathrm{dim}\,\mathcal{X}_{\bar{s}}=d$, $\mathrm{vol}(K_{\mathcal{X}_{\bar{s}}}+\Delta_{\bar{s}})=v$ for any geometric}\\
&\text{point $\bar{s}\in S$.}
\end{array}\right\}.$$
For the definition of K-flatness, see \cite[Definition 7.1]{kollar-moduli}.
We call this stack as the {\it KSBA moduli stack} after the pioneering works by Koll\'ar--Shepherd-Barron \cite{KSB} and Alexeev \cite{Al}.
Let $f\colon (\mathcal{X},\Delta_{\mathcal{X}})\to S$ and $g\colon (\mathcal{Y},\Delta_{\mathcal{Y}})\to S$ be two objects of $\mathcal{M}^{\mathrm{KSBA}}_{d,v,c}(S)$.
We set an isomorphism $h\colon \mathcal{X}\to \mathcal{Y}$ in the groupoid $\mathcal{M}^{\mathrm{KSBA}}_{d,v,c}(S)$ as an isomorphism over $S$ such that $\Delta_{\mathcal{X}}=h^*\Delta_{\mathcal{Y}}$. 
Furthermore, as the discussion in Subsection \ref{subsec--k-moduli-cy-fib}, we can define a $\mathbb{Q}$-line bundle $\lambda^{\mathrm{KSBA}}_{\mathrm{CM}}$ on $\mathcal{M}^{\mathrm{KSBA}}_{d,v,c}$ such that $\alpha^*\lambda^{\mathrm{KSBA}}_{\mathrm{CM}}=\lambda_{n+1}$ for any morphism $\alpha\colon T\to \mathcal{M}^{\mathrm{KSBA}}_{d,v,c}$ from a scheme, where
\[
\mathrm{det}(f_*\mathcal{O}_X(m(K_{X/S}+\Delta)))=\bigotimes_{i=0}^{n+1}\lambda_i^{\otimes\binom{m}{i}}
\]
is the Knudsen--Mumford expansion for any sufficiently divisible and large $m\in\mathbb{Z}$ and $f\colon (X,\Delta)\to T$ is the family corresponding to $\alpha$. 
We call $\lambda^{\mathrm{KSBA}}_{\mathrm{CM}}$ the CM line bundle on the KSBA moduli stack $\mathcal{M}^{\mathrm{KSBA}}_{d,v,c}$.
Then, we state the following.
\begin{thm}[\cite{kollar-moduli,PX}]\label{thm--ksba}
$\mathcal{M}^{\mathrm{KSBA}}_{d,v,c}$ is a proper Deligne--Mumford stack of finite type over $\mathbbm{k}$ with a projective coarse moduli space $\pi\colon \mathcal{M}^{\mathrm{KSBA}}_{d,v,c}\to M^{\mathrm{KSBA}}_{d,v,c}$.
Furthermore, there exists an ample $\mathbb{Q}$-line bundle $\Lambda^{\mathrm{KSBA}}_{\mathrm{CM}}$ on $M^{\mathrm{KSBA}}_{d,v,c}$ such that $\pi^*\Lambda^{\mathrm{KSBA}}_{\mathrm{CM}}\sim_{\mathbb{Q}}\lambda^{\mathrm{KSBA}}_{\mathrm{CM}}$.
\end{thm}

Let $S$ be a scheme and $f\colon (\mathcal{X},\Delta_{\mathcal{X}})\to S$ be an object of $\mathcal{M}^{\mathrm{KSBA}}_{d,v,c}(S)$.
Let $\alpha\colon S\to \mathcal{M}^{\mathrm{KSBA}}_{d,v,c}$ be the induced morphism.
We say that $\alpha^*\lambda^{\mathrm{KSBA}}_{\mathrm{CM}}$ is the {\it CM line bundle for the stable family} $f$.
We write this as $\lambda_{\mathrm{CM},(\mathcal{X},\Delta_{\mathcal{X}})}$.
If $\Delta_{\mathcal{X}}=0$, $\lambda_{\mathrm{CM},(\mathcal{X},\Delta_{\mathcal{X}})}$ is a positive multiple of the CM line bundle defined in Subsection \ref{subsec--CM}.

\section{A generalization of the Picard scheme to $\mathbb{Q}$-Cartier divisors}\label{sec--3}

In this section, we generalize the Picard scheme to flat $\mathbb{Q}$-Cartier Weil divisorial sheaves.

If $S=\mathrm{Spec}\, K$, where $K$ is a field, we can define $\mathbf{Pic}^0_{X/K}$ as the identity component of $\mathbf{Pic}_{X/K}$.
We note that the following lemma holds.

\begin{lem}\label{lem--pic-const}
    In the situation of Definition \ref{de--pic}, we further assume that any geometric fiber is of lc type.
    Then, there exists an open and closed group subscheme $\mathbf{Pic}^0_{X/S}\subset \mathbf{Pic}^\tau_{X/S}$ such that $\mathbf{Pic}^0_{X/S}$ is projective and smooth over $S$ and any fiber of $\mathbf{Pic}^0_{X/S}$ over $s\in S$ coincides with $\mathbf{Pic}^0_{X_{\kappa(s)}/\kappa(s)}$.
\end{lem}

\begin{proof}
    It is easy to see that there exists an open group subscheme $\mathbf{Pic}^0_{X/S}\subset \mathbf{Pic}^\tau_{X/S}$ such that any fiber of $\mathbf{Pic}^0_{X/S}$ over $s\in S$ coincides with $\mathbf{Pic}^0_{X_{\kappa(s)}/\kappa(s)}$ by \cite[Corollary 9.5.14, Proposition 9.5.20]{FGA} and \cite[Corollary 2.64]{kollar-moduli}.
    The projectivity follows from \cite[Theorem 9.5.4 and Proposition 9.5.20]{FGA}.
    The smoothness follows from \cite[Corollary 2.64]{kollar-moduli} and the same argument of \cite[Remark 9.5.21]{FGA}.
\end{proof}

We note that there exists the torsion component $\mathbf{Pic}^\tau_{X/S}\subset \mathbf{Pic}_{X/S}$ as a quasi-projective scheme by \cite[9.6.16]{FGA} in this situation.

\begin{de}\label{de--Q-Cartier--divisorial}
Let $f\colon (X,\Delta)\to S$ be a projective locally stable morphism over a locally Noetherian scheme with only geometrically connected fibers.  
We say that a coherent sheaf $\mathscr{F}$ flat over $S$ is {\it divisorial} if for any geometric point $\bar{s}\in S$,
$\mathscr{F}_{\bar{s}}$ is a $S_2$ sheaf of rank one that is locally free around codimension one points of $X_{\bar{s}}$.
In this case, we note that $\mathscr{F}=\mathscr{F}^{[**]}$.
Set the following functor $\mathfrak{W}\mathfrak{Pic}_{X/S}$ as  
\[
\mathfrak{W}^\mathbb{Q}\mathfrak{Pic}_{X/S}(T):=\{\text{divisorial sheaves on $X_T$ over $T$}\}/\sim_T
\]
for any $S$-scheme $T$, where $\mathscr{F}\sim_T\mathscr{G}$ means that $\mathscr{F}[\otimes]\mathscr{G}^{-1}$ is a line bundle such that $\mathscr{F}[\otimes]\mathscr{G}^{-1}\sim_T\mathcal{O}_{X_T}$.
Let $(\mathfrak{W}\mathfrak{Pic}_{X/S})^{\text{\'et}}$ be the \'etale sheafification of $\mathfrak{W}\mathfrak{Pic}_{X/S}$.
Fix a $f$-very ample line bundle $A$ on $X$ and a polynomial $P$.
Set 
$\mathfrak{W}\mathfrak{Pic}^{P}_{X/S}$ as a subfunctor of $\mathbb{Q}$-Cartier divisorial sheaves with Hilbert polynomial $P$ with respect to $A$ and $(\mathfrak{W}\mathfrak{Pic}_{X/S}^{P})^{\text{\'et}}$ as the \'etale sheafification.
It is easy to check that $(\mathfrak{W}\mathfrak{Pic}_{X/S}^{P})^{\text{\'et}}$ is open and closed in $(\mathfrak{W}\mathfrak{Pic}_{X/S})^{\text{\'et}}$.

Furthermore, if $\mathscr{F}^{[m]}[\otimes]\mathscr{G}$ exists for any $m\in\mathbb{Z}$, and $\mathscr{F}^{[m]}$ is locally free for some $m\ne0$, then we say that $\mathscr{F}$ is {\it $\mathbb{Q}$-Cartier}.
If a $\mathbb{Q}$-Cartier sheaf $\mathscr{F}$ satisfies the following condition, then we say that $\mathscr{F}$ is {\it multiplicative}. 
\begin{itemize}
    \item   $\mathscr{F}_T^{[-1]}[\otimes]\mathscr{G}$ and $\mathscr{F}_T[\otimes]\mathscr{G}$ exist for any \'etale morphism $T\to S$ from a scheme and flat $\mathbb{Q}$-Cartier divisorial sheaf $\mathscr{G}$ on $X_T$.
\end{itemize}

 Set the following functor $\mathfrak{W}^\mathbb{Q}\mathfrak{Pic}_{X/S}$ as  
\[
\mathfrak{W}^\mathbb{Q}\mathfrak{Pic}_{X/S}(T):=\{\text{$\mathbb{Q}$-Cartier divisorial sheaves on $X_T$ over $T$}\}/\sim_T
\]
for any $S$-scheme $T$.
Let $(\mathfrak{W}^{\mathbb{Q}}\mathfrak{Pic}_{X/S})^{\text{\'et}}$ be the \'etale sheafification of $\mathfrak{W}^{\mathbb{Q}}\mathfrak{Pic}_{X/S}$.
Set 
$\mathfrak{W}^\mathbb{Q}\mathfrak{Pic}^{P}_{X/S}$ as a subfunctor of $\mathbb{Q}$-Cartier divisorial sheaves with Hilbert polynomial $P$ with respect to $A$ and $(\mathfrak{W}^{\mathbb{Q}}\mathfrak{Pic}_{X/S}^{P})^{\text{\'et}}$ as the \'etale sheafification.
It is easy to check that $(\mathfrak{W}^{\mathbb{Q}}\mathfrak{Pic}_{X/S}^{P})^{\text{\'et}}$ is open and closed in $(\mathfrak{W}^{\mathbb{Q}}\mathfrak{Pic}_{X/S})^{\text{\'et}}$.
Note also that $(\mathfrak{W}^{\mathbb{Q}}\mathfrak{Pic}_{X/S})^{\text{\'et}}$ is an \'etale subsheaf of $(\mathfrak{W}\mathfrak{Pic}_{X/S})^{\text{\'et}}$.
For any $\lambda\in(\mathfrak{W}^{\mathbb{Q}}\mathfrak{Pic}_{X/S})^{\text{\'et}}(T)$ and $m\in\mathbb{Z}$, where $T$ is a $S$-scheme, we can define $\lambda^{[m]}$, which corresponds to the operation $\mathscr{F}\mapsto \mathscr{F}^{[m]}$ \'etale locally.
\end{de}

First, we put the following lemma, which we use at the end of Section \ref{sec--3}. 

\begin{lem}\label{lem--divisor--restr}
    Let $f\colon (X,\Delta)\to S$ be a projective locally stable morphism over a locally Noetherian scheme with only geometrically connected and normal fibers.
    Suppose that $S$ is reduced.
    If $\mathscr{F}$ is a $\mathbb{Q}$-Cartier divisorial sheaf flat over $S$, then $\mathscr{F}$ is multiplicative.
\end{lem}

\begin{proof}
Take a flat $\mathbb{Q}$-Cartier divisorial sheaf $\mathscr{G}$.
It suffices to show that $\mathscr{F}^{[-1]}[\otimes]\mathscr{G} $  exists.
By \cite[Theorem 9.40]{kollar-moduli}, there exists a locally closed decomposition $S'\to S$ such that for any morphism $T\to S$, $\mathscr{F}_T^{[-1]}[\otimes]\mathscr{G}_T $ exists if and only if $T$ factors through $S'$.
Since $S$ is reduced, it suffices to show that $S'\to S$ is a closed immersion.
For this, we may assume that $S$ is a smooth affine curve.
   In this case, we may assume that there exist $\mathbb{Q}$-Cartier relative Mumford divisors $D_1$ and $D_2$ such that $\mathscr{F}=\mathcal{O}_{X}(D_1)$ and $\mathscr{G}=\mathcal{O}_{X}(D_2)$.
   Then $\mathcal{O}_X(D_2-D_1)$ is flat and divisorial over $S$ by \cite[Proposition 2.79]{kollar-moduli}.
   This means that $\mathscr{F}^{[-1]}[\otimes]\mathscr{G} $  exists as $\mathcal{O}_X(D_2-D_1)$.
   We complete the proof.
\end{proof}

We will extend \cite[Theorem 9.2.5]{FGA} to our $\mathfrak{W}^{\mathbb{Q}}\mathfrak{Pic}_{X/S}$.

\begin{prop}\label{prop--sheafif-wpic}
Let $f\colon (X,\Delta)\to S$ be a projective locally stable morphism over a Noetherian scheme with only geometrically connected fibers.  
Then the natural maps $\mathfrak{W}\mathfrak{Pic}_{X/S}\to(\mathfrak{W}\mathfrak{Pic}_{X/S})^{\text{\'et}}$ and $\mathfrak{W}^{\mathbb{Q}}\mathfrak{Pic}_{X/S}\to(\mathfrak{W}^{\mathbb{Q}}\mathfrak{Pic}_{X/S})^{\text{\'et}}$ are injective.
Furthermore, let $X_{\mathrm{sm}}\subset X$ be the largest open subset such that $f|_{X_{\mathrm{sm}}}$ is smooth.
If $f$ admits a section $g$ such that $g(S)\subset X_{\mathrm{sm}}$, then these maps are also surjective. 
\end{prop}

\begin{proof}
We only deal with $\mathfrak{W}\mathfrak{Pic}_{X/S}$. 
Take two elements $\mathscr{F},\mathscr{G}\in\mathfrak{W}\mathfrak{Pic}_{X/S}(T)$ for some scheme.   
If $\mathscr{F}$ and $\mathscr{G}$ are mapped to the same element in $(\mathfrak{W}\mathfrak{Pic}_{X/S})^{\text{\'et}}(T)$, then it is easy to see that $\mathscr{F}^{[-1]}[\otimes]\mathscr{G}$ exists as a line bundle.
Then the first part of \cite[Theorem 9.2.5]{FGA} shows that $\mathscr{F}^{[-1]}[\otimes]\mathscr{G}\sim_T\mathcal{O}_{X_T}$.
Thus, the natural map is injective.

Suppose that $f$ admits a section $g$ such that $g(S)\subset X_{\mathrm{sm}}$.
Then, we note that any divisorial sheaf $\mathscr{F}$ is locally trivial around $g(S)$.
Therefore, $g^*\mathscr{F}$ is also locally trivial.
Take any element $\lambda\in (\mathfrak{W}\mathfrak{Pic}_{X/S})^{\text{\'et}}(T)$ for any scheme $T$.
By taking an \'etale surjection $h\colon T'\to T$ such that $h_{X_T}^*\lambda$ is represented by a divisorial sheaf $\mathscr{F}$.
Here, we may assume that there exists an isomorphism $u_{T'}\colon g_{T'}^*\mathscr{F}\cong \mathcal{O}_{T'}$.
Furthermore, on $X_{T'\times_TT'}$, there exists an isomorphism $v'$ from the pullback of $\mathscr{F}$ via the first projection on to the pullback via the second preserving the pullbacks of $u_{T'}$.
Consider the three projections $X_{T'\times_TT'\times_TT'}\to X_{T'\times_TT'}$.
Let $v'_{ij}$ denote the pullback of $v'$ via the projection to the $i$th and $j$th factors.
Then $v'^{-1}_{13}v'_{23}v'_{12}$ is an automorphism of the pullback of $\mathscr{F}$ via the first projection preserving the pullback of $u_{T'}$.
As \cite[Lemma 9.2.10]{FGA}, we see that $v'^{-1}_{13}v'_{23}v'_{12}$ is the identity.
Hence, $\mathscr{F}$ descends to a coherent sheaf $\mathscr{G}$ on $X_T$.
It is easy to check that $\mathscr{G}$ is also a divisorial sheaf.
Furthermore, if $\mathscr{F}^{[l]}$ is a line bundle for some $l\in\mathbb{Z}$, then so is $\mathscr{G}^{[l]}$.
Therefore, two maps are surjective in this case.
We complete the proof.
\end{proof}

The following is the main result of this section.

\begin{thm}\label{thm--wpic--existence}
Let $f\colon (X,\Delta)\to S$ be a projective locally stable morphism over a Noetherian scheme with only geometrically connected klt fibers. 
Fix a $f$-very ample line bundle $A$ on $X$ and a polynomial $P$.

Then, $(\mathfrak{W}^{\mathbb{Q}}\mathfrak{Pic}_{X/S})^{\text{\'et}}$ is represented by a separated scheme $\mathbf{W}^{\mathbb{Q}}\mathbf{Pic}_{X/S}$ over $S$ and there exists a canonical open immersion $\mathbf{Pic}_{X/S}\hookrightarrow \mathbf{W}^{\mathbb{Q}}\mathbf{Pic}_{X/S}$.
Moreover, $(\mathfrak{W}^{\mathbb{Q}}\mathfrak{Pic}^{P}_{X/S})^{\text{\'et}}$ is represented by a quasi-projective scheme $\mathbf{W}^{\mathbb{Q}}\mathbf{Pic}_{X/S}^P$ over $S$.
\end{thm}

Before showing this, we put the following proposition.

\begin{prop}\label{prop--xz}
  Let $f\colon (X,\Delta)\to S$ be a projective flat log $\mathbb{Q}$-Gorenstein family such that $(X_{\bar{s}},\Delta_{\bar{s}})$ is klt for any geometric point $\bar{s}\in S$, where $S$ is a Noetherian scheme.

Then, there exists a positive integer $I$ such that for any $\mathbb{Q}$-Cartier Weil divisor $D$ on $X_{\bar{s}}$ for any geometric point $\bar{s}\in S$, $ID$ is Cartier.
\end{prop}

\begin{proof}
We may assume that $S$ is connected.
We may also replace $S$ with its normalization and assume that $S$ is normal and irreducible.
    Take a $f$-ample line bundle $L$ on $X$.
    Suppose that $\mathrm{dim}\,X_{\bar{s}}=n$.
    By \cite[Theorem 1.4]{XZuniqueness} and \cite[Theorem D]{BlJ}, we can see as the proof of \cite[Theorem 1.5]{XZuniqueness} that if a positive integer $m$ satisfies that $m\ge \frac{(n+1)^n}{(\delta_{(X_{\bar{s}},\Delta_{\bar{s}})}(L_{\bar{s}}))^nL_{\bar{s}}^n}$, then $m!D$ is Cartier for any $\mathbb{Q}$-Cartier Weil divisor $D$ on $X_{\bar{s}}$.
    By \cite[Theorem 6.6]{BL}, the map
    \[
    S\ni s\mapsto \delta_{(X_{\bar{s}},\Delta_{\bar{s}})}(L_{\bar{s}})\in\mathbb{R}_{>0}
    \]
    is lower semicontinuous.
    Since $S$ is quasi-compact, there exists a positive integer $m$ such that $m\ge \frac{(n+1)^n}{(\delta_{(X_{\bar{s}},\Delta_{\bar{s}})}(L_{\bar{s}}))^nL_{\bar{s}}^n}$ for any geometric point $\bar{s}\in S$.
    Therefore, we complete the proof by putting $I=m!$.
\end{proof}

We recall the Castelnuovo--Mumford regularity (cf.~\cite[Section 5.2]{FGA}), which is necessary for the proof of Theorem \ref{thm--wpic--existence}.

\begin{de}
    Let $X$ be a projective variety over a field with a very ample line bundle $A$.
    For any $m\in\mathbb{Z}$, we say that a coherent sheaf $\mathscr{F}$ is $m$-{\it regular} if $H^i(X,\mathscr{F}\otimes A^{\otimes (m-i)})=0$ for each $i>0$.
\end{de}

\begin{proof}[Proof of Theorem \ref{thm--wpic--existence}]
The assertion is local on $S$ and hence we may assume that $A$ induces a closed immersion $X\subset \mathbb{P}^N_S$ by shrinking $S$.
First, we deal with $(\mathfrak{W}\mathfrak{Pic}^{P}_{X/S})^{\text{\'et}}$.    
Note that if $S$ is a spectrum of a field, then a divisorial sheaf $\mathscr{F}$ is geometrically stable with respect to $A$ in the sense of \cite[Definition 1.2.9]{HL} by the fact that $\mathscr{F}$ is of rank one torsion free sheaf and \cite[Lemma 1.5.10]{HL}.
The set of sheaves on $X_{\bar{s}}$ of rank one with Hilbert polynomial $P$ is bounded for any geometric point $\bar{s}\in S$ by \cite[Theorem 3.3.7]{HL}.
By this, there exists $m\in\mathbb{Z}_{>0}$ such that $\mathscr{F}$ is $m$-regular for any sheaf $\mathscr{F}$ on $X_{\bar{s}}$ of rank one with Hilbert polynomial $P$.
Consider $Q:=\mathbf{Quot}^{\mathcal{O}_{X}(-m)^{\oplus P(m)},P}_{X/S}$.
Then, there exists an open subscheme $Z_1\subset Q$ that parameterizes all quotients that are divisorial by \cite[Corollary 10.12]{kollar-moduli}.
It is easy to see that the quotient stack $[Z_1/PGL_S(P(m))]$ is an algebraic space by \cite[Corollary 8.3.5]{Ols}.
By \cite[Theorem 4.3.7]{HL}, $[Z_1/PGL_S(P(m))]$ is quasi-projective.
Then, it is easy to see that $(\mathfrak{W}\mathfrak{Pic}^{P}_{X/S})^{\text{\'et}}$ is represented by $[Z_1/PGL_S(P(m))]$ by using Proposition \ref{prop--sheafif-wpic} since $f\colon X\to S$ \'etale locally admits a section $g$ such that $g(S)$ lies in the smooth locus of $f$.

Next, we deal with $(\mathfrak{W}^{\mathbb{Q}}\mathfrak{Pic}^{P}_{X/S})^{\text{\'et}}$.
Let $\mathscr{U}$ be the universal quotient sheaf on $Q$.
Take $I$ as Proposition \ref{prop--xz}.
Then, take a partial locally closed decomposition $Z_2\to Z_1$ such that for any morphism $g\colon T\to Z_1$ from a scheme, $g$ factors through $Z_2$ if and only if $(g_{X_{Z_1}}^*(\mathscr{U}|_{X_{Z_1}}))^{[j]}$ exists for any $1\le j\le I$ and $(g_{X_{Z_1}}^*(\mathscr{U}|_{X_{Z_1}}))^{[I]}$ is locally free by \cite[Theorem 9.40]{kollar-moduli}.
Then, we see that $(\mathfrak{W}^{\mathbb{Q}}\mathfrak{Pic}^{P}_{X/S})^{\text{\'et}}$ is represented by a quasi-projective scheme $[Z_2/PGL_S(P(m))]$. Now, the assertions for $\mathbf{W}^\mathbb{Q}\mathbf{Pic}_{X/S}$ are trivial. We complete the proof.
\end{proof}

Then, we can define the following notion.

\begin{de}\label{de--mu--i}
Let $k\in\mathbb{Z}$ and $f\colon (X,\Delta)\to S$ be a projective locally stable morphism over a Noetherian scheme with only geometrically connected klt fibers.
For any $\mathbb{Q}$-Cartier divsiroial sheaf $\mathscr{F}$, we can take $\mathscr{F}^{[k]}$ by the definition.
By this operation, we can define the multiplication morphism by $k$:
\[
\mu_k\colon\mathbf{W}^\mathbb{Q}\mathbf{Pic}_{X/S}\to\mathbf{W}^\mathbb{Q}\mathbf{Pic}_{X/S}.
\]   
Take $I$ as Proposition \ref{prop--xz}.
Then, $\mu_I$ factors through $\mathbf{Pic}_{X/S}$.
We define an open subscheme $\mathbf{W}^\mathbb{Q}\mathbf{Pic}_{X/S}^\tau:=\mu_I^{-1}\mathbf{Pic}_{X/S}^\tau$.
Note that $\mathbf{W}^\mathbb{Q}\mathbf{Pic}_{X/S}^\tau$ parameterizes all numerically trivial $\mathbb{Q}$-Cartier divisorial sheaves.
\end{de}

First, we will put the following lemma, which shows that $\mu_k$ is of finite type.

\begin{lem}\label{lem--boundedness-tau}
Let $f\colon (X,\Delta)\to S$ be a projective locally stable family over a Noetherian scheme such that any geometric fiber is klt and connected.
Fix $k\in\mathbb{Z}_{>0}$, a polynomial $P$, and a $f$-very ample line bundle $L$ on $X$.
Then there exist $m\in\mathbb{Z}_{>0}$ and finitely many polynomials $P_1,\ldots,P_l$ that satisfy the following:
\begin{itemize}
    \item For any geometric point $\bar{s}\in S$, let $\mathscr{F}$ be a $\mathbb{Q}$-Cartier divisorial sheaf on $X_{\bar{s}}$ such that $\mathscr{F}^{[k]}$ has a Hilbert polynomial $P$.
    Then, $\mathscr{F}$ is $m$-regular and the Hilbert polynomial coincides with one of $P_1,\ldots,P_l$ with respect to $L_{\bar{s}}$.
\end{itemize}
\end{lem}

\begin{proof}
By Noetherian induction on $S$, we may freely shrink $S$.
Shrinking $S$, we may assume that $d:=\mathrm{dim}\,X_{s}$ for any $s\in S$.
We note that $\mathscr{F}^{[k]}$ is geometrically stable and hence bounded by \cite[Theorem 3.3.7]{HL}.
Therefore, there exists $m\in\mathbb{Z}_{>0}$ independent from the choice of $\bar{s}$ such that  $(m-d)L_{\bar{s}}+\mathscr{F}-K_{X_{\bar{s}}}-\Delta_{\bar{s}}$ is ample by \cite[Proposition 1.41]{KM}.
By the Kawamata--Viehweg vanishing theorem, $\mathscr{F}$ is $m$-regular for $\mathscr{F}$ as above.

Therefore, it suffices to bound $H^0(X_{\bar{s}},\mathscr{F}\otimes\mathcal{O}_{X_{\bar{s}}}(m_iL_{\bar{s}}))$ for some sufficiently large $m_i$ for $i=0,1,\ldots,d$ by the $m$-regularity since $\chi(X_{\bar{s}},\mathscr{F}\otimes\mathcal{O}_{X_{\bar{s}}}(mL_{\bar{s}}))$ is a polynomial of degree at most $d$. 
We will show the assertion by induction on $d$.
When $d=1$, the assertion easily follows from the existence of $\mathbf{Pic}^\tau_{X/S}$ as a quasi-projective scheme \cite[Theorem 9.6.16]{FGA}.
Suppose that the assertion holds for families of dimension $d-1$.
Shrinking $S$ if necessary, take a general element $H\in |m'L|$ for any sufficiently large $m'\in\mathbb{Z}_{>0}$ such that $H$ is a relative Mumford divisor on $X$.
Then,
\[
0\to H^0(X_{\bar{s}},\mathscr{F})\to H^0(X_{\bar{s}},\mathscr{F}\otimes\mathcal{O}_{X_{\bar{s}}}(mL_{\bar{s}}))\to H^0(H_{\bar{s}},\mathcal{O}_{H_{\bar{s}}}(mL|_{H_{\bar{s}}})\otimes \mathscr{F}|_{H_{\bar{s}}}),
\]
and we see that $h^0(X_{\bar{s}},\mathscr{F}\otimes\mathcal{O}_{X_{\bar{s}}}(mL_{\bar{s}}))\le h^0(H_{\bar{s}},\mathcal{O}_{H_{\bar{s}}}(mL|_{H_{\bar{s}}})\otimes \mathscr{F}|_{H_{\bar{s}}})+1$.
By shrinking $S$, we may assume that $(H,\Delta|_H)$ is again locally stable whose geometric fibers are klt.
Thus, we may assume that $\mathscr{F}|_{H_{\bar{s}}}$ is $m'$-regular by what we have shown in the first paragraph.
By the induction hypothesis, there are only finitely many choices of the Hilbert polynomials of $\mathscr{F}|_{H_{\bar{s}}}$.
Hence, there are also only finitely many possibilities of the Hilbert polynomials of $\mathscr{F}$.
We complete the proof.
\end{proof}

The following is the key in Section \ref{sec-5}.

\begin{prop}\label{prop--mu_k}
 The situation is as in Definition \ref{de--mu--i} and fix $k\in\mathbb{Z}_{>0}$.
 Then $\mu_k$ is a finite morphism.
\end{prop}

\begin{proof}
First, we see that $\mu_k$ is of finite type by Lemma \ref{lem--boundedness-tau}, \cite[Theorem 3.3.7]{HL}, and Theorem \ref{thm--wpic--existence}.
To see that $\mu_k$ is quasi-finite, we may assume that $S=\mathrm{Spec}\,\mathbbm{k}$.
Then, it suffices to show that $\mu_{k}^{-1}\mu_k(F)$ is zero-dimensional for any $F\in \mathbf{W}^{\mathbb{Q}}\mathbf{Pic}_X(\mathbbm{k})$.
We see that for any $\mathbbm{k}$-valued point $G\in \mu_{k}^{-1}\mu_k(F)$, $G[\otimes]F^{[-1]}$ is numerically trivial.
Assume that $\mu_{k}^{-1}\mu_k(F)$ is not zero-dimensional around $G$.
Then there exist an affine curve $C$ and a family of $\mathbb{Q}$-Cartier divisorial sheaves $\mathscr{F}$ on $X\times C$ such that $\mathscr{F}_c[\otimes]G\in \mu_{k}^{-1}\mu_k(F)$ for any closed point $c\in C$ and for some closed point $0\in C$ such that $\mathscr{F}_0\cong \mathcal{O}_X$.
By shrinking $C$, we may assume that $\mathscr{F}$ is locally free.
Thus, $\mathscr{F}_c\in (\mathbf{Pic}^\tau_{X/\mathbbm{k}}\cap \mu_k^{-1}\mu_k(\mathcal{O}_X))(\mathbbm{k})$.
Because $\mu_k|_{\mathbf{Pic}^\tau_{X/\mathbbm{k}}}$ is finite, $(\mathbf{Pic}^\tau_{X/\mathbbm{k}}\cap \mu_k^{-1}\mu_k(\mathcal{O}_X))$ is also finite.
This leads to a contradiction.
Therefore, $\mu_k$ is quasi-finite.

Next, we will show that $\mu_k$ is proper.
We know that $\mu_k$ is separated.
Therefore, we can apply \cite[Theorem 11.5.1]{Ols} to $\mu_k$.
Let $R$ be an arbitrary discrete valuation ring with fractional field $K$ over $S$.
Consider $X_R:=X\times_S\mathrm{Spec}\,R$ and suppose that there exists an element $F\in \mathbf{W}^{\mathbb{Q}}\mathbf{Pic}_{X/S}(R)$ such that $F_K=\mu_k(G)$ for some $G\in \mathbf{W}^{\mathbb{Q}}\mathbf{Pic}_{X/S}(K)$.
By replacing $R$ with a finite extension $R'$ that is regular, we may assume that $F$ and $G$ are $\mathbb{Q}$-Cartier divisorial sheaves.
Note that $R$ is no longer a discrete valuation ring in general but a Dedekind domain.
In this case, $F_K=G^{[k]}$.
By replacing $F$ with $F\otimes \mathcal{O}_{X_R}(kA)$ and $G$ with $G\otimes\mathcal{O}_{X_K}(A_K)$ for some ample line bundle $A$, we may assume that there exists a $\mathbb{Q}$-Cartier effective divisor $D$ such that $G\cong\mathcal{O}_{X_K}(D)$.
Let $\overline{D}$ be the closure of $D$ in $X_{R}$.
$\overline{D}$ is an effective relative Mumford divisor and $\mathcal{O}_{X_R}(k\overline{D})\cong F$ by \cite[II, Proposition 6.5]{Ha}.
This means that $\overline{D}$ is also $\mathbb{Q}$-Cartier.
Then, we set $\overline{G}:=\mathcal{O}_{X_R}(\overline{D})$.
By \cite[Proposition 2.79]{kollar-moduli}, we see that $\overline{G}$ coincides with its universal hull.
Thus, $\mu_k$ is proper.
Since $\mu_k$ is quasi-finite and proper, $\mu_k$ is finite by \cite[Tag 02OG]{Stacks}.
\end{proof}

We note that $\mu_k|_{\mathbf{Pic}_{X/S}}$ is not proper in general (e.g.~consider a $\mathbb{Q}$-Gorenstein smoothing of $\mathbb{P}(a^2:b^2:c^2)$ to $\mathbb{P}^2$, where $a,b$ and $c\in\mathbb{Z}_{>0}$ satisfy $a^2+b^2+c^2=3abc$ cf.~\cite{HP}).

The following are useful observations, which follow from Proposition \ref{prop--mu_k}.

\begin{cor}\label{cor--torsion--components--number}
Let $f\colon (X,\Delta)\to S$ be a projective locally stable morphism over a Noetherian scheme with only geometrically connected klt fibers.

Then $(\mathbf{W}^\mathbb{Q}\mathbf{Pic}_{X_{\bar{s}}/\overline{\kappa(s)}}^\tau)_{\mathrm{red}}/\mathbf{Pic}^0_{X_{\bar{s}}/\overline{\kappa(s)}}$ is a finite group and there exists a positive integer $m\in\mathbb{Z}_{>0}$ such that $$\#((\mathbf{W}^\mathbb{Q}\mathbf{Pic}_{X_{\bar{s}}/\overline{\kappa(s)}}^\tau)_{\mathrm{red}}/\mathbf{Pic}^0_{X_{\bar{s}}/\overline{\kappa(s)}})\le m$$ for any $s\in S$.
\end{cor}

\begin{proof}
    This follows from the same argument of \cite[Corollary 9.6.17]{FGA}.
    We note that it is easy to see that $(\mathbf{W}^\mathbb{Q}\mathbf{Pic}_{X_{\bar{s}}/\overline{\kappa(s)}}^\tau)_{\mathrm{red}}$ is a group scheme by Lemma \ref{lem--divisor--restr} because if $\mathscr{F}$ is the universal sheaf on $(\mathbf{W}^\mathbb{Q}\mathbf{Pic}_{X_{\bar{s}}/\overline{\kappa(s)}}^\tau)_{\mathrm{red}}\times X_{\bar{s}}$ (cf.~Proposition \ref{prop--sheafif-wpic}), then we can construct a morphism $\mu\colon (\mathbf{W}^\mathbb{Q}\mathbf{Pic}_{X_{\bar{s}}/\overline{\kappa(s)}}^\tau)_{\mathrm{red}}\times(\mathbf{W}^\mathbb{Q}\mathbf{Pic}_{X_{\bar{s}}/\overline{\kappa(s)}}^\tau)_{\mathrm{red}}\to(\mathbf{W}^\mathbb{Q}\mathbf{Pic}_{X_{\bar{s}}/\overline{\kappa(s)}}^\tau)_{\mathrm{red}}$ that is induced by $p_1^*\mathscr{F}[\otimes]p_2^*\mathscr{F}$, where $$p_1,p_2\colon (\mathbf{W}^\mathbb{Q}\mathbf{Pic}_{X_{\bar{s}}/\overline{\kappa(s)}}^\tau)_{\mathrm{red}}\times(\mathbf{W}^\mathbb{Q}\mathbf{Pic}_{X_{\bar{s}}/\overline{\kappa(s)}}^\tau)_{\mathrm{red}}\times X_{\bar{s}}\to (\mathbf{W}^\mathbb{Q}\mathbf{Pic}_{X_{\bar{s}}/\overline{\kappa(s)}}^\tau)_{\mathrm{red}}\times X_{\bar{s}}$$
    are the morphisms induced by the first and second projections, respectively.
    We can check that $\mu$ corresponds to the multiplication morphism of $(\mathbf{W}^\mathbb{Q}\mathbf{Pic}_{X_{\bar{s}}/\overline{\kappa(s)}}^\tau)_{\mathrm{red}}$.
\end{proof}

\begin{cor}\label{cor--mu_i-proper}
Let $f\colon (X,\Delta)\to S$ be a projective locally stable morphism over a Noetherian scheme with only geometrically connected klt fibers.
    Then $\mathbf{W}^\mathbb{Q}\mathbf{Pic}_{X/S}^\tau$ is projective over $S$.
\end{cor}

\begin{proof}
This follows from Proposition \ref{prop--mu_k} applied to $\mu_{m}$, where $m$ is as in Corollary \ref{cor--torsion--components--number}, because $\mu_{m}$ factors through $\mathbf{Pic}^0_{X/S}$, which is projective over $S$ by Lemma \ref{lem--pic-const}.
\end{proof}

\section{CM line bundle and specially K-stable varieties} 
\label{sec-4}

In this section, we will discuss CM line bundles and their positivity for families of polarized specially K-stable varieties.
In \cite{Hat23}, we dealt with the bigness of the CM line bundle for families $f\colon (X,\Delta;L)\to S$ of specially K-stable polarized log pairs with maximal variation of $(X,\Delta)$.
In the previous paper, we did not deal with the bigness of the CM line bundle for families $f\colon (X,\Delta;L)\to S$, where the linear equivalence classes of $L_{\bar{s}}$ only vary for geometric points $\bar{s}\in S$.
The following example indicates that the CM line bundle depends not only on the isomorphic classes of the varieties but also the polarizations.

\begin{ex}\label{ex--cm}
Let $C$ be a smooth curve of genus $g$ at least three such that $\mathrm{Aut}(C)$ is trivial (cf.~\cite{Ba}).
We know that the diagonal $\Delta\subset C\times C$ satisfies $\Delta^2=2-2g$.
Via the second projection $p_2\colon C\times C\to C$, we can regard $(C\times C,\mathcal{O}_{C\times C}(\Delta))$ as a polarized family of curves with line bundles of degree one.
For any point $c\in C$, $(p_2^{-1}(c),\mathcal{O}_{C\times C}(\Delta)|_{p_2^{-1}(c)})$ admits the same cscK metric since $\mathrm{Aut}(C)$ is trivial (cf.~\cite{CPZ}).
On the other hand, let $\mathscr{M}$ be the analytic moduli space parameterizing cscK metrics of curves constructed in \cite{FS,DN}.
Note that $\mathscr{M}$ does not distinguish the linear equivalence class of polarization.
Then, $(C\times C,\mathcal{O}_{C\times C}(\Delta))$ defines a constant morphism $C\to \mathscr{M}$ but the linear equivalence class of $\mathcal{O}_C(c)$ varies for $c\in C$.
Here, one might expect that the degree of the CM line bundle of $(C\times C,\mathcal{O}_{C\times C}(\Delta))$ would vanish as \cite[Corollary 4.5]{DN}.
However, this is not the case.
Indeed, 
\begin{align*}
    \mathrm{deg}(\lambda_{\mathrm{CM},(C\times C,\mathcal{O}_{C\times C}(\Delta))})&=K_{C\times C/C}\cdot \Delta-(g-1)\Delta^2\\
    &=2g(g-1)>0.
\end{align*}
 At least, this phenomenon indicates that we cannot descend the CM line bundle to $\mathscr{M}$.
\end{ex}

To pick the information of the variation of the linear equivalence classes, we define the following notion.

\begin{de}
    Let $f\colon (X,\Delta;L)\to S$ be a {\it polarized log family} over a scheme $S$ projective over $\mathbbm{k}$, that is, $f\colon (X,\Delta)\to S$ is log $\mathbb{Q}$-Gorenstein with a $f$-ample line bundle $L$.
    We say that $f$ is {\it isotrivial} over $S$ if there exists an isomorphism $\varphi\colon (X_s,\Delta_s)\to (X_{s'},\Delta_{s'})$ such that $\varphi^*L_{s'}\sim L_s$ for any closed points $s,s'\in S$.
    We say that $f$ is of {\it maximal variation} (as a family of polarized log pairs) if for any general closed point $s\in S$ and any curve $C$ passing through $s$, $(X_C,\Delta_C;L_C)\to C$ is not isotrivial. 
\end{de}

We extend Koll\'ar's ampleness criterion (cf.~\cite[Theorem 6.6]{XZ}) for polarized log families of maximal variation.

\begin{stup}[cf.~{\cite[Definition 6.1]{XZ}}]\label{setup--d-cond}
    Let $f\colon (X,\Delta;L)\to S$ be a polarized log family over a scheme $S$ projective and smooth over $\mathbbm{k}$. Suppose that $L$ is a $f$-very ample line bundle, i.e., there exists a closed immersion $\iota\colon X\hookrightarrow \mathbb{P}_S(f_*\mathcal{O}_X(L))$.
    Set $D:=\mathrm{Supp}(\Delta)$.
Suppose that there exists a positive integer $d\in\mathbb{Z}_{>0}$ such that the following hold.
\begin{enumerate}
    \item $H^j(X_s,\mathcal{O}_{X_s}(kL_s))=H^j(D_s,\mathcal{O}_{D_s}(kL_s|_{D_s}))=0$ for any $k,j\in\mathbb{Z}_{>0}$ and closed point $s\in S$,
    \item two natural homomorphisms $\mu_s\colon\mathrm{Sym}^dH^0(X_s,\mathcal{O}_{X_s}(L_s))\to H^0(X_s,\mathcal{O}_{X_s}(dL_s))$ and $\nu_s\colon H^0(X_s,\mathcal{O}_{X_s}(dL_s))\to H^0(D_s,\mathcal{O}_{D_s}(dL_s|_{D_s}))$ are surjective for any point $s\in S$,
    \item for any $s\in S$, elements of $\mathrm{Ker}(\mu_s)$ and $\mathrm{Ker}(\nu_s\circ\mu_s)$ generate the ideal sheaves of  $\iota_s\colon X_s\hookrightarrow \mathbb{P}^{h^0(X_s,\mathcal{O}_{X_s}(L_s))-1}$ and $D_s\subset \mathbb{P}^{h^0(X_s,\mathcal{O}_{X_s}(L_s))-1}$ respectively.
\end{enumerate}
Let $H$ be an ample line bundle on $S$, $W:=f_*\mathcal{O}_X(L)$, and $Q:=f_*\mathcal{O}_X(dL)\oplus f_*\mathcal{O}_{D}(dL|_D)$.
 Note that $W$ is a locally free sheaf but $Q$ is not.
 There exists an open subset $S^\circ\subset S$ such that $\mathrm{codim}_{S}(S\setminus S^\circ)\ge2$ and $D|_{S^\circ}$ is flat over $S^\circ$.
 Then, $Q|_{S^\circ}$ is a locally free sheaf.
 Note that $\mathrm{det}(Q|_{S^\circ})$ can be extended to a unique Cartier divisor $B$ up to linear equivalence since $S$ is smooth.
 Let $w$ and $q$ be the rank of $W$ and $Q$.
 We note that there exists a natural homomorphism $(\mathrm{Sym}^dW)^{\oplus 2}\to Q$ that is surjective on $S^\circ$.

\end{stup}

\begin{thm}\label{thm--Kollar's--lemma}
 In the situation of Setup \ref{setup--d-cond}, suppose that $f\colon (X,\Delta;L)\to S$ has a maximal variation.
 Then, there exists a positive integer $m$ depending only on $d$, $H$ and the family $f$ such that there exists a nonzero homomorphism 
 \[
 \mathrm{Sym}^{dqm}(W^{\oplus 4w})\to \mathcal{O}_S(mB-H).
 \]
\end{thm} 

\begin{proof}
We follow the idea and terminologies of \cite[Theorem 6.6]{XZ}. 
Let $V:=W^{\oplus 2}$, $v:=2w$, and $\mathbb{P}:=\mathbb{P}_S((V^*)^{\oplus v})$.
Let $\pi\colon\mathbb{P}\to S$ be the canonical morphism.
There exists a morphism 
\[
\zeta\colon \mathcal{O}_{\mathbb{P}}^{\oplus v}\to \pi^*V\otimes\mathcal{O}_{\mathbb{P}}(1).
\]
Let $G\subset \mathbb{P}$ be the prime divisor of the determinant of $\zeta$.
Then, we get the morphism
\[
\tilde{u}\colon \bigwedge^q(\mathrm{Sym}^{d}(\mathcal{O}_{\mathbb{P}}^{\oplus v})) \to \pi^*\mathcal{O}_S(qB)\otimes\mathcal{O}_{\mathbb{P}}(dq),
\]
which is surjective on $\mathbb{P}\setminus G$.
This morphism is induced by the composition $$U_{\mathrm{Gr}}\colon \mathrm{Sym}^{d}(\mathcal{O}_{\mathbb{P}}^{\oplus v})\to \pi^*\mathcal{O}_S(B)\otimes\mathcal{O}_{\mathbb{P}}(d)$$ of $\mathrm{Sym}^d\zeta$ and the morphism $\pi^*\mathrm{Sym}^dV\otimes \mathcal{O}_{\mathbb{P}}(d)\to \pi^*\mathcal{O}_S(B)\otimes\mathcal{O}_{\mathbb{P}}(d)$ induced by the morphism $\mathrm{Sym}^dV\to Q$.
Let $\mathrm{Gr}:=\mathrm{Gr}(w',q)$ the Grassmannian variety and then $\tilde{u}$ gives a morphism
\[
u\colon \mathbb{P}\setminus G\to \mathrm{Gr}.
\]
Let $Y$ be the image of $\pi\times u$ and $\overline{Y}$ the Zariski closure of $Y$ in $S\times\mathrm{Gr}$.
Let $\pi_1\colon \overline{Y}\to S$ and $\pi_2\colon \overline{Y}\to \mathrm{Gr}$ be the canonical morphisms.

By the proof of \cite[Theorem 6.6]{XZ}, it suffices to observe that $\pi_2$ is generically finite.
To show this, we note that the proof of \cite[Theorem 6.6]{XZ} also shows that for any general closed point $t\in Y$, $(X_{\pi_1(t)},D_{\pi_1(t)})$ are uniquely determined by $\pi_2(t)$.
Therefore, we may assume that there exists a subvariety $T\subset \pi_2^{-1}\pi_2(t)$ such that $\pi_1(T)$ is not a point and $(X_T,\Delta_T)$ is isotrivial.
By taking an \'etale cover of $T$, we may assume that $(X_T,\Delta_T)\cong (X_t,\Delta_t)\times T$ (see the proof of \cite[Lemma 2.4]{XZ}) and there exists a closed immersion $\iota'\colon X_T\hookrightarrow \mathbb{P}^v_T$ such that $\iota'^*\mathcal{O}_{\mathbb{P}^v_T}(1)\sim_T L_T$.
We show here that $(X_T,\Delta_T;L_T)$ is isotrivial.
Indeed, by assumption, we see that there exists an element $g_{t'}\in PGL(v+1)$ such that the image of $(X_t,\Delta_t)$ under $g_{t'}\circ\iota'_{t'}$ coincides with the image of $(X_{t'},\Delta_{t'})$ under $\iota'_{t'}$ for any $t'\in T$.
Therefore, by replacing $T$ with its suitable \'etale cover, we may assume that there exists $G\in PGL(v+1)(T)$ such that $G\circ\iota'_{t}\times\mathrm{id}_T((X_T,\Delta_T))=\iota'((X_T,\Delta_T))$.
By replacing $T$ with its suitable \'etale cover, we may assume that there exists a $T$-automorphism $\psi\colon (X_T,\Delta_T)\to(X_T,\Delta_T)$ such that $G\circ\iota'_{t}\times\mathrm{id}_T=\iota'\circ\psi$. 
By this equality, it is easy to see that $(X_T,\Delta_T;L_T)$ is isotrivial over $T$.
This contradicts the assumption that $\pi_1(T)$ is not a closed point.
We complete the proof.
\end{proof}

\begin{cor}\label{cor--cm--special--bigness}
In the situation of Theorem \ref{thm--Kollar's--lemma}, suppose that $(X_{\bar{s}},\Delta_{\bar{s}})$ is klt for any geometric point $\bar{s}\in S$, there exists a closed point $s_0\in S$ such that $(X_{s_0},\Delta_{s_0},L_{s_0})$ is specially {\rm K}-stable and $K_{X/S}+\Delta+\delta_{(X_{s_0},\Delta_{s_0})}(L_{s_0})L$ is $f$-ample.
Suppose further that $\mathbbm{k}=\mathbb{C}$.
Then the CM-line bundle $\lambda_{\mathrm{CM},f}$ (cf.~\cite[Definition 3.2]{Hat23}) is big and nef.
\end{cor}

\begin{proof}
    This follows from Theorem \ref{thm--Kollar's--lemma} and the proof of \cite[Theorem 5.4]{Hat23}.
\end{proof}

\begin{cor}\label{cor--partial--positivity--of--cm}
 Fix $d\in\mathbb{Z}_{>0}$, $u,v\in\mathbb{Q}_{>0}$ and $w\in\mathbb{Q}_{>0}$.
 Then, there exists $t_0\in\mathbb{Q}_{>0}$ that satisfies the following:
 For any $\mu>t_0$, and proper subspace $B\subset {M}_{d,v,u,r,w}$, $\Lambda_{\mathrm{CM},\mu}|_B$ is ample.
\end{cor}

\begin{proof}
First, we deal with the case where $\mathbbm{k}=\mathbb{C}$.
By \cite[Corollary 1.7]{HH}, there exists $t_0>0$ such that $(X,A+\mu K_{X})$ is specially K-stable for any $(X,A)\in \mathcal{M}_{d,v,u,r,w}(\mathbbm{k})$ and $\mu>t_0$.
Note that $K_{X}+\delta_{(X,0)}(A+\mu K_{X})( A+\mu K_{X})$ is ample for any $(X,A)$ as above by \cite[Theorem A]{BlJ}.
By \cite[Proposition 2.7]{kollar-moduli-stable-surface-proj}, we see that there exists a finite cover $g\colon B'\to B$ such that $B'$ is a proper scheme and $B'\to {M}_{d,v,u,r,w}$ factors through $B'\to\mathcal{M}_{d,v,u,r,w}$.
    Using Corollary \ref{cor--cm--special--bigness} and \cite[Theorem 4.3]{Hat23}, we see that $g^*\Lambda_{\mathrm{CM},\mu}|_V$ is big and nef for any proper irreducible subscheme $V\subset B'$.
    By the Nakai-Moishezon criterion \cite[Theorem 1.3]{FM3}, $\Lambda_{\mathrm{CM},\mu}|_B$ is ample.

    For the general $\mathbbm{k}$, we can reduce to the special case as the proof of Theorem \ref{thm--kst--u>0--CY-fib}.
    We complete the proof.
\end{proof}

\section{The numerical moduli of Viehweg revisited}
\label{sec-5}

In this section, we deal with Theorem \ref{thm--intro--const--N}. 
Our aim is to generalize Viehweg's method \cite{viehweg91} to construct the moduli of numerical equivalence classes for the klt case.
From now on, we will construct the following moduli stack as a quotient stack.

\begin{stup}\label{stup--5}
    Fix $d\in\mathbb{Z}_{>0}$, $v\in\mathbb{Q}_{>0}$, $u\in\mathbb{Q}_{\ne0}$ and take $w\in\mathbb{Q}_{>0}$ such that $\mathfrak{Z}_{d, v,u,w}\to \mathfrak{Z}_{d, v,u}$ is surjective.
    Then, there exists $I\in\mathbb{Z}_{>0}$ such that $IA$ is a very ample Cartier divisor, $H^i(X,\mathcal{O}_X(IA))=0$ for any $i>0$, and there are only finitely many possibilities of Hilbert polynomials of $X$ with respect to $A$ for any element of $\mathfrak{Z}_{d, v,u,w}$ (cf.~\cite[Lemma 3.8]{HH}).
    We may assume that $r$ as Definition \ref{defn--HH-moduli} divides $I$.
    Moreover, by Proposition \ref{prop--xz}, we may assume that $ID$ is Cartier for any $\mathbb{Q}$-Cartier Weil divisor $D$ on $X_{\bar{s}}$ for any $s\in S$.
Using Corollary \ref{cor--torsion--components--number} (see also \cite[Proposition 6.1]{HH}), we may further assume that the cardinality of $\mathbf{W}^\mathbb{Q}\mathbf{Pic}_{X/\mathbbm{k}}^\tau/\mathbf{Pic}^0_{X/\mathbbm{k}}$ divides $I$ for any element of $\mathfrak{Z}_{d, v,u,w}$.
We can consider $\mathcal{M}_{d,I^{d-1}\cdot v,u,I^d\cdot w,r}$ the moduli stack.
Let $(\mathcal{U},\mathcal{A})\to \mathcal{M}_{d,I^{d-1}\cdot v,u,r,I^d\cdot w}$ be the universal family.
We explain briefly how $(\mathcal{U},\mathcal{A})\to \mathcal{M}_{d,I^{d-1}\cdot v,u,r,I^d\cdot w}$ is defined.
Note that each connected component of $\mathcal{M}_{d,I^{d-1}\cdot v,u,r,I^d\cdot w}$ can be represented as $[H'/G]$ for some quasi-projective scheme $H'$ with an action of a connected affine semisimple group. 
See the proof of \cite[Theorem 5.1 and Proposition 5.6]{HH} for details and note that $H'$ is a locally closed subscheme of a certain Hilbert scheme.
We also note that $(\mathcal{U}_{H'},\mathcal{A}_{H'})\to H'$ is the universal family corresponding to $H'\to \mathcal{M}_{d,I^{d-1}\cdot v,u,I^d\cdot w,r}$.
Then, we can define $[\mathcal{U}_{H'}/G]\to [H'/G]$, $[\mathbf{Pic}_{\mathcal{U}_{H'}/H'}/G]$, and $[\mathbf{W}^{\mathbb{Q}}\mathbf{Pic}_{\mathcal{U}_{H'}/H'}/G]$.
Indeed, it is easy to check that $G$ acts on $\mathbf{Pic}_{\mathcal{U}_{H'}/H'}$, $\mathbf{Pic}^0_{\mathcal{U}_{H'}/H'}$, $\mathbf{Pic}^\tau_{\mathcal{U}_{H'}/H'}$, $\mathbf{W}^{\mathbb{Q}}\mathbf{Pic}_{\mathcal{U}_{H'}/H'}$ and $\mathbf{W}^{\mathbb{Q}}\mathbf{Pic}^\tau_{\mathcal{U}_{H'}/H'}$ in natural ways.
By patching them together over all connected components, we can define $\mathcal{U}$, $\mathbf{Pic}_{\mathcal{U}/\mathcal{M}_{d,I^{d-1}\cdot v,u,r,I^d\cdot w}}$ and $\mathbf{W}^{\mathbb{Q}}\mathbf{Pic}_{\mathcal{U}/\mathcal{M}_{d,I^{d-1}\cdot v,u,r,I^d\cdot w}}$.
Let $\sigma$ be the morphism $\mathcal{M}_{d,I^{d-1}\cdot v,u,r,I^d\cdot w}\to  \mathbf{Pic}_{\mathcal{U}/\mathcal{M}_{d,I^{d-1}\cdot v,u,r,I^d\cdot w}}$, which corresponds to the canonical object defined by patching $\mathcal{A}$ as above.
\end{stup}

We define the following two important two stacks.

\begin{de}\label{de--WI}
We set $$\mathcal{M}^W_{d,v,u,r,w}:=\mathcal{M}_{d,I^{d-1}\cdot v,u,r,I^d\cdot w}\times_{\sigma,\mathbf{Pic}_{\mathcal{U}/\mathcal{M}_{d,I^{d-1}\cdot v,u,r,I^d\cdot w}},\mu_I}\mathbf{W}^{\mathbb{Q}}\mathbf{Pic}_{\mathcal{U}/\mathcal{M}_{d,I^{d-1}\cdot v,u,r,I^d\cdot w}}.$$ 
Note that the canonical morphism $\mathcal{M}^W_{d,v,u,r,w}\to \mathcal{M}_{d,I^{d-1}\cdot v,u,r,I^d\cdot w}$ is a finite morphism by Proposition \ref{prop--mu_k}.
Thus, the image of $\mathcal{M}^W_{d,v,u,r,w}\to \mathcal{M}_{d,I^{d-1}\cdot v,u,r,I^d\cdot w}$ is closed.
Then, we set the scheme-theoretic image as $\mathcal{M}^{W,I,\mathrm{im}}_{d,v,u,r,w}\subset \mathcal{M}_{d,I^{d-1}\cdot v,u,r,I^d\cdot w}$.
Note that the reduced structures are independent of the choice of $r$ by Proposition \ref{prop--univ--closed--r}.
We write their reduced structures as $(\mathcal{M}^{W}_{d,v,u,w})_{\mathrm{red}}$ and $(\mathcal{M}^{W,I}_{d,v,u,w})_{\mathrm{red}}$, respectively.
\end{de}

We note that $\mathcal{M}^W_{d,v,u,r,w}$ and $\mathcal{M}^{W,I,\mathrm{im}}_{d,v,u,r,w}$ are separated Deligne--Mumford stacks of finite type over $\mathbbm{k}$.
Thus, they admit the coarse moduli spaces $M^W_{d,v,u,r,w}$ and $M^{W,I,\mathrm{im}}_{d,v,u,r,w}$ by \cite{KeM}.
We also note that there exists a natural morphism $(\mathcal{M}_{d,v,u,w})_{\mathrm{red}}\to (\mathcal{M}^{W}_{d,v,u,w})_{\mathrm{red}}$.
The image is contained in the open substack
\[
(\mathcal{M}_{d,I^{d-1}\cdot v,u,r,I^d\cdot w}\times_{\sigma,\mathbf{Pic}_{\mathcal{U}/\mathcal{M}_{d,I^{d-1}\cdot v,u,r,I^d\cdot w}},\mu_I}\mathbf{Pic}_{\mathcal{U}/\mathcal{M}_{d,I^{d-1}\cdot v,u,r,I^d\cdot w}})_{\mathrm{red}}.
\]
On the other hand, there exists a natural morphism
\[
(\mathcal{M}_{d,I^{d-1}\cdot v,u,r,I^d\cdot w}\times_{\sigma,\mathbf{Pic}_{\mathcal{U}/\mathcal{M}_{d,I^{d-1}\cdot v,u,r,I^d\cdot w}},\mu_I}\mathbf{Pic}_{\mathcal{U}/\mathcal{M}_{d,I^{d-1}\cdot v,u,r,I^d\cdot w}})_{\mathrm{red}}\to (\mathcal{M}_{d,v,u,w})_{\mathrm{red}}.
\]
It is easy to see that the above is an isomorphism. Hence, the morphism $(\mathcal{M}_{d,v,u,w})_{\mathrm{red}}\to (\mathcal{M}^{W}_{d,v,u,w})_{\mathrm{red}}$ is an open immersion and the morphism $(\mathcal{M}_{d,v,u,w})_{\mathrm{red}}\to(\mathcal{M}^{W,I}_{d,v,u,w})_{\mathrm{red}}$ is quasi-finite.

We want to construct $\mathscr{N}$ in Subsection \ref{subsec--main--result} as the quotient stack of $\mathcal{M}^{W,I,\mathrm{im}}_{d,v,u,r,w}$ by the action of the Picard scheme.
However, if you consider such a stack naively, then the stack should not be Deligne--Mumford but an Artin stack without affine diagonal. Due to this, we cannot apply the theory of \cite{AHLH} to obtain any moduli space.
To avoid this subtlety, we will find a smaller Abelian subscheme of the Picard scheme as Definition \ref{de--smaller--abelian--scheme} after replacing $\mathcal{M}^{W,I,\mathrm{im}}_{d,v,u,r,w}$.

By construction of $\mathcal{M}^{W,I,\mathrm{im}}_{d,v,u,r,w}$, for any connected component $\mathcal{M}^\mathrm{im}$, there exists a quasi-projective scheme $H^\mathrm{im}$ such that $\mathcal{M}^\mathrm{im}\cong [H^\mathrm{im}/G]$.
Let $M^\mathrm{im}$ be the coarse moduli space of $\mathcal{M}^\mathrm{im}$.
Suppose that $H'$ as Setup \ref{stup--5} contains $H^\mathrm{im}$.
Let $(\mathcal{U}_{H^\mathrm{im}},\mathcal{A}_{H^\mathrm{im}})\to {H}^\mathrm{im}$ be the restriction of the universal closed subscheme over $H'$ with natural polarization.
We can also regard $(\mathcal{U}_{H^\mathrm{im}},\mathcal{A}_{H^\mathrm{im}})$ as the pullback of the universal family over $\mathcal{M}_{d,I^{d-1}\cdot v,u,r,I^d\cdot w}$ to $H^\mathrm{im}$.
By construction, we see that $\mathcal{A}_{H^\mathrm{im}}$ is a line bundle.
Suppose that any geometric fiber of $(\mathcal{U}_{H^\mathrm{im}},\mathcal{A}_{H^\mathrm{im}})$ has the Hilbert polynomial $P$ over $H^\mathrm{im}$ and let $Q$ be the polynomial such that $Q(\bullet)=P(2\bullet)$.
By the proof of \cite[Theorem 5.1 and Proposition 5.6]{HH}, we know that $\mathcal{A}_{H^\mathrm{im}}$ is relatively ample and $\mathcal{U}_{H^\mathrm{im}}\subset \mathbb{P}^{P(J)-1}\times \mathbb{P}^{P(J+1)-1}\times \mathbb{P}^{h^0(\mathcal{U}_s,\mathcal{O}_{\mathcal{U}_s}(rK_{\mathcal{U}_s}))-1}\times H^\mathrm{im}$ for some $s\in H^\mathrm{im}$, $J$ and $r\in\mathbb{Z}_{>0}$.
Let $\mathrm{Aut}^{Q,\mathrm{alg}}_{H^\mathrm{im}}(\mathcal{U}_{H^\mathrm{im}})$ be the largest open subgroup scheme of $\mathrm{Aut}_{\mathcal{U}_{H^\mathrm{im}}/H^\mathrm{im}}$ such that for any morphism $T\to H^\mathrm{im}$ from a scheme and $\varphi\in \mathrm{Aut}^{Q}_{H^\mathrm{im}}(\mathcal{U}_{H^\mathrm{im}})(T)$, $\varphi^*\mathcal{A}_T\otimes\mathcal{A}_T^{\otimes(-1)}\in\mathbf{Pic}^0_{\mathcal{U}_T/T}(T)$ if and only if $\varphi$ belongs to $\mathrm{Aut}^{Q,\mathrm{alg}}_{H^\mathrm{im}}(\mathcal{U}_{H^\mathrm{im}})(T)$.
In this case, $Q$ is the Hilbert polynomial of $\mathcal{U}_T$ with respect to $\varphi^*\mathcal{A}_T\otimes\mathcal{A}_T$ (cf.~\cite[Theorem 9.6.3]{FGA}) and hence $\mathrm{Aut}^{Q,\mathrm{alg}}_{H^\mathrm{im}}(\mathcal{U}_{H^\mathrm{im}})$ is of finite type over $H^\mathrm{im}$.
Consider a morphism $\psi\colon\mathrm{Aut}^{Q,\mathrm{alg}}_{H^\mathrm{im}}(\mathcal{U}_{H^\mathrm{im}})\times_{H^\mathrm{im}} \mathbf{Pic}^0_{\mathcal{U}_{H^\mathrm{im}}/H^\mathrm{im}}\to H^\mathrm{im}\times H^\mathrm{im}$ such that for any $H^\mathrm{im}$-scheme $T$ and any element $(\varphi,\lambda)\in (\mathrm{Aut}^{Q,\mathrm{alg}}_{H^\mathrm{im}}(\mathcal{U}_{H^\mathrm{im}})\times_H \mathbf{Pic}^0_{\mathcal{U}_{H^\mathrm{im}}/H^\mathrm{im}})(T)$ over $\mathcal{U}_T$, $\psi$ maps $(\varphi,\lambda)$ to $((\mathcal{U}_T,\varphi_*\mathcal{A}_T+\lambda),(\mathcal{U}_T,\mathcal{A}_T))$.
We note that $G$ acts on $\mathrm{Aut}^{Q,\mathrm{alg}}_{H^\mathrm{im}}(\mathcal{U}_{H^\mathrm{im}})$ and $\mathbf{Pic}^0_{\mathcal{U}_{H^\mathrm{im}}/H^\mathrm{im}}$ in the canonical way.
\begin{lem}\label{lem--psi--inv}
    $\psi$ is $G$-equivariant and proper.
    Furthermore, if we let $p_{H^\mathrm{im},2}\colon H^\mathrm{im}\times H^\mathrm{im}\to H^\mathrm{im}$ be the second projection, then $p_{H^\mathrm{im},2}\circ\psi$ is proper.
\end{lem}

\begin{proof}
 We know that $\psi$ is a separated morphism of finite type between Noetherian schemes. 
 We note that the $G$-equivariantness of $\psi$ is easy to deduce.
 Therefore, it suffices to show that $p_{H^\mathrm{im},2}\circ\psi$ is proper.
 Since $p_{H^\mathrm{im},2}\circ\psi$ is the structure morphism $\mathrm{Aut}^{Q,\mathrm{alg}}_{H^\mathrm{im}}(\mathcal{U}_{H^\mathrm{im}})\times_{H^\mathrm{im}} \mathbf{Pic}^0_{\mathcal{U}_{H^\mathrm{im}}/{H^\mathrm{im}}}\to H^\mathrm{im}$ and $\mathbf{Pic}^0_{\mathcal{U}_{H^\mathrm{im}}/H^\mathrm{im}}$ is proper over $H^\mathrm{im}$, it suffices to show that $\mathrm{Aut}^{Q,\mathrm{alg}}_{H^\mathrm{im}}(\mathcal{U}_{H^\mathrm{im}})$ is proper over $H^\mathrm{im}$.
 We can apply \cite[Theorem 11.5.1]{Ols} to $\mathrm{Aut}^{Q,\mathrm{alg}}_{H^\mathrm{im}}(\mathcal{U}_{H^\mathrm{im}})$.
 Let $R$ be a discrete valuation ring essentially of finite type over $\mathbbm{k}$ with the fractional field $K$ (cf.~\cite[Lemma A.11]{AHLH}).
 Suppose that there exists a morphism $\varphi_2\colon \mathrm{Spec}\,R\to H^\mathrm{im}$ such that there exists an element $\phi\in \mathrm{Aut}^{Q,\mathrm{alg}}_{H^\mathrm{im}}(\mathcal{U}_{H^\mathrm{im}})(K)$ such that $\phi$ is mapped to $(\varphi_2)_K$, which is a restriction of $\varphi_2$ to $\mathrm{Spec}\,K$.
 Let $(X_2,A_2)$ be the family over $R$ corresponding to $\varphi_2$.
$\phi^*\mathcal{A}_K\otimes \mathcal{A}^{\otimes(-1)}_K\in \mathbf{Pic}^0_{\mathcal{U}_{H^\mathrm{im}}/H^\mathrm{im}}(K)$ can be extended to some $\lambda\in \mathbf{Pic}^0_{\mathcal{U}_{H^\mathrm{im}}/H^\mathrm{im}}(R)$ by the properness.
Taking a finite cover of $R$, we may assume that $\lambda$ is a line bundle.
Here, we note that $R$ is no longer a discrete valuation ring, but a Dedekind domain, in general.
Applying \cite[Theorem 4.6]{HH} (this theorem is still valid for the case where the base is a Dedekind domain essentially of finite type over $\mathbbm{k}$) to $\phi$ for two families $(X_2,A_2+\lambda)$ and $(X_2,A_2)$, $\phi$ extends to an isomorphism $\overline{\phi}\colon (X_2,A_2+\lambda)\to (X_2,A_2)$ over $R$.
Therefore, $\mathrm{Aut}^{Q,\mathrm{alg}}_{H^\mathrm{im}}(\mathcal{U}_{H^\mathrm{im}})$ is proper over $H^\mathrm{im}$.
We complete the proof.
\end{proof}

Therefore, the image $Z$ of $\psi$ is closed in $H^\mathrm{im}\times H^\mathrm{im}$.
We set $\mathcal{R}\subset M^\mathrm{im}\times M^\mathrm{im}$ as the coarse moduli space of $[Z/G]$ by Lemma \ref{lem--psi--inv}.
Note that $\mathcal{R}$ is a closed equivalence relation of the algebraic space $M^\mathrm{im}$ such that the first and second projections $p_1,p_2\colon \mathcal{R}\to M^\mathrm{im}$ are proper.

\begin{prop}\label{prop--aut--dimension}
  $\mathrm{dim}\,\mathrm{Aut}_{X}$ is the same for any $(X,L)$ parameterized by $M^\mathrm{im}$.
  
  Furthermore, $\mathrm{dim}\,p_2^{-1}(m)$ is independent of the choice of $m\in M^\mathrm{im}$.
\end{prop}

\begin{proof}
This is essentially discussed in \cite[Lemma 4.1.4]{kollar-toward} but we give a proof here for the reader's convenience.

    Since each fiber $\mathcal{R}_{m}:=p_2^{-1}(m)$ of $p_2$ over a point $m\in M^\mathrm{im}$ corresponding to a polarized variety $(X_m,A_m)$ is of dimension $h^1(X_m,\mathcal{O}_{X_m})-\mathrm{dim}\,\mathrm{Aut}_{X_m}$ and $p_2$ is proper, we see that $\mathrm{dim}\,\mathrm{Aut}_{X_m}$ is lower semicontinuous in $m$ by \cite[Tag 0D4Q]{Stacks} because $h^1(X_m,\mathcal{O}_{X_m})$ is independent of the choice of $m$ by \cite[Corollary 2.61]{kollar-moduli}.
    On the other hand, it is easy to check that $\mathrm{dim}\,\mathrm{Aut}_{X_m}$ is upper semicontinuous in $m$ by \cite[Tag 02FZ]{Stacks}.
    This means that $m\mapsto \mathrm{dim}\,\mathrm{Aut}_{X_m}$ is locally constant.
    We conclude the proof.
\end{proof}

We also note that there exists a natural surjective morphism $\mathbf{Pic}^0_{\mathcal{U}_{H^\mathrm{im}}/H^\mathrm{im}}\to \mathcal{R}$. 
Consider the following morphism 
\[
h\colon \mathbf{Pic}^0_{\mathcal{U}_{H^\mathrm{im}}/H^\mathrm{im}}\to \mathcal{R}\times_{p_2,M^\mathrm{im}}H^\mathrm{im}.
\]
Let $Y^\mathrm{im}\to \mathcal{R}\times_{p_2,M^\mathrm{im}}H^\mathrm{im}$ be the Stein factorization.
Let $s^{\mathrm{im}}\colon H^{\mathrm{im}}\to Y^{\mathrm{im}}$ be the image of the zero section of $\mathbf{Pic}^0_{\mathcal{U}_{H^{\mathrm{im}}}/H^{\mathrm{im}}}$ and $K^{\mathrm{im}}:=\mathbf{Pic}^0_{\mathcal{U}_{H^{\mathrm{im}}}/H^{\mathrm{im}}}\times_{Y^{\mathrm{im}},s^{\mathrm{im}}}H^{\mathrm{im}}$.

\begin{de}
    Let $\pi\colon V\to S$ be a proper morphism of locally Noetherian schemes such that $\pi_*\mathcal{O}_V\cong \mathcal{O}_S$.
    If $V$ is a group scheme over $S$, we say that $V$ is an {\it Abelian scheme} over $S$. See \cite[Section 6]{GIT} for more details.
\end{de}
We note that $\mathbf{Pic}^0_{\mathcal{U}_{H^\mathrm{im}}/H^\mathrm{im}}$ is an Abelian scheme over $H^{\mathrm{im}}$ by Lemma \ref{lem--pic-const}.
Then, we claim the following.
\begin{prop}\label{prop--de-Y}
There exists a closed subscheme $H\subset H^{\mathrm{im}}$ such that $H$ contains $H^{\mathrm{im}}_{\mathrm{red}}$ and for any morphism $g\colon T\to H^{\mathrm{im}}$ from a scheme, $K:=K^{\mathrm{im}}\times_{H^{\mathrm{im}}}T$ is an Abelian scheme if and only if $g$ factors through $H$.
Furthermore,  there exists a projective Abelian scheme $Y:=\mathbf{Pic}^0_{\mathcal{U}_{H}/H}/K$ over $H$ with a natural action of $G$ and a canonical homeomorphism to $Y^{\mathrm{im}}$.
\end{prop}

\begin{proof}
    We claim that $\mathbf{Pic}^0_{\mathcal{U}_{H^\mathrm{im}}/H^\mathrm{im}}$ acts on $Y^\mathrm{im}$ over $H^\mathrm{im}$.
Indeed, it is easy to check that $h$ is naturally $\mathbf{Pic}^0_{\mathcal{U}_{H^\mathrm{im}}/H^\mathrm{im}}$-equivariant, where $\mathbf{Pic}^0_{\mathcal{U}_{H^\mathrm{im}}/H^\mathrm{im}}$ acts on $\mathcal{R}$ on the first factor.
By \cite[III, Proposition 9.3]{Ha}, we see that $Y^{\mathrm{im}}$ also admits a natural $\mathbf{Pic}^0_{\mathcal{U}_{H^{\mathrm{im}}}/H^{\mathrm{im}}}$-action over $H^{\mathrm{im}}$ such that $h$ is $\mathbf{Pic}^0_{\mathcal{U}_{H^{\mathrm{im}}}/H^{\mathrm{im}}}$-equivariant. 
Similarly, it is easy to see that $Y^{\mathrm{im}}$ also admits a natural action of $G$.
Therefore, $K^{\mathrm{im}}$ is the stabilizer group scheme of $s^{\mathrm{im}}$ and admits a natural $G$-action.

Next, we will show that $K^{\mathrm{im}}\times_{H^{\mathrm{im}}}H_{\mathrm{red}}^{\mathrm{im}}$ is smooth over $H_{\mathrm{red}}^{\mathrm{im}}$.
Since we are working over an algebraically closed field $\mathbbm{k}$ of characteristic zero, 
any group scheme over $\mathbbm{k}$ is smooth (cf.~\cite[Section 11, Theorem]{Ab}).
Thus, we see that any geometric fiber of $K^{\mathrm{im}}\times_{H^{\mathrm{im}}}H_{\mathrm{red}}^{\mathrm{im}}$ over $H_{\mathrm{red}}^{\mathrm{im}}$ is smooth, connected, proper, and of the same dimension by Proposition \ref{prop--aut--dimension} and construction.
Fix a relatively ample line bundle $A'$ on $K^{\mathrm{im}}\times_{H^{\mathrm{im}}}H_{\mathrm{red}}^{\mathrm{im}}$ over $H_{\mathrm{red}}^{\mathrm{im}}$.
For any morphism $C\to H^{\mathrm{im}}_{\mathrm{red}}$ from a smooth curve, consider $K^{\mathrm{im}}\times_{H^{\mathrm{im}}}C$.
We note that the reduced structure of $K^{\mathrm{im}}\times_{H^{\mathrm{im}}}C$ is a variety projective over $C$ such that any fiber over $C$ is smooth.
By \cite[III, Theorem 9.11]{Ha}, we see that any geometric fiber over $C$ has the same Hilbert polynomial with respect to $A'$.
By the fact that we can find a morphism $C\to H_{\mathrm{red}}$ from a connected curve passing through two points (cf.~\cite[Section 6, Lemma]{Ab}) and \cite[Proposition 2.1.2]{HL}, $K^{\mathrm{im}}\times_{H^{\mathrm{im}}}H_{\mathrm{red}}^{\mathrm{im}}$ is smooth over $H_{\mathrm{red}}^{\mathrm{im}}$.
This shows the claims in this paragraph.


By \cite[Theorem 5.13]{FGA} and what we have shown, there exists a closed subscheme $H$ such that $H$ contains $H^{\mathrm{im}}_{\mathrm{red}}$ and for any morphism $g\colon T\to H^{\mathrm{im}}$, $K^{\mathrm{im}}\times_{H^{\mathrm{im}}}T$ is flat if and only if $g$ factors through $H$.
Since any geometric fiber is smooth and connected, $K^{\mathrm{im}}\times_{H^{\mathrm{im}}}T$ is an Abelian scheme over $T$ in the above situation.
Consider an Artin stack $[\mathbf{Pic}^0_{\mathcal{U}_{H}/H}/K]$ (cf.~\cite[Example 8.1.12]{Ols}).
By \cite[Corollary 8.3.5]{Ols}, $[\mathbf{Pic}^0_{\mathcal{U}_{H}/H}/K]$ is an algebraic space that admits a quotient group structure. We write it as $Y:=\mathbf{Pic}^0_{\mathcal{U}_{H}/H}/K$.
By \cite[Theorem 1.9]{FC}, $Y$ is an Abelian scheme over $H$.
Since $K$ and $\mathbf{Pic}^0_{\mathcal{U}_{H}/H}$ admit natural $G$-actions, so does $Y$.
By construction of $Y$, it is easy to see that there is a canonical homeomorphism to $Y^{\mathrm{im}}$.
Applying Corollary \ref{cor--partial--positivity--of--cm} and \cite[Tags 0D36 and 0D3A]{Stacks} to $p_2$, we see that $Y^\mathrm{im}\to H^{\mathrm{im}}$ is projective. Therefore, $Y$ is also projective over $H$.
We complete the proof.
\end{proof}

Let $\mathcal{M}:=[H/G]$, $\widehat{Y}:=\mathbf{Pic}^0_{Y/H}$ and $M$ be the coarse moduli space of $\mathcal{M}$. 
Let $X$ denote $\mathbf{Pic}^0_{\mathcal{U}_H/H}$. 
We note that $\widehat{X}:=\mathbf{Pic}^0_{X/H}$ is also an Abelian scheme projective over $H$.
Fix a relatively ample line bundle $L_{\widehat{X}}$ on $\widehat{X}$ over $H$.
Let $m\colon \widehat{X}\times_H\widehat{X}\to \widehat{X}$ be the multiplication morphism and $p_1,p_2\colon \widehat{X}\times_H\widehat{X}\to \widehat{X}$ the first and second projections.
Regarding $m^*L_{\widehat{X}}\otimes p_1^*L_{\widehat{X}}^{\otimes-1}\otimes p_2^*L_{\widehat{X}}^{\otimes-1}$ as a line bundle over $\widehat{X}$, we see that there exists a morphism $ \widehat{X}\to X$.
This is a finite \'etale morphism by \cite[(20.G)]{Mat} (cf.~\cite[Section 6.2]{GIT}).
Let $\Lambda(L_{\widehat{X}})$ denote this morphism.
We note that  $\widehat{Y}$ and $\widehat{X}$ also admit natural $G$-actions.
Then, the following holds.
\begin{lem}\label{lem--G-equiv}
    $\Lambda(L_{\widehat{X}})$ is $G$-equivariant.
\end{lem}

\begin{proof}
    Let $a_{\widehat{X}}\colon G\times\widehat{X}\to G\times \widehat{X}$ and $a_X\colon G\times X\to G\times X$ be the product of the action morphisms and the first projections.
    Let $\pi_{\widehat{X}}\colon G\times\widehat{X}\to G\times H$ and $\pi_{X}\colon G\times {X}\to G\times H$ be the canonical projections.
    We set $\xi\colon G\times H\to G\times H$ as the morphism that is obtained by the action morphism $G\times H\to H$ and the first projection.
    Then, it is easy to see that $\xi\circ \pi_{\widehat{X}}=\pi_X\circ a_{\widehat{X}}$.
    Let $\mathrm{id}_G\times\Lambda(L_{\widehat{X}})\colon G\times \widehat{X}\to G\times X$ be the morphism induced by $\Lambda(L_{\widehat{X}})$.
    Here, it suffices to show that $a_X\circ(\mathrm{id}_G\times\Lambda(L_{\widehat{X}}))=(\mathrm{id}_G\times\Lambda(L_{\widehat{X}}))\circ a_{\widehat{X}}$.
    Via $\xi\circ \pi_{\widehat{X}}$, we can regard $a_X\circ(\mathrm{id}_G\times\Lambda(L_{\widehat{X}}))$ and $(\mathrm{id}_G\times\Lambda(L_{\widehat{X}}))\circ a_{\widehat{X}}$ as morphisms of schemes proper and flat over $G\times H$.
    Since $G\times{X}$ is an Abelian scheme over $G\times H$, it is enough to show that $a_X\circ(\mathrm{id}_G\times\Lambda(L_{\widehat{X}}))-(\mathrm{id}_G\times\Lambda(L_{\widehat{X}}))\circ a_{\widehat{X}}$ factors through the zero section.
    We note that $a_X\circ(\mathrm{id}_G\times\Lambda(L_{\widehat{X}}))-(\mathrm{id}_G\times\Lambda(L_{\widehat{X}}))\circ a_{\widehat{X}}$ maps $\pi_{\widehat{X}}^{-1}(\{e_G\}\times H)$ into the zero section, where $e_G$ is the image of the identity section $\mathbf{0}_{X}\colon H\to X$.
    Let $\mathbf{0}_{\widehat{X}}\colon H\to\widehat{X}$ be the zero section.
    Then, $a_X\circ(\mathrm{id}_G\times\Lambda(L_{\widehat{X}}))-(\mathrm{id}_G\times\Lambda(L_{\widehat{X}}))\circ a_{\widehat{X}}$ maps $(\mathbf{0}_{\widehat{X}})\times\{h\}$ to $(\mathbf{0}_{X})\times\{g\cdot h\}$ for any closed points $g\in G$ and $h\in H$.
    By \cite[Lemma 6.1]{GIT}, $a_X\circ(\mathrm{id}_G\times\Lambda(L_{\widehat{X}}))-(\mathrm{id}_G\times\Lambda(L_{\widehat{X}}))\circ a_{\widehat{X}}$ factors through the zero section.
    We are done.
\end{proof}

We note that the natural morphism $\iota_{\widehat{Y}}\colon \widehat{Y}\hookrightarrow \widehat{X}$ is a closed immersion.
It is also easy to see that $p\circ \Lambda(L_{\widehat{X}})\circ\iota_{\widehat{Y}}=\Lambda(L_{\widehat{X}}|_{\widehat{Y}})$ and hence it is proper and \'etale.

\begin{lem}\label{lem--im--str}
Let $\widehat{\widehat{Y}}$ be the image structure of $\Lambda(L_{\widehat{X}})\circ\iota_{\widehat{Y}}$.
Then, $\widehat{\widehat{Y}}$ is an Abelian scheme over $H$ and $\widehat{Y}\to\widehat{\widehat{Y}}$ is an \'etale and proper morphism.
\end{lem}

\begin{proof}
Since $\Lambda(L_{\widehat{X}}|_{\widehat{Y}})$ is a finite \'etale group morphism, we see that $\mathrm{Ker}(\Lambda(L_{\widehat{X}}|_{\widehat{Y}}))$ is a group scheme \'etale over $H$.   
Consider $\mathrm{Ker}(\Lambda(L_{\widehat{X}})\circ\iota_{\widehat{Y}})$.
Then $\mathrm{Ker}(\Lambda(L_{\widehat{X}})\circ\iota_{\widehat{Y}})\subset \mathrm{Ker}(\Lambda(L_{\widehat{X}}|_{\widehat{Y}}))$.
We note that $(\Lambda(L_{\widehat{X}})\circ\iota_{\widehat{Y}})(\mathrm{Ker}(\Lambda(L_{\widehat{X}}|_{\widehat{Y}})))$ is contained in a group subscheme in $X$ finite and \'etale over $H$.
Therefore, we see that $\mathrm{Ker}(\Lambda(L_{\widehat{X}})\circ\iota_{\widehat{Y}})\to H$ is an open morphism.
We conclude that $\mathrm{Ker}(\Lambda(L_{\widehat{X}})\circ\iota_{\widehat{Y}})$ is a group scheme \'etale over $H$. 
Then, we note that $\widehat{\widehat{Y}}=\widehat{Y}/\mathrm{Ker}(\Lambda(L_{\widehat{X}})\circ\iota_{\widehat{Y}})$.
We complete the proof.
\end{proof}

\begin{de}\label{de--smaller--abelian--scheme}
Note that $\widehat{\widehat{Y}}$ admits a natural $G$-action inherited from $\mathbf{Pic}^0_{\mathcal{U}_H/H}$.
We set $\mathcal{Y}$ as the quotient stack $[\widehat{\widehat{Y}}/G]$. 
We note that the natural morphism $\mathcal{Y}\to \mathcal{M}$ is representable, proper and smooth.
\end{de}

\begin{lem}\label{lem--y--identify}
    Let $g_1,g_2\colon T\to \mathcal{M}$ be two morphisms from the same scheme and $(\mathcal{X}_1,\mathcal{L}_1)$ and $(\mathcal{X}_2,\mathcal{L}_2)$ the families corresponding to $g_1$ and $g_2$, respectively.
    Suppose that there exists an isomorphism $f\colon \mathcal{X}_1\to \mathcal{X}_2$ over $T$ and let $f^*\colon \mathbf{Pic}^0_{\mathcal{X}_2/T}\to \mathbf{Pic}^0_{\mathcal{X}_1/T}$ be the induced isomorphism.
    Suppose further that $f^*\mathcal{L}_2-\mathcal{L}_1$ is an element of $\mathbf{Pic}^0_{\mathcal{X}_1/T}(T)$.

    Then, $f^*$ maps $\mathcal{Y}\times_{\mathcal{M},g_2}T$ to $\mathcal{Y}\times_{\mathcal{M},g_1}T$.
\end{lem}

\begin{proof}
It suffices to deal with the case $T$ is a spectrum of an algebraically closed field.
Thus, we may assume that $T=\mathrm{Spec}\,\mathbbm{k}$.
First, we consider the case where $f^*\mathcal{L}_2\sim_T \mathcal{L}_1$.
Then $f^*$ maps $\mathcal{Y}\times_{\mathcal{M},g_2}T$ to $\mathcal{Y}\times_{\mathcal{M},g_1}T$ by Proposition \ref{prop--de-Y}, Lemmas \ref{lem--G-equiv} and \ref{lem--im--str}.

Next, we consider the general case.
We identify $\mathcal{X}_1$ and $\mathcal{X}_2$ by $f$.
By assumption, it is easy to see that there exists a family $(\mathcal{X}_1\times C,\mathcal{L})\to C$ corresponding to 
a morphism $h\colon C\to \mathcal{M}$ from an irreducible smooth affine curve $C$ with two closed points $c_1$ and $c_2$ such that $h(c_1)$ and $h(c_2)$ correspond to $(\mathcal{X}_1,\mathcal{L}_1)$ and $(\mathcal{X}_1,\mathcal{L}_2)$, respectively.
Since the canonical morphism $\pi_H\colon H\to \mathcal{M}$ is smooth, for any $c\in C$ there exists an \'etale morphism $\pi_{C'}\colon C'\to C$ with a closed point $c'\in C'$ such that $\pi_{C'}(c')=c$ and a morphism $h'\colon C'\to H$ such that $h\circ \pi_{C'}$ and $\pi_H\circ h'$ are isomorphic.  
Thus, by what we have shown in the first paragraph, we may assume that there exists an \'etale morphism $\pi_{C'}\colon C'\to C$ as above with two closed points $c'_1$ and $c'_2\in C'$ such that $C'$ is irreducible and that $\pi_{C'}(c'_1)=c_1$ and $\pi_{C'}(c'_2)=c_2$.
Let $h'(c_1')$ and $h'(c_2')$ be $x_1$ and $x_2\in H$, respectively.

In this paragraph, we show that $Y_{x_1}$ and $Y_{x_2}$ are canonically isomorphic.
Note that there exists a canonical isomorphism $\mathcal{R}\times_{p_2,M}\mathrm{Spec}(\kappa(x_1))\cong \mathcal{R}\times_{p_2,M}\mathrm{Spec}(\kappa(x_2))$.
Using this, we may identify the canonical morphism $\mathbf{Pic}^0_{\mathcal{X}_{1}}\to \mathcal{R}\times_{p_2,M}\mathrm{Spec}(\kappa(x_1))$ with the morphism $\mathbf{Pic}^0_{\mathcal{X}_{1}}\to \mathcal{R}\times_{p_2,M}\mathrm{Spec}(\kappa(x_2))$.
Because $Y_{x_1}$ and $Y_{x_2}$ are characterized by the Stein factorizations of $\mathbf{Pic}^0_{\mathcal{X}_{1}}\to \mathcal{R}\times_{p_2,M}\mathrm{Spec}(\kappa(x_1))$ and $\mathbf{Pic}^0_{\mathcal{X}_{1}}\to \mathcal{R}\times_{p_2,M}\mathrm{Spec}(\kappa(x_2))$, respectively, we obtain the claim.

To show the assertion of this lemma, it suffices to show that $f^*$ induces $\widehat{\widehat{Y}}_{x_1}\cong\widehat{\widehat{Y}}_{x_2}$.
Note that $\widehat{\widehat{Y}}_{x_1}$ and $\widehat{\widehat{Y}}_{x_2}$ are the images of $\Lambda(L_{\widehat{X}})\circ\iota_{\widehat{Y}}|_{\widehat{Y}_{x_1}}$ and $\Lambda(L_{\widehat{X}})\circ\iota_{\widehat{Y}}|_{\widehat{Y}_{x_2}}$, respectively. 
By the canonical isomorphism $Y_{x_1}\cong Y_{x_2}$, it is easy to see that $\widehat{Y}_{x_1}=\widehat{Y}_{x_2}$ as subvarieties of $\mathbf{Pic}^0_{\mathbf{Pic}^0_{\mathcal{X}_1}}=\widehat{X}_{x_1}=\widehat{X}_{x_2}$.
Thus, it suffices to show that $\Lambda(L_{\widehat{X}}|_{\widehat{X}_{x_1}})=\Lambda(L_{\widehat{X}}|_{\widehat{X}_{x_2}})$.
By the construction of $h'$, we see that $\widehat{X}\times_{H}C'\cong \widehat{X}_{x_1}\times C'$.
Therefore $L_{\widehat{X}}|_{\widehat{X}_{x_1}}$ and $L_{\widehat{X}}|_{\widehat{X}_{x_2}}$ are algebraically equivalent.
By \cite[Proposition 6.1]{GIT}, we see that $\Lambda(L_{\widehat{X}}|_{\widehat{X}_{x_1}})=\Lambda(L_{\widehat{X}}|_{\widehat{X}_{x_2}})$.
We complete the proof.
\end{proof}

We define the following prestack $\mathcal{N}^{\mathrm{pre}}$ as follows.

\begin{de}\label{de--stack--N}
For any scheme $T$, we set the collection of objects $\mathcal{N}^{\mathrm{pre}}(T)$ as the collection of $\mathcal{Y}\times_{\mathcal{M},g}T$ for some morphism $g\colon T\to\mathcal{M}$ that corresponds to a family of polarized varieties $(\mathcal{X},\mathcal{L})\to T$.
Note that there exists a natural morphism $\Phi_g\colon \mathcal{Y}\times_{\mathcal{M},g}T\to\mathcal{M}$ such that
for any $T$-scheme $S$, $\mathcal{Y}\times_{\mathcal{M},g}T(S)\to\mathcal{M}(S)$ maps $\mathcal{L}'\in\mathcal{Y}\times_{\mathcal{M},g}T(S)$ to $(\mathcal{X}_S,\mathcal{L}_S\otimes\mathcal{L}')$.
We also set the collection of isomorphisms $\mathbf{Isom}^{\mathcal{N}^{\mathrm{pre}}}_T(\mathcal{Y}\times_{\mathcal{M},g_1}T,\mathcal{Y}\times_{\mathcal{M},g_2}T)$ of two objects $\mathcal{Y}\times_{\mathcal{M},g_1}T$ and $\mathcal{Y}\times_{\mathcal{M},g_2}T$ for some $g_1,g_2\colon T\to\mathcal{M}$ as follows.
Suppose that $g_i$ corresponds to $(\mathcal{X}_i,\mathcal{L}_i)\to T$ for $i=1,2$.
A pair of isomorphisms $\sigma\colon\mathcal{Y}\times_{\mathcal{M},g_1}T\to \mathcal{Y}\times_{\mathcal{M},g_2}T $ over $\mathcal{M}\times T$ and $f\colon \mathcal{X}_1\to\mathcal{X}_2$ over $T$ belongs to $\mathbf{Isom}^{\mathcal{N}^{\mathrm{pre}}}_T(\mathcal{Y}\times_{\mathcal{M},g_1}T,\mathcal{Y}\times_{\mathcal{M},g_2}T)$ if the following hold.

\begin{itemize}
\item Using $f^*\colon \mathcal{Y}\times_{\mathcal{M},g_2}T\to \mathcal{Y}\times_{\mathcal{M},g_1}T$, which is the restriction of $f^*$ as in Lemma \ref{lem--y--identify}, we can identify $\mathcal{Y}\times_{\mathcal{M},g_2}T$ with a trivial principal $\mathcal{Y}\times_{\mathcal{M},g_1}T$-fiber bundle by Lemma \ref{lem--y--identify}.
Then $\sigma$ is an isomorphism of principal $\mathcal{Y}\times_{\mathcal{M},g_1}T$-fiber bundles. 
    \item Let $\iota\colon T\to \mathcal{Y}\times_{\mathcal{M},g_1}T$ be the zero section. 
     Then $(\mathcal{X}_2,\mathcal{L}_2)$ is the family corresponding to the composition of $\sigma\circ\iota$ and the natural morphism $\Phi_{g_2}\colon\mathcal{Y}\times_{\mathcal{M},g_2}T\to\mathcal{M}$.
\end{itemize}
Note that a composition of two morphisms is defined.
Then $\mathcal{N}^{\mathrm{pre}}$ is well defined as a prestack.
\end{de}

We set $\mathcal{N}$ as the stackification of $\mathcal{N}^{\mathrm{pre}}$ (cf.~\cite[Theorem 4.6.5]{Ols}) with respect to the big \'etale topology and let $\xi\colon \mathcal{M}\to\mathcal{N}$ be the canonical forgetful morphism of stacks.

\begin{prop}\label{prop--representability--quotient--morph}
    $\mathcal{N}$ is an Artin stack of finite type over $\mathbbm{k}$ such that $\mathbf{Isom}^{\mathcal{N}}_T(u_1,u_2)$ is a separated algebraic space over $T$ for any $u_1,u_2\colon T\to \mathcal{N}$, where $T$ is an arbitrary scheme, and $\xi$ is representable by algebraic spaces, proper, surjective, and smooth.
\end{prop}

\begin{proof}
Let $u_1,u_2\colon T\to \mathcal{N}$ be morphisms from a scheme.
Replacing $T$ with its \'etale cover if necessary, we may assume that $u_1$ and $u_2$ correspond to $\mathcal{Y}\times_{\mathcal{M},g_1}T$ and $\mathcal{Y}\times_{\mathcal{M},g_2}T$ for some $g_1,g_2\colon T\to\mathcal{M}$.
By definition, it is easy to check that $\mathbf{Isom}^{\mathcal{N}}_T(u_1,u_2)$ is isomorphic to a closed subscheme of the scheme $\mathbf{Isom}_T(\mathcal{X}_1,\mathcal{X}_2)\times_T (\mathcal{Y}\times_{\mathcal{M},g_1}T)$ separated over $T$.
This and \cite[Tag 0BGQ]{Stacks} show that $\mathbf{Isom}^{\mathcal{N}}_T(u_1,u_2)$ is an algebraic space separated over $T$ in general.

Therefore, it suffices to deal with the latter assertion.
Consider a morphism $g\colon S\to\mathcal{N}$ from a scheme.
Here, we can replace $S$ with its \'etale cover freely.
By the definition of $\mathcal{N}$, we may assume that there exists a morphism $h\colon S\to\mathcal{M}$ such that $g=\xi\circ h$ by replacing $S$ with its \'etale cover.
Then it is easy to see that $\mathcal{M}\times_{\mathcal{N}}S\cong \mathcal{Y}\times_{\mathcal{M}}S$.
Since $\mathcal{Y}\times_{\mathcal{M}}S\to S$ is projective, smooth and surjective, we obtain the proof. 
\end{proof}

\begin{thm}\label{thm:N--DMstack}
    $\mathcal{N}$ is a separated Deligne-Mumford stack of finite type over $\mathbbm{k}$.
\end{thm}

\begin{proof}
By the definition of $\mathcal{Y}$, it is easy to see that each stabilizer of $\mathcal{N}$ is a finite group. This and Proposition \ref{prop--representability--quotient--morph} show that $\mathcal{N}$ is a Deligne-Mumford stack (cf.~\cite[Remark 8.3.4]{Ols}).
To see the conclusion, it suffices to show that $\mathcal{N}$ is separated.
Note that the diagonal morphism of $\mathcal{N}$ is separated by Proposition \ref{prop--representability--quotient--morph}.
Take a discrete valuation ring $R$ with the fractional field $K$.
Let $f_1,f_2\colon \mathrm{Spec}\,R\to \mathcal{N}$ be morphisms such that $f_1|_{\mathrm{Spec}\,K}=f_2|_{\mathrm{Spec}\,K}$.
Replacing $R$ with its finite cover $R'$ with the fractional field $K'$, we may assume that there exist $g_1\colon \mathrm{Spec}\,R'\to \mathcal{M}$ and $g_2\colon \mathrm{Spec}\,R'\to \mathcal{M}$ such that $f_1\circ\pi_{R'}\cong \xi\circ g_1$ and $f_2\circ \pi_{R'}\cong\xi\circ g_2$, where $\pi_{R'}\colon \mathrm{Spec}\,R'\to\mathrm{Spec}\,R$ is the canonical morphism by Proposition \ref{prop--representability--quotient--morph}.
Suppose that $(\mathcal{X}_1,\mathcal{L}_1)$ and $(\mathcal{X}_2,\mathcal{L}_2)$ are the families corresponding to $g_1$ and $g_2$.
By the assumption that $f_1|_{\mathrm{Spec}\,K}=f_2|_{\mathrm{Spec}\,K}$, we may identify $(\mathcal{X}_1)_{K'}$ with $(\mathcal{X}_2)_{K'}$ and $\mathcal{L}_2\otimes \mathcal{L}_1^{-1}|_{(\mathcal{X}_2)_{K'}}\in \mathcal{Y}(K')$.
Since $\mathcal{Y}\to\mathcal{M}$ is proper, $\mathcal{L}_2\otimes \mathcal{L}_1^{-1}|_{(\mathcal{X}_2)_{K'}}$ can be extended to some $\mathcal{L}_3 \in \mathcal{Y}(R')$.
If we replace $g_2$ by the morphism corresponding to $(\mathcal{X}_2,\mathcal{L}_2\otimes\mathcal{L}_3^{\otimes-1})$, we may assume that $g_1|_{\mathrm{Spec}\,K'}=g_2|_{\mathrm{Spec}\,K'}$. 
By the separatedness of $\mathcal{M}$, we see that $g_1=g_2$.
Therefore, $f_1\circ\pi_{R'}=f_2\circ\pi_{R'}$.
By \cite[Theorem 11.5.1]{Ols}, this shows the properness of the diagonal morphism of $\mathcal{N}$.
We conclude that $\mathcal{N}$ is separated.
\end{proof}

\begin{cor}\label{cor--N--closedpoint}
    Let $N$ be the coarse moduli space of $\mathcal{N}$ and let $\Omega$ be an algebraically closed field over $\mathbbm{k}$.
    Then, $N(\Omega)$ is bijective to the set
    \[
    \{(X,L)\in M(\Omega)\}/\equiv.
    \]
    Here, the equivalence relation $\equiv$ is defined in the way that $(X,A)\equiv(Y,B)$ if and only if there exists an isomorphism $\varphi\colon X\to Y$ such that $\varphi^*B-A$ is numerically trivial.
\end{cor}

\begin{proof}
 We note that $N$ actually exists by \cite{KeM} and Theorem \ref{thm:N--DMstack}.
    Then the assertion follows from the definition of $\mathcal{N}$.
\end{proof}

\begin{de}\label{de--Num--consruction}
Fix $d\in\mathbb{Z}_{>0}$, $v\in\mathbb{Q}_{>0}$, $u\in\mathbb{Q}_{\ne0}$ and take $w\in\mathbb{Q}_{>0}$ such that $\mathfrak{Z}_{d, v,u,w}\to \mathfrak{Z}_{d, v,u}$ is surjective.
For any connected component $\mathcal{M}^{\mathrm{im}}\subset \mathcal{M}^{W,I,\mathrm{im}}_{d,v,u,r,w}$, we can construct stacks $\mathcal{M}$ and $\mathcal{N}$ as in Definition \ref{de--stack--N}.  
Let $\mathcal{M}^{W,I}_{d,v,u,r,w}$ and $\mathcal{N}^{W,I}_{d,v,u,r,w}$ be the disjoint union of such $\mathcal{M}$ and $\mathcal{N}$, respectively.
By construction $\mathcal{M}^{W,I}_{d,v,u,r,w}$ contains $(\mathcal{M}^{W,I}_{d,v,u,w})_{\mathrm{red}}$ as the reduced structure.

Note that there exists a natural morphism $\xi\colon \mathcal{M}^{W,I}_{d,v,u,r,w}\to\mathcal{N}^{W,I}_{d,v,u,r,w}$ that is constructed by patching the forgetful morphisms.
We call this the {\it quotient morphism}.
Note that $\xi$ is representable by algebraic spaces, proper, surjective, and smooth by Proposition \ref{prop--representability--quotient--morph}.

Let $\overline{\xi}\colon M^{W,I}_{d,v,u,r,w}\to N^{W,I}_{d,v,u,r,w}$ be the morphism between the coarse moduli spaces induced by $\xi$.

Let $\Lambda_{\mathrm{CM},t}$ be the CM line bundle on $M_{d,I^{d-1}\cdot v,u,r,I^d\cdot w}$, as Definition \ref{de--CM--level}.
Let $\Lambda_{\mathrm{CM},t}$ also denote the restriction of $\Lambda_{\mathrm{CM},t}$ to $M^{W,I}_{d,v,u,r,w}$.
\end{de}
Then the following holds for the case where $u>0$.
\begin{thm}\label{thm--rel--ample}
In Definition \ref{de--Num--consruction}, suppose further that $u>0$.
  Then there exists $t_0\in\mathbb{Q}_{>0}$ such that $\Lambda_{\mathrm{CM},\mu}$ is $\overline{\xi}$-ample in the sense of \cite[Tag 0D30]{Stacks} for any $\mu>t_0$. 
\end{thm}

\begin{proof}
Choose $t_0$ as Corollary \ref{cor--partial--positivity--of--cm} for $M_{d,I^{d-1}\cdot v,u,r,I^d\cdot w}$.
    By \cite[Tags 0D36 and 0D3A]{Stacks}, it suffices to check for any closed point $p\in N^{W,I}_{d,v,u,r,w}$, $\Lambda_{\mathrm{CM},\mu}|_{\overline{\xi}^{-1}(p)}$ is ample.
    Since $\overline{\xi}^{-1}(p)\subset M_{d,I^{d-1}\cdot v,u,r,I^d\cdot w}$ is a proper closed subspace by Proposition \ref{prop--representability--quotient--morph}, we complete the proof by Corollary \ref{cor--partial--positivity--of--cm}.
\end{proof}
As we saw in Example \ref{ex--cm}, $\overline{\xi}$ is not finite in general in the case where $u>0$.
Finally, we note the following phenomenon in the case where $u<0$.

\begin{prop}\label{prop--unnec--u<0}
   In Definition \ref{de--Num--consruction}, if $u<0$, then $\xi$ is finite.
   In particular, $\Lambda_{\mathrm{CM},t}$ is $\overline{\xi}$-ample for any $t\in\mathbb{Q}_{>0}$. 
\end{prop}

\begin{proof}
    It follows from Proposition \ref{prop--representability--quotient--morph} and \cite[Proposition 6.3]{Hat23} by construction of $\xi$.
\end{proof}

\section{Projectivity of moduli spaces of stable quasimaps}\label{sec--quasimap}

Let $\iota\colon X\hookrightarrow \mathbb{P}^N$ be a closed immersion of a proper normal variety.
We define the affine cone $\mathrm{Cone}(X)\subset \mathbb{A}^{N+1}$ of $X$.

\begin{de}[Quasimaps]\label{de--qmaps}
Let $C$ be a proper connected curve with only nodal singularities.
We say that $f\colon C\to [\mathrm{Cone}(X)/\mathbb{G}_m]$ is a {\it quasimap} if the generic point of any irreducible component is mapped to $X\subset [\mathrm{Cone}(X)/\mathbb{G}_m]$.
If $f(C)\cap X$ is a closed point of $X$, then we say that $f$ is {\it constant}.

Let $S$ be a scheme over $\mathbbm{k}$ and $\pi\colon \mathcal{C}\to S$ a flat proper morphism whose geometric fibers are connected and have only nodal singularities.
We say that $q\colon \mathcal{C}\to [\mathrm{Cone}(X)\times S/\mathbb{G}_{m,S}]$ is a {\it family of quasimaps} if $q_{\bar{s}}$ is a quasimap for any geometric point $\bar{s}$.
We note that there exists the {\it line bundle} $\mathscr{L}$ on $\mathcal{C}$ {\it associated with $q$} defined as follows:
By definition, $q$ is associated to a principal $\mathbb{G}_m$-bundle $\mathcal{P}$ over $\mathcal{C}$ and a $\mathbb{G}_m$-equivariant morphism $p\colon \mathcal{P}\to\mathrm{Cone}(X)$.
Then, we set $\mathscr{L}$ such that there exists an isomorphism $\mathbb{A}_{\mathcal{C}}(\mathscr{L})\setminus\mathbf{0}=\mathcal{P}$ as principal $\mathbb{G}_m$-bundles, where $\mathbf{0}$ is the zero section.
Let $\mathcal{O}_{\mathcal{C}}^{\oplus N+1}\to \mathscr{L}$ be the canonical morphism induced by $q$ and we call the image of the morphism $\pi_*\mathcal{O}_{\mathcal{C}}^{\oplus N+1}\to\pi_* \mathscr{L}$ induced by the above morphism {\it the module of sections} denoted by $\mathrm{Sec}_q(\mathscr{L})$.
We also define $|\mathscr{L}|_q:=(\mathrm{Sec}_q(\mathscr{L})\setminus\{0\})/\mathbbm{k}^\times$.
For any nonzero section $f\in \mathrm{Sec}_q(\mathscr{L})$, we set $\mathrm{div}_{\mathscr{L}}(f)\in |\mathscr{L}|_q$ as the closed subscheme defined by $f$.
Furthermore, we fix a canonical coordinate of $\mathbb{A}^{N+1}$ and let $e_0,\ldots,e_N$ be the canonical coordinates.
If $S$ is the spectrum of a discrete valuation ring or a field $R$ and $f\in \mathrm{Sec}_q(\mathscr{L})$ is not zero at any generic point of fibers, we note that $\mathrm{div}_{\mathscr{L}}(f)$ is a Cartier divisor.
In this case, we set the {\it fixed part} of $q$ as the largest effective divisor $\mathrm{Bs}(q)$ contained in $\mathrm{div}_{\mathscr{L}}(f)$ for any $f\in \mathrm{Sec}_q(\mathscr{L})$ nonzero at any generic point.
We will write $\mathrm{div}_{\mathscr{L}}(f)$ as $\mathrm{div}(f)$ if there is no fear of confusion.

 We set an isomorphism $\phi$ from $q_1$ to $q_2$ for any two families of quasimaps $q_1\colon \mathcal{C}_1\to [\mathrm{Cone}(X)_S/\mathbb{G}_{m,S}]$ and $q_2\colon \mathcal{C}_2\to [\mathrm{Cone}(X)_S/\mathbb{G}_{m,S}]$, as an isomorphism $\phi\colon \mathcal{C}_1\to \mathcal{C}_2$ such that $q_2\circ\phi$ and $q_1$ are isomorphic.

Let $w\in\mathbb{Q}\cap (0,1]$ and $m\in\mathbb{Q}_{>0}$.
We say that a family of quasimaps $q$ is a {\it family of stable quasimaps of weight $w$ and of degree $m$} if $\mathrm{deg}\mathscr{L}_{\bar{s}}=m$ and there exists a divisor $D\in |\mathscr{L}_{\bar{s}}|_{q_{\bar{s}}}$ such that $(\mathcal{C}_{\bar{s}},wD)$ is slc and $K_{\mathcal{C}_{\bar{s}}}+wD$ is ample for any geometric point $\bar{s}\in S$. 
If $S=\mathrm{Spec}\,\mathbbm{k}$, then we say that $q$ is a {\it stable quasimap} of weight $w$.
\end{de}

\begin{rem}\label{rem--toda}
Consider the case where $X=\mathbb{P}^N$.
We note that our stable quasimap $q$ to $X$ of weight $\epsilon$ and degree $d$ is nothing but the $\epsilon$-stable quotient of type $(N-1,N,d)$ with the target $\mathbb{P}^N$, which was introduced by Toda \cite[Definition 2.3]{Toda}.
Recall that a family of stable quotients $\pi\colon \mathcal{C}\to B$ of type $(N,N+1,d)$ is defined to be a projective locally stable family $\pi$ and a short exact sequence 
\[
0\to \mathcal{S}\to\mathcal{O}_{\mathcal{C}}^{\oplus N+1}\to \mathcal{Q}\to0, 
\]
where $\mathcal{Q}$ is flat over $B$, $\mathcal{S}$ is of rank one, $(\mathcal{C}_{\bar{s}},\epsilon \tau(\mathcal{Q}_{\bar{s}}))$ is slc and $K_{\mathcal{C}_{\bar{s}}}+\epsilon \mathrm{det}(\mathcal{Q}_{\bar{s}})$ is ample for any geometric point $\bar{s}\in B$.
Here, $\tau(\mathcal{Q}_{\bar{s}})$ denotes the torsion part of $\mathcal{Q}_{\bar{s}}$.
Note that $\mathcal{Q}_{\bar{s}}$ is locally free around the singular locus of $\mathcal{C}_{\bar{s}}$.
Considering the dual morphism $\mathcal{O}_{\mathcal{C}}^{\oplus N+1}\to \mathscr{L}$, where $\mathscr{L}^\vee=\mathcal{S}$, we obtain the information of a family of stable quasimaps. 
Conversely, suppose that a structure of a family $q$ of stable quasimaps of weight $\epsilon$ and degree $d$ is given to $\pi$. Let $\mathscr{L}$ be the associated line bundle and $\mathcal{S}$ the dual of $\mathscr{L}$.
Taking the dual of the induced morphism $\mathcal{O}_{\mathcal{C}}^{\oplus N+1}\to \mathscr{L}$, we obtain a short exact sequence
\[
0\to \mathcal{S}\to\mathcal{O}_{\mathcal{C}}^{\oplus N+1}\to \mathcal{Q}\to0, 
\]
where $\mathcal{Q}$ is the cokernel.
Since $\mathcal{S}_b\to \mathcal{O}_{\mathcal{C}_b}^{\oplus N+1}$ is injective for any point $b\in B$, $\mathcal{Q}$ is flat over $B$ by \cite[(20.E)]{Mat}.
Therefore, the short exact sequence gives a structure of a family of stable quotients (cf.~\cite[Remark 2.6]{Toda}).
Therefore, if $\epsilon\le\frac{1}{d}$, then our stable quasimap to $\mathbb{P}^N$ of weight $\epsilon$ and degree $d$ coincides with stable quotients by \cite[Theorem 1.1]{Toda}, which was introduced by \cite{MOP}.
\end{rem}

\begin{de}[K-stability of quasimaps]\label{de--K-st--qmaps}
Let $q\colon C\to [\mathrm{Cone}(X)/\mathbb{G}_m]$ be a quasimap of degree $m$.
Let $w\in(0,1]$ and suppose that $\frac{|\mu|}{\mu}(K_C+w\mathscr{L})$ is ample, where $\mu:=2p_a(C)-2+wm\ne0$. 
Take a semiample test configuration $(\mathcal{C},\mathcal{L})$ for $(C,\frac{|\mu|}{\mu}(K_C+w\mathscr{L}))$, where $\mathscr{L}$ is the line bundle associated with $q$.
Then we set
\[
\mathrm{DF}_q(\mathcal{C},\mathcal{L}):=\mathrm{DF}_{w\mathrm{div}_\mathscr{L}(f)}(\widetilde{\mathcal{C}},\nu^*\mathcal{L}),
\]
for any general $f\in\mathrm{Sec}_q(\mathscr{L})$, where $\nu\colon \widetilde{\mathcal{C}}\to \mathcal{C}$ is the normalization.

We say that $q$ is {\it K-stable} if $\mathrm{DF}_q(\mathcal{C},\mathcal{L})>0$ for any nontrivial deminormal semiample test configuration and $q$ is {\it K-semistable} if $\mathrm{DF}_q(\mathcal{C},\mathcal{L})\ge0$ for any semiample test configuration.
\end{de}

As in the following proposition, stable quasimaps are indeed K-stable.

\begin{prop}\label{prop--K-stable-quasimaps--stability}
In Definition \ref{de--K-st--qmaps}, suppose that $\mu>0$. Then the following are equivalent.
\begin{enumerate}
    \item $q$ is K-stable,
    \item $q$ is K-semistable, 
    \item $q$ is a stable quasimap in the sense of Definition \ref{de--qmaps}.
\end{enumerate}
\end{prop}

\begin{proof}
(1)$\Rightarrow$(2) follows from \cite[Proposition 3.8]{Odaka} and \cite[Proposition 4]{Fjtop}.

For (2)$\Rightarrow$(3), assume that there exists a K-semistable quasimap $q$ that is not stable.
Then $(C,w\mathrm{div}(f))$ is not slc for any $f\in\mathrm{Sec}_q(\mathscr{L})$.
It is easy to check that there exists a point $p\in \mathrm{Bs}(q)$ such that $(C,w\mathrm{div}(f))$ is not slc around $p$ for any $f\in\mathrm{Sec}_q(\mathscr{L})$.
Then, considering a test configuration $\mathrm{Bl}_{p\times\{0\}}(C\times\mathbb{A}^1)$ that is the blow up at $p\times\{0\}$, we can see that $q$ is not K-semistable by the proof of \cite[Theorem 9.1]{BHJ}.
Therefore, we obtain (2)$\Rightarrow$(3).

(3)$\Rightarrow$(1) follows from \cite[Corollary 9.3]{BHJ} and \cite[Proposition 6]{Fjtop}.
\end{proof}

We summarize the results from \cite{HH2} on K-semistable quasimaps with $\mu<0$ in Definition \ref{de--K-st--qmaps}, which we call {\it log Fano K-semistable quasimaps}.

\begin{de}[Moduli stack of K-semistable log Fano quasimaps]\label{de--k-moduli--log--fano--qmaps}
Fix $m\in\mathbb{Z}_{>0}$, $w\in\mathbb{Q}\cap(0,1)$ and $v\in\mathbb{Q}_{>0}$. 
Let $\iota\colon X\hookrightarrow\mathbb{P}^N$ be a closed immersion of a normal variety.
Let $\mathcal{M}^{\mathrm{Kss,qm}}_{m,w,-v,\iota}$ be a stack such that for any scheme $S$, the collection of objects $\mathcal{M}^{\mathrm{Kss,qm}}_{m,w,-v,\iota}(S)$ is 
$$\left\{
 \vcenter{
 \xymatrix@C=12pt{
\mathcal{C}\ar[rr]^-{q}\ar[dr]_{\pi}&& [\mathrm{Cone}(X)_S/\mathbb{G}_{m,S}] \ar[dl]\\
&S
}
}
\;\middle|
\begin{array}{rl}
(i)&\text{$q$ is a family of K-semistable log Fano}\\
&\text{quasimaps of degree $m$ and weight $w$,}\\
(ii)&\text{$\mathrm{deg}_{\mathcal{C}_{\bar{s}}}(K_{\mathcal{C}_{\bar{s}}})+wm=-v$}\\
&\text{for any geometric point $\bar{s}\in S$}
\end{array}\right\},$$
and we set its arrows to be isomorphisms of families of quasimaps.
\end{de}

\begin{thm}[{\cite[Theorem 3.45]{HH2}}]\label{thm--logfanoqmaps}
 Notations as in Definition \ref{de--k-moduli--log--fano--qmaps}.
 $\mathcal{M}^{\mathrm{Kss,qm}}_{m,w,-v,\iota}$ is an Artin stack of finite type with a projective good moduli space $M^{\mathrm{Kps,qm}}_{m,w,-v,\iota}$.   
 Furthermore, $M^{\mathrm{Kps,qm}}_{m,w,-v,\iota}$ admits an ample $\mathbb{Q}$-line bundle $\Lambda^{\mathrm{qmaps}}_{\mathrm{CM}}$, which is defined as \cite[Definition 3.34 and Lemma 3.36]{HH2}.
\end{thm}
For the definition of good moduli space, see \cite{alper,AHLH} or \cite[Section 8]{Xu} for example.
From now on, we concentrate on stable quasimaps.

\begin{de}[Moduli stack of stable quasimaps]\label{de--moduli--stablequot}
Fix $m\in\mathbb{Z}_{>0}$, $w\in\mathbb{Q}\cap(0,1]$ and $v\in\mathbb{Q}_{>0}$.
Let $\iota\colon X\hookrightarrow\mathbb{P}^N$ be a closed immersion of a normal variety.
Let $\mathcal{M}^{\mathrm{Kss,qm}}_{m,w,v,\iota}$ be a stack such that for any scheme $S$, the collection of objects $\mathcal{M}^{\mathrm{Kss,qm}}_{m,w,v,\iota}(S)$ is 
$$\left\{
 \vcenter{
 \xymatrix@C=12pt{
\mathcal{C}\ar[rr]^-{q}\ar[dr]_{\pi}&& [\mathrm{Cone}(X)_S/\mathbb{G}_{m,S}] \ar[dl]\\
&S
}
}
\;\middle|
\begin{array}{rl}
(i)&\text{$q$ is a family of stable quasimaps}\\
&\text{of degree $m$ and weight $w$,}\\
(ii)&\text{$\mathrm{deg}_{\mathcal{C}_{\bar{s}}}(K_{\mathcal{C}_{\bar{s}}})+wm=v$}\\
&\text{for any geometric point $\bar{s}\in S$}
\end{array}\right\},$$
and we set its arrows to be isomorphisms of families of quasimaps.
\end{de}

\begin{prop}\label{lem--quasimap--stack}
 Notations as in Definition \ref{de--moduli--stablequot}.
 $\mathcal{M}^{\mathrm{Kss,qm}}_{m,w,v,\iota}$ is a proper Deligne--Mumford stack. In particular, there exists a coarse moduli space $M^{\mathrm{Kps,qm}}_{m,w,v,\iota}$.   
\end{prop}

\begin{proof}
It is well known that there exists a positive integer $l_0\in\mathbb{Z}_{>0}$ such that $l_0(K_C+w\mathscr{L})$ is a very ample Cartier divisor such that $H^1(C,\mathcal{O}_C(ll_0(K_C+w\mathscr{L})))=0$ for any $l\in\mathbb{Z}_{>0}$ and object $(q\colon C\to [\mathrm{Cone}(X)/\mathbb{G}_m])\in \mathcal{M}^{\mathrm{Kss,qm}}_{m,w,v,\mathrm{id}_{\mathbb{P}^N}}(\mathbbm{k})$ (cf.~\cite[Theorem 1.1]{hmx-boundgentype}).
Here, $\mathscr{L}$ is the line bundle induced by $q$.
Furthermore, there are only finitely many possibilities of the Hilbert polynomial $H^0(C,\mathcal{O}_C(ll_0(K_C+w\mathscr{L})))$.
By \cite[Theorem 1.1]{hmx-boundgentype}, we also see that $\mathscr{L}$ is bounded, that is, there exists $l_1\in\mathbb{Z}_{>0}$ such that $\mathscr{L}$ is $l_1$-regular with respect to $l_0(K_C+w\mathscr{L})$ for any object $(q\colon C\to [\mathrm{Cone}(X)/\mathbb{G}_m])\in \mathcal{M}^{\mathrm{Kss,qm}}_{m,r,u,v,\mathrm{id}_{\mathbb{P}^N}}(\mathbbm{k})$ and there are only finitely many possibilities of Hilbert polynomials $P(l)=H^0(C,\mathcal{O}_C(ll_0(K_C+w\mathscr{L})))$ and $Q(l)=\chi(C,\mathscr{L}\otimes\mathcal{O}_C(ll_0(K_C+w\mathscr{L})))$.  
To show the assertion, it suffices to show that the following closed and open substack $\mathcal{M}^{\mathrm{Kss,qm}}_{P,Q}$ is an Artin stack of finite type:
$$\mathcal{M}^{\mathrm{Kss,qm}}_{P,Q}(S):=\left\{
 q 
\;\middle|
\begin{array}{rl}
&\text{$P(l)=H^0(\mathcal{C}_{\bar{s}},\mathcal{O}_{\mathcal{C}_{\bar{s}}}(ll_0(K_{\mathcal{C}_{\bar{s}}}+w\mathscr{L}^q_{\bar{s}})))$ and}\\
&\text{$Q(l)=\chi(\mathcal{C}_{\bar{s}},\mathscr{L}^q_{\bar{s}}\otimes\mathcal{O}_{\mathcal{C}_{\bar{s}}}(ll_0(K_{\mathcal{C}_{\bar{s}}}+w\mathscr{L}^q_{\bar{s}})))$,}\\
&\text{for any geometric point $\bar{s}\in S$, where}\\
&\text{$\mathscr{L}^q$ is the induced line bundle by $q$.}
\end{array}\right\}.$$

Consider the Hilbert scheme $H$ of closed subschemes of $\mathbb{P}^{P(1)-1}$ with the Hilbert polynomial $P$.
Let $\pi_{\mathcal{U}}\colon\mathcal{U}\to H$ be the universal closed subscheme of $\mathbb{P}^{P(1)-1}_H$ and $\mathcal{A}:=\mathcal{O}_{\mathbb{P}^{P(1)-1}_H}(1)|_{\mathcal{U}}$.
By \cite[Proposition 13.2.15]{Ols}, take an open subscheme $Z_1\subset H$ such that 
\[
Z_1:=\{s\in H| \text{ $\mathcal{U}_{\bar{s}}$ has at worst nodal singularities}\}.
\]
Next, consider the following open subscheme $Z_2\subset \mathbf{Quot}_{\mathcal{U}_{Z_1}/Z_1}^{\mathcal{O}(-l_1\mathcal{A})^{Q(l_1)-1}, Q}$ such that if we let $\mathscr{F}$ be the universal quotient sheaf, then
\[
Z_2:=\{s\in \mathbf{Quot}_{\mathcal{U}_{Z_1}/Z_1}^{\mathcal{O}(-l_1\mathcal{A})^{Q(l_1)-1}, Q}| \mathscr{F}_{s} \text{ is an invertible sheaf}\}.
\]
By the same argument in \cite[Section 3.2]{HH2}, we take a locally closed subscheme $Z_3\subset Z_2$ such that for any morphism $g\colon T\to Z_2$ from a scheme, $g$ factors through $Z_3$ if and only if $g$ satisfies the following:
\begin{itemize}
    \item $\omega_{\mathcal{U}_T/T}^{\otimes l_0}\otimes \mathscr{F}^{\otimes l_0w}_T\sim_{T}\mathcal{A}$, and 
    \item $\pi_{\mathcal{U},T*}\mathcal{O}_{\mathcal{U}_T}(\mathcal{A}_T)\to H^0(\mathcal{A}_{t})$ is surjective for any $t\in T$.
\end{itemize} 
Next, we consider $\mathcal{H}:=\mathbf{Hom}_{Z_{3}}(\mathcal{O}_{\mathcal{U}_{Z_{3}}}^{N+1},\mathscr{F}_{Z_3})$.
Let $g\in\mathfrak{Hom}(\mathcal{O}_{\mathcal{U}_{\mathcal{H}}}^{N+1},\mathscr{F}_{\mathcal{H}})(\mathcal{H})$ be the universal homomorphism.
We note that $g$ corresponds to a global section $(f_0,\ldots,f_N)\in H^0(\mathcal{H},\pi_{\mathcal{U}_{\mathcal{H}}*}\mathscr{F}_\mathcal{H})^{\oplus N+1}$.
Let $Z_4\subset H$ be an open subscheme such that for any $s\in \mathcal{H}$, $s$ is contained in $Z_4$ if and only if $(\mathcal{U}_{\bar{s}},w\mathrm{div}(f_{\bar{s}}))$ is slc for some $(a_i)_{i=0}^N\in\overline{\kappa(s)}$, where $f_{\bar{s}}=\sum_{i=0}^N a_if_{i,\bar{s}}$.
Furthermore, as in the discussion of \cite[Section 3.2]{HH2}, there exists a closed subscheme $Z_5\subset Z_4$ such that for any morphism $\mu\colon T\to Z_4$, $\mu$ factors through $Z_5$ if and only if $g_{\mathcal{U}_T}$ factors through $\mathrm{Cone}(X)\times T\subset \mathbb{A}^{N+1}_{T}$.
Note that $Z_4$ is a scheme quasi-projective over $\mathbbm{k}$.

By the argument of the proof of \cite[Theorem 3.23]{HH2}, we see that $\mathcal{M}^{\mathrm{Kss,qm}}_{P,Q}\cong [Z_5/PGL(P(1))\times PGL(Q(l_1))]$.
Thus, $\mathcal{M}^{\mathrm{Kss,qm}}_{m,w,v,\iota}$ is an Artin stack of finite type and a closed substack of $\mathcal{M}^{\mathrm{Kss,qm}}_{m,w,v,\mathrm{id}_{\mathbb{P}}^N}$, which has $[Z_4/PGL(P(1))\times PGL(Q(l_1))]$ as a connected component.
$\mathcal{M}^{\mathrm{Kss,qm}}_{m,w,v,\mathrm{id}_{\mathbb{P}}^N}$ is a proper Deligne--Mumford stack by \cite[Theorem 2.12]{Toda} (cf.~Remark \ref{rem--toda}) and hence so is $\mathcal{M}^{\mathrm{Kss,qm}}_{m,w,v,\iota}$.
The last assertion follows from \cite{KeM}.
\end{proof}

\begin{de}
    Let $q\colon \mathcal{C}\to [\mathrm{Cone}(X)_S/\mathbb{G}_{m,S}]$ be a family of stable quasimaps of degree $m$ and weight $w$.
    Then, we set the CM line bundle $\lambda_{\mathrm{CM},q}$ as follows:
    Let $\pi\colon \mathcal{C}\to S$ be the canonical morphism and $\mathscr{L}$ the induced line bundle.
    Take a sufficiently large and divisible integer $M>0$ such that $M(K_{\mathcal{C}/S}+w\mathscr{L})$ is a $\pi$-very ample line bundle. By \cite[Theorem 4]{KnMu}, there exist unique line bundles $\lambda_0$, $\lambda_1$, and $\lambda_2$ such that 
    \[
    \mathrm{det}\pi_*\mathcal{O}_Y(m(M(K_{Y/S}+B)+Mw\mathscr{L}))\sim\lambda_0^{\otimes\binom{m}{2}}\otimes\lambda_1^{\otimes m}\otimes \lambda_2
    \]
    for any $m\in\mathbb{Z}_{>0}$.
    Then, we set
    \[
    \lambda_{\mathrm{CM},q}:=\frac{2}{M^2}\lambda_0.
    \]
    For any morphism $h\colon C\to S$ from a smooth proper curve, we see that
    \[
    \mathrm{deg}_{C}h^*\lambda_{\mathrm{CM},q}=(K_{\mathcal{C}_C/C}+w\mathscr{L}_C)^2.
    \]

By using the data of the CM line bundle of the universal family of $\mathcal{M}^{\mathrm{Kss,qm}}_{m,w,v,\iota}$, we can define a $\mathbb{Q}$-line bundle $\lambda_{\mathrm{CM}}^{\mathrm{qmaps}}$ such that for any object $q\colon \mathcal{C}\to [\mathrm{Cone}(X)_S/\mathbb{G}_{m,S}]$ of $\mathcal{M}^{\mathrm{Kss,qm}}_{m,w,v,\iota}(S)$, $\lambda_{\mathrm{CM},q}$ is the pullback of $\lambda_{\mathrm{CM}}^{\mathrm{qmaps}}$.
Moreover, $\lambda_{\mathrm{CM}}^{\mathrm{qmaps}}$ is a pullback of a $\mathbb{Q}$-line bundle $\Lambda_{\mathrm{CM}}^{\mathrm{qmaps}}$ on $M^{\mathrm{Kps,qm}}_{m,w,v,\iota}$.
These facts can be shown by the same discussion as \cite[Lemma 3.36]{HH2}.
We also call $\lambda_{\mathrm{CM}}^{\mathrm{qmaps}}$ and $\Lambda_{\mathrm{CM}}^{\mathrm{qmaps}}$ the CM line bundles.
\end{de}

The following is the main result in this section.

\begin{thm}\label{thm--qm--cm--positive}
   $\Lambda_{\mathrm{CM}}^{\mathrm{qmaps}}$ is ample. In particular,  $M^{\mathrm{Kps,qm}}_{m,w,v,\iota}$ is a projective scheme.
\end{thm}

To show this, we introduce the following notion.

\begin{de}
  Let $q\colon \mathcal{C}\to [\mathrm{Cone}(X)_S/\mathbb{G}_{m,S}]$ be a family of stable quasimaps of degree $m$ and weight $w$.
  We say that $q$ has {\it maximal variation} if the morphism $S\to M^{\mathrm{KSBA,qm}}_{m,w,v,\iota}$ induced by $q$ is generically finite.
\end{de}

\begin{lem}\label{lem--cm-posi-const}
   Let $q\colon \mathcal{C}\to [\mathrm{Cone}(X)_S/\mathbb{G}_{m,S}]$ be a family of stable quasimaps of degree $m$ and weight $w$ over a projective normal variety $S$.
   Suppose that $q$ has maximal variation and $q_s$ is constant for any general closed point $s\in S$.
   Then, $\lambda_{\mathrm{CM},q}$ is big and nef on $S$.
\end{lem}

\begin{proof}
    Consider the following substack $\mathcal{M}^{\mathrm{Kss,ncqm}}_{m,w,v,\iota}$ of $\mathcal{M}^{\mathrm{Kss,qm}}_{m,w,v,\iota}$ such that for any object $q\in \mathcal{M}^{\mathrm{Kss,qm}}_{m,w,v,\iota}(S)$, $q$ belongs to $ \mathcal{M}^{\mathrm{Kss,ncqm}}_{m,w,v,\iota}$ if and only if
    $q_{\bar{s}}$ is not constant for any geometric point $\bar{s}\in S$.
    As in \cite[Lemma 3.26]{HH2}, we see that $\mathcal{M}^{\mathrm{Kss,ncqm}}_{m,w,v,\iota}$ is an open substack.
Let $M^{\mathrm{Kps,cqm}}_{m,w,v,\iota}\subset M^{\mathrm{Kps,qm}}_{m,w,v,\iota}$ be the image of the complement of $\mathcal{M}^{\mathrm{Kss,ncqm}}_{m,w,v,\iota}$ with the reduced structure.
By the same argument as in \cite[Lemma 3.42]{HH2}, we see that there exists a natural finite morphism $\mu\colon M^{\mathrm{Kps,cqm}}_{m,w,v,\iota}\to X\times M^{\mathrm{KSBA}}_{1,v,w}$.
Let $p_1$ and $p_2$ be the projections of $X\times M^{\mathrm{KSBA}}_{1,v,w}$.
To show the assertion, it suffices to show that $\Lambda_{\mathrm{CM}}^{\mathrm{qmaps}}|_{M^{\mathrm{Kps,cqm}}_{m,w,v,\iota}}\equiv 2wvp_1^*\iota^*\mathcal{O}_{\mathbb{P}^N}(1)+p_2^*\Lambda^{\mathrm{KSBA}}_{\mathrm{CM}}$.

Finally, we deal with the numerical equivalence above.
To show this, it is enough to show that for any family of stable quasimaps $q\colon Y\to [\mathrm{Cone}(X)_C/\mathbb{G}_{m,C}]$ over a proper smooth curve that induces a morphism $f\colon C\to M^{\mathrm{Kps,cqm}}_{m,w,v,\iota}$,
\[
(K_{Y/C}+w\mathscr{L})^2=(K_{Y/C}+\mathrm{Bs}(q))^2+2wv\mathrm{deg}_C(f^*p_1^*\iota^*\mathcal{O}_{\mathbb{P}^N}(1)),
\]
where $\mathscr{L}$ is the induced line bundle and $\mathrm{Bs}(q)$ is the horizontal part of $\mathrm{div}(f)$ for any general $f\in \mathrm{Sec}_q(\mathscr{L})$.
Since $\mathscr{L}-\mathrm{Bs}(q)$ is linearly equivalent to $\pi^*f^*p_1^*\iota^*\mathcal{O}_{\mathbb{P}^N}(1)$, where $\pi\colon Y\to C$ is the natural morphism. We obtain the proof.
\end{proof}

\begin{prop}\label{prop--wx}
Let $f\colon (X,\Delta;L)\to C$ be a polarized flat family of relative dimension $d$ over a proper smooth curve $C$ with $f_*\mathcal{O}_X\cong \mathcal{O}_C$.
Suppose that there exists a non-empty open subset $U\subset C$ such that $K_{X_U}+\Delta_U\sim_{\mathbb{Q}}L_U$ and $f_U$ is stable.
Then, there exist a surjective finite morphism $\pi\colon C'\to C$ of degree $r$ and a stable family $(X',\Delta')\to C'$ such that $(X'_U,\Delta'_U)\cong (X_U\times_U\pi^{-1}U,\Delta_U\times_U\pi^{-1}U)$.
Furthermore, 
\[
(K_{X'/C'}+\Delta')^{d+1}\le r((d+1)(K_{X/C}+\Delta)\cdot L^d-dL^{d+1}).
\]
\end{prop}

\begin{proof}
The first assertion follows from the properness of the moduli stack $\mathcal{M}^{\mathrm{KSBA}}_{d,v,c}$.    
To see the second assertion, we may replace $(X,\Delta)$ and $(X',\Delta')$ with their normalizations by \cite[Lemma 3.2]{PX}. 
Therefore, we may assume that $X$ is normal and \cite[Theorem 6]{WX} shows the assertion in this case.
We obtain the proof.
\end{proof}

\begin{lem}\label{lem--CM-minimize}
  Let $q\colon \mathcal{C}\to [\mathrm{Cone}(X)_S/\mathbb{G}_{m,S}]$ be a family of stable quasimaps of degree $m$ and weight $w$ over a projective smooth curve $S$.
  Let $\mathscr{L}$ be the induced line bundle and $\mathcal{D}$ a horizontal effective $\mathbb{Q}$-divisor on $\mathcal{C}$ such that $(\mathcal{C}_\eta,\mathcal{D}_\eta)$ is slc and  $w\mathscr{L}|_{\mathcal{C}_{\eta}}\sim_{\mathbb{Q}}\mathcal{D}|_{\mathcal{C}_{\eta}}$, where $\eta$ is the generic point.

  If $w\mathscr{L}-\mathcal{D}$ is effective, then there exist a finite cover $S'\to S$ from a smooth curve of degree $r$ and a stable family $(\mathcal{C}',\mathcal{D}')$ such that $(\mathcal{C}'_{\eta},\mathcal{D}'_{\eta})$ is isomorphic to $((\mathcal{C},\mathcal{D})\times_SS')_{\eta}$ and $\mathrm{deg}\lambda_{\mathrm{CM},q}\ge \frac{1}{r}\mathrm{deg}\lambda_{\mathrm{CM},(\mathcal{C}',\mathcal{D}')}$.
\end{lem}

\begin{proof}
By the relative ampleness of $K_{\mathcal{C}/S}+w\mathscr{L}$, we have that
\[
\mathrm{deg}\lambda_{\mathrm{CM},q}\ge 2(K_{\mathcal{C}/S}+\mathcal{D})\cdot (K_{\mathcal{C}/S}+w\mathscr{L})-(K_{\mathcal{C}/S}+w\mathscr{L})^2.
\]
By Proposition \ref{prop--wx}, there exist a finite morphism $S'\to S$ from a smooth curve of degree $r$ and a family of stable log pairs $(\mathcal{C}',\mathcal{D}')$ over $S'$ such that $(\mathcal{C}'_{\eta},\mathcal{D}'_{\eta})$ is isomorphic to $((\mathcal{C},\mathcal{D})\times_SS')_{\eta}$ and
\[
r(2(K_{\mathcal{C}/S}+\mathcal{D})\cdot (K_{\mathcal{C}/S}+w\mathscr{L})-(K_{\mathcal{C}/S}+w\mathscr{L})^2)\ge \mathrm{deg}\lambda_{\mathrm{CM},(\mathcal{C}',\mathcal{D}')}.
\]    
We obtain the proof by the two inequalities.
\end{proof}

\begin{prop}\label{prop--key--to--positivity}
    Let $q\colon \mathcal{C}\to [\mathrm{Cone}(X)_S/\mathbb{G}_{m,S}]$ be a family of stable quasimaps of degree $m$ and weight $w$ over a projective normal variety $S$.
   Suppose that $q$ has maximal variation.
   Then, $\lambda_{\mathrm{CM},q}$ is big and nef.
\end{prop}

\begin{proof}
First, we deal with nefness.
To show this, we may assume that $S$ is a proper smooth curve.
Take a general section $f\in \mathrm{Sec}_q(\mathscr{L})$ corresponding to a general coordinate function of $\mathbb{A}^{N+1}_{\mathbbm{k}}$.
Then, $(\mathcal{C}_s,w\mathrm{div}(f)_s)$ is stable for any general $s\in S$.
By Lemma \ref{lem--CM-minimize}, there exist a finite cover $S'\to S$ of degree $r$ and a family of stable log pairs $(\mathcal{C}',\mathcal{D}')\to S'$ such that $\mathrm{deg}\lambda_{\mathrm{CM},q}\ge \frac{1}{r}\mathrm{deg}\lambda_{\mathrm{CM},(\mathcal{C}',\mathcal{D}')}$.
By \cite[Theorem 1.1]{PX}, we see that $\mathrm{deg}\lambda_{\mathrm{CM},(\mathcal{C}',\mathcal{D}')}\ge0$.
Therefore, $\lambda_{\mathrm{CM},q}$ is nef.

    By Lemma \ref{lem--cm-posi-const}, we may assume that $q_s$ is not constant for any general $s\in S$.
By \cite[Lemma 3.35]{HH2} and the same argument in the proof of \cite[Lemma 3.43]{HH2}, taking general $N+2$ sections $f_0,\ldots, f_{N+1}$ corresponding to coordinate functions of $\mathbb{A}^{N+1}$, we see that
 \[
 \left(\mathcal{C},w\left(\frac{1}{2N+4}\sum_{i=0}^{N+1}\mathrm{div}_{q}(f_i)+\frac{1}{2N+2}\sum_{i=1}^{N+1}\mathrm{div}_{q}(f_i+f_0)\right)\right)
 \]
 has maximal variation and general fibers are stable.
 This means that there exists a generically finite rational map $g\colon S\dashrightarrow M^{\mathrm{KSBA}}_{1,v,(2(N+1)(N+2))^{-1}w}$ induced by the above family.
 Taking a resolution of the indeterminacy, we may assume that $g$ is a morphism.

 Next, we claim that $\lambda_{\mathrm{CM},q}-g^*\Lambda^{\mathrm{KSBA}}_{\mathrm{CM}}$ is pseudoeffective.
 To see this, by \cite[Theorem 2.2]{BDPP}, it suffices to show that for any general curve $C\subset S$, $(\lambda_{\mathrm{CM},q}-g^*\Lambda^{\mathrm{KSBA}}_{\mathrm{CM}})|_C$ has a nonnegative degree.
 This follows from Lemma \ref{lem--CM-minimize}.
 Therefore, $\lambda_{\mathrm{CM},q}-g^*\Lambda^{\mathrm{KSBA}}_{\mathrm{CM}}$ is pseudoeffective.
 Since $g^*\Lambda^{\mathrm{KSBA}}_{\mathrm{CM}}$ is big by Theorem \ref{thm--ksba}, $\lambda_{\mathrm{CM},q}$ is also big.
\end{proof}

\begin{proof}[Proof of Theorem \ref{thm--qm--cm--positive}]
  This immediately follows from Proposition \ref{prop--key--to--positivity} and \cite[Theorem 3.11]{kollar-moduli-stable-surface-proj} as the proof of Corollary \ref{cor--partial--positivity--of--cm}.  
\end{proof}

\section{Compactification of Moduli of klt Calabi-Yau varieties under the B-semiampleness conjecture}
\label{sec-7}

We say that $X$ is a projective {\it klt Calabi-Yau variety} if $X$ is a projective klt variety such that $K_X$ is numerically trivial.
In this case, $K_X\sim_{\mathbb{Q}}0$ by \cite[Theorem 1.2]{gongyo}.
Define the following set for any $d,v\in\mathbb{Z}_{>0}$;
$$\mathfrak{F}^{\text{klt,CY}}_{d,v}:=\left\{
 (X,L)
\;\middle|
\begin{array}{l}
\text{$X$ is a klt projective Calabi-Yau variety of dimension $d$ and}\\
\text{$L$ is an ample $\mathbb{Q}$-Cartier Weil divisor with $\mathrm{vol}(L)=v$.}
\end{array}\right\}.$$

Note that the following holds.

\begin{prop}\label{prop--birkar}
    There exists a positive integer $m\in\mathbb{Z}_{>0}$ depending only on $d$ and $v$ such that $mL$ is Cartier, very ample, and $H^i(X,\mathcal{O}_X(mjL))=0$ for any $(X,L)\in \mathfrak{F}^{\text{klt,CY}}_{d,v}$ and $i,j>0$.
    Furthermore, there are only finitely many polynomials $P_1,\ldots,P_k$ such that for any $(X,L)\in \mathfrak{F}^{\text{klt,CY}}_{d,v}$, there exists $j\in\{1,\ldots,k\}$ such that $P_j(l)=\chi(X,\mathcal{O}_X(lmL))$.
\end{prop}

\begin{proof}
This proposition is well-known to specialists, but we write a proof for the reader's convenience.
 By \cite[Corollary 1.6]{Bir}, we see that the set of $X$ appearing in $\mathfrak{F}^{\text{klt,CY}}_{d,v}$ is bounded and hence there exists $m_1\in\mathbb{Z}_{>0}$ depending only on $d$ and $v$ such that $m_1K_X$ is Cartier by \cite[Lemma 2.24]{Bir-bab}. 
 Then \cite[Theorem 1.1]{Bir} implies that there exists $m_1\in\mathbb{Z}_{>0}$ depending only on $d$ and $v$ such that for any $(X,L)\in \mathfrak{F}^{\text{klt,CY}}_{d,v}$, there exists an effective Weil divisor $E$ such that $E\sim m_2L$.
 Again by \cite[Corollary 1.6]{Bir}, the set of $(X,\mathrm{Supp}(E))$ such that there exists an element $(X,L)\in \mathfrak{F}^{\text{klt,CY}}_{d,v}$ such that $E$ is an effective Weil divisor satisfying $m_2L\sim E$ is log bounded.
By \cite[Lemma 2.24]{Bir-bab}, we can choose $m_1$ as above such that $m_1m_2L$ is Cartier.

The remaining assertions follow from the Kawamata--Viehweg vanishing theorem, \cite[Theorem 1.1]{kollar-eff-basepoint-free}, \cite[Lemma 7.1]{fujino-eff-slc}, and \cite[Lemma 3.9]{HH}.
\end{proof}

Then, the following result is well-known. 

\begin{thm}\label{thm--klt--CY--moduli}
    Let $\mathcal{M}^{\mathrm{klt,CY}}_{d,v}$ be a substack of $\mathfrak{Pol}$ such that the objects $\mathcal{M}^{\mathrm{klt,CY}}_{d,v}(S)$ for any scheme $S$ is the following collection.
    \[
\left\{
 f\colon(\mathcal{X},\mathcal{L})\to S
\;\middle|
\begin{array}{l}
\text{$f$ is a projective and locally stable family such that $(\mathcal{X}_{\bar{s}},\mathcal{L}_{\bar{s}})$}\\
\text{belongs to $\mathfrak{F}^{\mathrm{klt,CY}}_{d,v}$ for any geometric point $\bar{s}\in S$.}
\end{array}\right\}.
    \]
    Then $\mathcal{M}^{\mathrm{klt,CY}}_{d,v}$ is a separated Deligne--Mumford stack of finite type over $\mathbbm{k}$.  

    Furthermore, for the canonical morphism $\beta\colon \mathcal{M}^{\mathrm{klt,CY}}_{d,v}\to M^{\mathrm{klt,CY}}_{d,v}$ to the coarse moduli space $M^{\mathrm{klt,CY}}_{d,v}$ of $\mathcal{M}^{\mathrm{klt,CY}}_{d,v}$, there exists an ample $\mathbb{Q}$-line bundle $\Lambda_{\mathrm{Hodge}}$ on $M^{\mathrm{klt,CY}}_{d,v}$ such that $\beta^*\Lambda_{\mathrm{Hodge}}\sim_{\mathbb{Q}}\lambda_{\mathrm{Hodge}}$ unique up to $\mathbb{Q}$-linear equivalence, where $\lambda_{\mathrm{Hodge}}$ is the CM line bundle canonically defined on $\mathcal{M}^{\mathrm{klt,CY}}_{d,v}$. We call these line bundles the {\rm Hodge line bundles}.
\end{thm}

\begin{proof}
By Proposition \ref{prop--birkar}, we note that there exists $l_0\in\mathbb{Z}_{>0}$ such that $l_0L$ is very ample and $H^j(X,\mathcal{O}_{X}(lL))$ for any $(X,L)\in\mathfrak{F}^{\mathrm{klt,CY}}_{d,v}$ such that $L$ is Cartier, $j>0$, and $l\ge l_0$, and there are only finitely many possibilities $P_1,\ldots,P_k$ of the Hilbert polynomial $\chi(X,\mathcal{O}_X(mL))$ for any $(X,L)\in\mathfrak{F}^{\mathrm{klt,CY}}_{d,v}$ such that $L$ is Cartier.
Fix $P\in \{P_1,\ldots,P_k\}$ and consider the open and closed substack $\mathcal{M}^{\mathrm{klt,CY}}_{d,v,P}\subset\mathcal{M}^{\mathrm{klt,CY}}_{d,v}$ such that
\[
\left\{
 f\colon(\mathcal{X},\mathcal{L})\to S
\;\middle|
\begin{array}{l}
\text{$\chi(\mathcal{X}_{\bar{s}},\mathcal{L}^{\otimes m}_{\bar{s}})=P(m)$ for any $m\in\mathbb{Z}$}\\
\text{and geometric point $\bar{s}\in S$.}
\end{array}\right\}.
\]
Then, it suffices to show that $\mathcal{M}^{\mathrm{klt,CY}}_{d,v,P}$ is a separated Deligne--Mumford stack of finite type over $\mathbbm{k}$.
Consider $\mathbf{Hilb}_{\mathbb{P}^{P(l_0)-1}\times\mathbb{P}^{P(l_0+1)-1}/\mathbbm{k}}^{Q,p_1^*\mathcal{O}(1)\otimes p_2^*\mathcal{O}(1)}$, where $p_1\colon \mathbb{P}^{P(l_0)-1}\times\mathbb{P}^{P(l_0+1)-1}\to \mathbb{P}^{P(l_0)-1}$ and $p_2\colon \mathbb{P}^{P(l_0)-1}\times\mathbb{P}^{P(l_0+1)-1}\to \mathbb{P}^{P(l_0+1)-1}$ are natural projections and $Q(k):=P((2l_0+1)k)$.
Let 
$$\iota\colon \mathcal{U}\hookrightarrow \mathbb{P}^{P(l_0)-1}\times\mathbb{P}^{P(l_0+1)-1}\times \mathbf{Hilb}_{\mathbb{P}^{P(l_0)-1}\times\mathbb{P}^{P(l_0+1)-1}/\mathbbm{k}}^{Q,p_1^*\mathcal{O}(1)\otimes p_2^*\mathcal{O}(1)}$$
 be the universal object, and we put $\mathcal{A}:=p_1^*\mathcal{O}(-1)\otimes p_2^*\mathcal{O}(1)|_{\mathcal{U}}$.
 By the same argument just after \cite[Theorem 2.26]{HH2}, there exists a locally closed subscheme $Z_1 \subset \mathbf{Hilb}_{\mathbb{P}^{P(l_0)-1}\times\mathbb{P}^{P(l_0+1)-1}/\mathbbm{k}}^{Q,p_1^*\mathcal{O}(1)\otimes p_2^*\mathcal{O}(1)}$ such that for any morphism $T\to \mathbf{Hilb}_{\mathbb{P}^{P(l_0)-1}\times\mathbb{P}^{P(l_0+1)-1}/\mathbbm{k}}^{Q,p_1^*\mathcal{O}(1)\otimes p_2^*\mathcal{O}(1)}$ from a scheme, $T$ factors through $Z_1$ if and only if the following conditions are satisfied.
\begin{enumerate}
    \item For every geometric point $\bar{t}\in T$, $\mathcal{U}_{\bar{t}}$ is normal and connected,
    \item \label{item--Z--required}$\mathcal{A}^{\otimes l_0}\sim_Tp_{1,T}^*\mathcal{O}(1)|_{\mathcal{U}_T}$ and $\mathcal{A}^{\otimes l_0+1}\sim_Tp_{2,T}^*\mathcal{O}(1)|_{\mathcal{U}_T}$,
    \item for any point $t\in T$, $p_{1,t}\circ\iota_t$ and $p_{2,t}\circ\iota_t$ are closed immersion,
    \item for every point $t\in T$, the natural morphisms $\mathcal{O}_{T}^{\oplus  P(l_0)}\to H^0(\mathcal{U}_t,\mathcal{A}^{\otimes l_0}_{t})$ and $\mathcal{O}_{T}^{\oplus P(l_0+1)}\to H^0(\mathcal{U}_t,\mathcal{A}^{\otimes P(l_0+1)}_{t})$ induced by (\ref{item--Z--required}) are surjective, and
    \item for every point $t\in T$, $\mathcal{A}_t$ is ample, and the equalities $h^0(\mathcal{U}_t,\mathcal{A}^{\otimes l_0}_{t})=P(l_0)$, $h^0(\mathcal{U}_t,\mathcal{A}^{\otimes l_0+1}_{t})=P(l_0+1)$ and $H^i(\mathcal{U}_t,\mathcal{A}^{\otimes j}_{t})=0$ hold for any $i,j>0$.
\end{enumerate} 
On the other hand, there exists a positive integer $m_1$ such that $m_1K_X\sim\mathcal{O}_X$ for any $(X,L)\in \mathfrak{F}^{\mathrm{klt,CY}}_{d,v}$ by \cite[Lemma 7.2]{birkar-nefpart}. 
Thus, there exists a partial locally closed decomposition $Z_2\to Z_1$ such that for any morphism $g\colon T\to Z_1$ from a scheme, there exists $\omega_{\mathcal{U}_T/T}^{[j]}$ for any $1\le j\le m_1$ and $\omega_{\mathcal{U}_T/T}^{[m_1]}$ is a line bundle if and only if $g$ factors through $Z_2$ by \cite[Theorem 9.40]{kollar-moduli}.
Then $$Z_3:=\{t\in Z_2|\mathcal{U}_t\text{ is klt}\}$$ is open by \cite[Corollary 4.10]{kollar-mmp}. 
Furthermore, there exists a closed subscheme $Z\subset Z_3$ such that for any morphism $g\colon T\to Z_3$ from a scheme, $g$ factors through $T$ if and only if $\omega_{\mathcal{U}_T/T}^{[m_1]}\sim_T0$ by the separatedness of $\mathbf{Pic}_{\mathcal{U}_{Z_3}/Z_3}$.
We note that for any object $f\colon (\mathcal{X},\mathcal{L})\to S\in\mathcal{M}^{\mathrm{klt,CY}}_{d,v}(S)$, $f^*f_*\omega_{\mathcal{X}/S}^{[m_1]}\cong\omega_{\mathcal{X}/S}^{[m_1]}$ by Lemma \ref{lem--dubois}.
By this, it is easy to show that $[Z/PGL(P(l_0))\times PGL(P(l_0+1))]\cong \mathcal{M}^{\mathrm{klt,CY}}_{d,v,P}$ as stacks.
The remaining part can be shown as \cite[Theorem 2.26]{HH2}.
\end{proof}

From now on, we fix an irreducible component $\mathcal{V}\subset \mathcal{M}^{\mathrm{klt,CY}}_{d,v}$.
Let $\mathcal{V}^{(m)}$ be the image structure of the canonical morphism $f_m\colon\mathcal{V}\to \mathcal{M}^{\mathrm{klt,CY}}_{d,m^d\cdot v}$ maps $(\mathcal{X},\mathcal{L})$ to $(\mathcal{X},\mathcal{L}^{\otimes m})$ for any $m\in\mathbb{Z}_{>0}$. 
\begin{lem}\label{lem--quasi-finite-exponent}
  The canonical morphism $\mathcal{V}\to \mathcal{V}^{(m)}$ is representable and quasi-finite, and $\mathcal{V}^{(m)}$ is an irreducible component of $\mathcal{M}^{\mathrm{klt,CY}}_{d,m^d\cdot v}$ for any $m\in\mathbb{Z}_{>0}$.
\end{lem}

\begin{proof}
Let $f_{m}\colon \mathcal{V}\to \mathcal{V}^{(m)}$ be the canonical morphism.
  Fix a closed point of $\mathcal{V}$ corresponding to $(X,L)$.
  Then the set of $(X',L')$ such that $(X,L^{\otimes m})$ and $(X',L'^{\otimes m})$ are isomorphic is finite
 by Proposition \ref{prop--mu_k}.
 Therefore, the canonical morphism $\mathcal{V}\to \mathcal{V}^{(m)}$ is quasi-finite.
 Furthermore, $\mathrm{Aut}(X,L)\subset \mathrm{Aut}(X,L^{\otimes m})$ and this shows that $f_{m}$ is representable by \cite[Lemma 4.4.3]{AV}, where $\mathrm{Aut}(X,L)$ denotes the closed subgroup of $\mathrm{Aut}_X$ preserving the linear equivalence class of $L$.

 Take a general closed point $p$ of $\mathcal{V}$ that is not contained in any other irreducible component of $\mathcal{V}$.
 To see that $\mathcal{V}^{(m)}$ is an irreducible component, it suffices to show that $f_{m}\colon \mathcal{V}\to\mathcal{M}^{\mathrm{klt,CY}}_{d,m^d\cdot v}$ is \'etale around $p$.
Set $A$ be an Artin local ring with the maximal ideal $\mathfrak{m}$ and an ideal $I$ such that $\mathfrak{m}\cdot I=0$.
Let $\kappa:=A/\mathfrak{m}$. 
Let $\pi_1\colon (\mathcal{X},\mathcal{L})\to \mathrm{Spec}\,A/I$ be an object of $\mathcal{V}(A/I)$ such that the restriction to $\mathrm{Spec}\,\kappa$ coincides with $(X,L)$.
On the other hand, let $\pi_2\colon (\overline{\mathcal{X}},\overline{\mathcal{A}})\to \mathrm{Spec}\,A$ be an object of $\mathcal{M}^{\mathrm{klt,CY}}_{d,m^d\cdot v}(A)$ such that the restriction to $\mathrm{Spec}\,A/I$ coincides with $(\mathcal{X},\mathcal{L}^{\otimes m})$.
Then, we claim that there exists a unique object $\overline{\pi_1}\colon (\overline{\mathcal{X}}',\overline{\mathcal{L}})\to \mathrm{Spec}\,A$ up to isomorphism whose restriction to $\mathrm{Spec}\,A/I$ coincides with $\pi_1$ such that there exists an isomorphism $g\colon (\overline{\mathcal{X}}',\overline{\mathcal{L}}^{\otimes m})\to (\overline{\mathcal{X}},\overline{\mathcal{A}})$ over $\mathrm{Spec}\,A$ such that $g_{(A/I)}$ is the identity. 
To show the existence of $\overline{\pi_1}$, we may choose $\overline{\mathcal{X}}'=\overline{\mathcal{X}}$.
Let $o(\mathcal{L})\in H^2(\mathcal{X}_{\kappa},\mathcal{O}_{\mathcal{X}_{\kappa}})\otimes_{\kappa}I$ be the deformation obstruction.
Since $mo(\mathcal{L})=o(\mathcal{L}^{\otimes m})=0$, we see that $o(\mathcal{L})=0$ and hence there exists a line bundle $\overline{\mathcal{L}}$ on $\overline{\mathcal{X}}$ such that $\overline{\mathcal{L}}|_{\mathcal{X}}=\mathcal{L}$ and $\overline{\mathcal{L}}^{\otimes m}\sim \overline{\mathcal{A}}$.
For the uniqueness, take an object $\overline{\pi_1}\colon (\overline{\mathcal{X}}',\overline{\mathcal{L}})\to \mathrm{Spec}\,A$ as above.
It is easy to see that $\overline{\mathcal{X}}'\cong \overline{\mathcal{X}}$ and hence we may assume that $\overline{\mathcal{X}}'= \overline{\mathcal{X}}$.
Suppose that there exists another line bundle $\overline{\mathcal{L}}'$ such that $\overline{\mathcal{L}}'|_{\mathcal{X}}=\mathcal{L}$ and $\overline{\mathcal{L}}'^{\otimes m}\sim \overline{\mathcal{L}}^{\otimes m}$.
Then, $\overline{\mathcal{L}}'\otimes\overline{\mathcal{L}}^{\otimes-1}$ induces a morphism $\eta\colon \mathrm{Spec}\,A\to\mathbf{Pic}^0_{\overline{\mathcal{X}}/A}$.
Since the kernel of the $m$-th power morphism of $\mathbf{Pic}^0_{\overline{\mathcal{X}}/A}$ is \'etale over $\mathrm{Spec}\,A$, we see that $\eta$ coincides with the identity section. 
Thus, $\overline{\mathcal{L}}'\sim \overline{\mathcal{L}}$.
This and \cite[Theorem 3.1 (iii)]{SGA} show that $f_{m}\colon \mathcal{V}\to\mathcal{M}^{\mathrm{klt,CY}}_{d,m^d\cdot v}$ is \'etale around $p$.
We obtain the proof.
\end{proof}

For any positive integers $l$ and $m$ such that $\frac{m}{l}\in\mathbb{Z}$, we can consider the canonical morphism $f_{l,m}\colon \mathcal{V}^{(l)}\to\mathcal{V}^{(m)}$.
\begin{prop}\label{prop--v^l-stop}
There exists $l\in\mathbb{Z}_{>0}$ such that $f_{l,m}$ is finite and surjective, and the morphism $\overline{f}_{l,m}\colon V^{(l)}\to V^{(m)}$ for the coarse moduli space induced by $f_{l,m}$ is bijective for any positive integer $m$ such that $\frac{m}{l}\in\mathbb{Z}$.  
Moreover, for any $(X,A)\in\mathcal{V}^{(l)}(\mathbbm{k})$, $A$ is very ample and $H^j(X,\mathcal{O}_X(ilA))=0$ for any $i,j>0$.
\end{prop}

\begin{proof}
By Proposition \ref{prop--birkar}, we may take $l$ such that $lL$ is Cartier for any $(X,L)\in \mathfrak{F}^{\text{klt,CY}}_{d,v}$.
Furthermore, we may assume that for any $(X,L)\in \mathfrak{F}^{\text{klt,CY}}_{d,v}$, if $M\in \mathbf{W}^{\mathbb{Q}}\mathbf{Pic}^\tau_{X/\mathbbm{k}}$, then $lM\in \mathbf{Pic}^0_{X/\mathbbm{k}}$ by Corollary \ref{cor--torsion--components--number}.

First, we see that for any $(X,A)\in \mathcal{V}^{(l)}(\mathbbm{k})$, there exists an element $(X,L)$ of $\mathfrak{F}^{\text{klt,CY}}_{d,v}$ such that $lL\sim A$. 
If $(X,A)$ corresponds to a general closed point of $\mathcal{V}^{(l)}$, then the claim follows from the fact that $f_l$ is dominant.
We deal with the general case.
Let $\mathcal{U}\subset\mathcal{V}^{(l)}$ be an open dense substack that is contained in the image of $f_{l}$. 
Fix a $\mathbbm{k}$-valued point $p\in \mathcal{V}^{(l)}$ and then it is easy to see that there exists a morphism $C\to \mathcal{V}^{(l)}$ from an affine smooth curve with a closed point $c\in C$ such that $c$ is assigned to $p$ but every point of $C\setminus\{c\}$ is assigned to $\mathcal{U}$.
Let $(\mathcal{X},\mathcal{A})$ be the family over $C$ corresponding to the above morphism.
Replacing $C$ with its finite cover and shrinking $C$ if necessary, we may assume that $C\setminus\{c\}\to\mathcal{U}$ factors through $\mathcal{V}$.
This means that there exists a line bundle $\mathcal{L}$ on $\mathcal{X}\times_C(C\setminus\{c\})$ such that $\mathcal{L}^{[l]}\sim \mathcal{A}|_{\mathcal{X}\times_C(C\setminus\{c\})}$.
Since $\mathcal{X}_c$ is irreducible and normal, there exists a divisorial sheaf $\overline{\mathcal{L}}$ on $\mathcal{X}$ such that $\overline{\mathcal{L}}|_{\mathcal{X}\times_C(C\setminus\{c\})}\sim\mathcal{L}$ and $\mathcal{A}\sim_C\overline{\mathcal{L}}^{[l]}$ by \cite[II, Proposition 6.5]{Ha}.
This means that $\overline{\mathcal{L}}$ is $\mathbb{Q}$-Cartier and by $\overline{\mathcal{L}}|_{\mathcal{X}_c}$ is also a $\mathbb{Q}$-Cartier divisorial sheaf by \cite[Corollary 5.25]{KM}.
Thus, $(\mathcal{X}_c,\overline{\mathcal{L}}|_{\mathcal{X}_c})\in \mathfrak{F}^{\text{klt,CY}}_{d,v}$ and $\overline{\mathcal{L}}|_{\mathcal{X}_c}^{[l]}\sim \mathcal{A}_c$.
By replacing $l$ if necessary, we have that for any $(X,A)\in\mathcal{V}^{(l)}(\mathbbm{k})$, $A$ is very ample and $H^j(X,\mathcal{O}_X(iA))=0$ for any $i,j>0$.

As in the proof of Lemma \ref{lem--quasi-finite-exponent}, it is easy to see that $f_{l,m}$ is representable.  
Since $f_m=f_{l,m}\circ f_l$, we see that $f_{l,m}$ is quasi-finite and dominant by Lemma \ref{lem--quasi-finite-exponent}.
To show that $f_{l,m}$ is finite and surjective, it suffices to show that $f_{l,m}$ is proper.
Let $R$ be a discrete valuation ring with the maximal ideal $\mathfrak{m}$, the fractional field $K$ and the residue field $\kappa$.
Let $g\colon\mathrm{Spec}\,R\to \mathcal{V}^{(m)}$ and $h\colon\mathrm{Spec}\,K\to \mathcal{V}^{(l)}$ be morphisms such that $g|_{\mathrm{Spec}\,K}=f_{l,m}\circ h$.
Then we claim that there exist a finite extension $R\subset R'$ and a unique morphism $\bar{h}\colon\mathrm{Spec}\,R'\to \mathcal{V}^{(l)}$ such that $\bar{h}|_{\mathrm{Spec}\,K}=h\circ\pi$ and $g\circ\pi=f_{l,m}\circ \bar{h}$, where $\pi\colon \mathrm{Spec}\,R'\to\mathrm{Spec}\,R$ is the canonical morphism.
Indeed, take a flat projective morphism $\mathcal{X}\to \mathrm{Spec}\,R$ with a line bundle $\mathcal{L}$ on $\mathcal{X}_K$ such that $\mathcal{L}^{\otimes \frac{m}{l}}$ can be extended to a line bundle $\mathcal{A}$ such that $(\mathcal{X},\mathcal{A})$ corresponds to $g$ and $(\mathcal{X}_K,\mathcal{L})$ corresponds to $h$.
Since $\mathcal{X}_0$ the fiber over the closed point of $\mathrm{Spec}\,R$ is normal and irreducible, $\mathcal{L}$ can be extended to a divisorial sheaf $\overline{\mathcal{L}}$ on $\mathcal{X}$.
Note that $\frac{m}{l}\overline{\mathcal{L}}\sim \mathcal{A}$ and hence $\overline{\mathcal{L}}$ is $\mathbb{Q}$-Cartier. \cite[Corollary 5.25]{KM} shows that $\overline{\mathcal{L}}$ is Cohen--Macaulay and $\mathcal{L}_\kappa:=\overline{\mathcal{L}}|_{\mathcal{X}_\kappa}$ is also a $\mathbb{Q}$-Cartier divisorial sheaf. 
On the other hand, there exists a finite extension $R\subset R'$ with the fractional field $K'$ with a $\mathbb{Q}$-Cartier divisorial sheaf $\mathcal{L}'$ on $\mathcal{X}_{K'}$ such that $\mathcal{L}_{K'}\sim \mathcal{L}'$ by what we have shown in the previous paragraph.
By the above discussion, we see that there exists a $\mathbb{Q}$-Cartier divisorial sheaf $\mathcal{L}'_0$ such that $\mathcal{L}'^{[l]}_0\sim\mathcal{L}_0$, where $0\in\mathrm{Spec}\,R'$ is an arbitrary closed point mapped to the closed point of $\mathrm{Spec}\,R$.
By this, we see that $(\mathcal{X}_0,\mathcal{L}_0)$ belongs to $\mathfrak{F}^{\text{klt,CY}}_{d,v}$.
Thus, $\mathcal{L}_0$ is Cartier by the choice of $l$.
This means that $(\mathcal{X},\overline{\mathcal{L}})$ belongs to $\mathcal{V}^{(l)}$ and $\bar{h}$ corresponding to $(\mathcal{X},\overline{\mathcal{L}})$ is the desired one.
Therefore, $f_{l,m}$ is proper by \cite[Theorem 11.5.1]{Ols}.

From now on, we will show that $\overline{f}_{l,m}$ is bijective.
Note that $\overline{f}_{l,m}$ is surjective, as is $f_{l,m}$.
Because $V^{(l)}$ and $V^{(m)}$ are algebraic spaces of finite type over $\mathbbm{k}$, it suffices to show that for any two elements $(X,A)$ and $(X',A')$ in $\mathcal{V}^{(l)}(\mathbbm{k})$ such that there is an isomorphism $g\colon X\to X'$ such that $\frac{m}{l}g^*A'\sim\frac{m}{l}A$, then there exists an isomorphism $h\colon X\to X'$ such that $h^*A'\sim A$.
Note that there exists a divisorial sheaf $L$ and $L'$ such that $lL=A$ and $lL'=A'$.
Then, $g^*L'-L\in \mathbf{W}^{\mathbb{Q}}\mathbf{Pic}^\tau_{X/\mathbbm{k}}$.
By the choice of $l$, we see that $g^*A'-A\in \mathbf{Pic}^0_{X/\mathbbm{k}}$.
By \cite[Lemma 6.2]{Hat23}, we see that there exists an automorphism $g'\colon X\to X$ such that $g'^*g^*A'\sim A$.
By letting $h=g\circ g'$, we obtain the assertion.
Therefore, $\overline{f}_{l,m}$ is injective.
\end{proof}

Take $l$ as the above proposition. Let $V^{\mathrm{sn}}$ be the seminormalization of $V^{(l)}$ (for the definition of the seminormalization and seminormal schemes, see \cite{kollar-moduli} and \cite[Definition 9.51]{AEHK}).
By Proposition \ref{prop--v^l-stop}, $V^{\mathrm{sn}}$ is the seminormalization of $V^{(m)}$ for any $m\in l\mathbb{Z}_{>0}$.
Let $ V^\nu$ be the normalization of $V^{\mathrm{sn}}$.

We will assume the following condition.

\begin{cond}\label{ass--b-ss}
    Let $f\colon (X,\Delta)\to S$ be an lc--trivial fibration over a normal projective variety $S$.
    Fix an ample line bundle $A$ on $X$.
    Suppose that there exists $m\in l\mathbb{Z}_{>0}$ such that $\Delta_s=0$ and $(X_s,A_s)$ belongs to $\mathcal{V}^{(m)}(\mathbbm{k})$ for any general closed point $s\in S$.
    Then, for any Ambro model $\tilde{S}\to S$, the moduli $\mathbb{Q}$-divisor $M_{\tilde{S}}$ is semiample.

    Furthermore, let $V_1\subset V^{\mathrm{sn}}$ be the normal locus.
\end{cond}

\begin{rem}\label{rem-bfmt}
The first part of the condition follows from the $\mathbf{B}$-semiampleness conjecture, which is solved by \cite[Theorem 1.5]{BFMT}.
We note that the latter part holds in the following cases.
\begin{enumerate}
    \item $V^{\mathrm{sn}}$ parameterizes Abelian varieties, or polarized symplectic varieties smoothable to irreducible holomorphic symplectic manifolds (cf.~ \cite[Remark 2.34 and Theorem 2.45]{HH2}).
    In this case, we can take $V_1$ as $V^{\mathrm{sn}}$.
    \item If $V_1$ parametrizes polarized smooth Calabi--Yau varieties with trivial canonical divisor, then $V_1$ is the coarse moduli space of a smooth stack (cf.~\cite{GHJ}).
    However, $\mathcal{M}^{\mathrm{klt,CY}}_{d,v}$ is not smooth in general by the example of Gross \cite{Gro}.
    \item If $V_1$ parameterizes polarized Enriques surfaces, then $V_1$ is normal because they have no deformation obstruction and the geometric genus is zero (cf.~\cite{BHPV}). 
\end{enumerate}
\end{rem}

\begin{thm}\label{thm--canonical--bundle--formula-sin}
Under Condition \ref{ass--b-ss}, fix $l$ as in Proposition \ref{prop--v^l-stop}.
Then, there exist a positive integer $k\in l\cdot\mathbb{Z}_{>0}$ and a normal projective scheme $\overline{V}^{\mathrm{BB}}$ that contains $V^{\nu}$ as an open subscheme with an ample $\mathbb{Q}$-line bundle $H$ depending only on $l$  that satisfy the following properties for any $m\in l\cdot\mathbb{Z}_{>0}$.
\begin{enumerate}
    \item $kH$ is very ample and fix a closed embedding $\iota\colon \overline{V}^{\mathrm{BB}}\hookrightarrow \mathbb{P}^L$ that is induced by the complete linear system $|kH|$,
    \item Let $f\colon (X,\Delta)\to S$ be an lc--trivial fibration over a normal projective variety $S$.
    Fix an $f$-ample line bundle $A$ on $X$.
    Suppose that there exists an open dense subset $U\subset S$ such that $\Delta_s=0$ and $(X_s,A_s)$ belongs to $\mathcal{V}^{(m)}(\mathbbm{k})$ for any closed point $s\in U$. 
    Suppose further that there exists the canonical morphism $\mathcal{P}\colon U\to V_1$.
    Let $\tilde{S}\to S$ be an Ambro model and $M_{\tilde{S}}$ the moduli $\mathbb{Q}$-divisor on $\tilde{S}$.
    Then there exists a unique morphism $\widetilde{\mathcal{P}}\colon \tilde{S}\to \overline{V}^{\mathrm{BB}}$ that is equivalent to $\mathcal{P}$ as birational maps such that $\widetilde{\mathcal{P}}^*(kH)\sim kM_{\tilde{S}}$.
\end{enumerate}
\end{thm}

\begin{de}\label{de--hodge}
    We call such $H$ as Theorem \ref{thm--canonical--bundle--formula-sin} the {\it Hodge line bundle} on $\overline{V}^{\mathrm{BB}}$ and let $\Lambda_{\mathrm{Hodge}}$ denote $H$.
\end{de}

To show this, we have to consider the following.
\begin{de}
Fix $\epsilon\in\mathbb{Q}_{>0}$ an arbitrary sufficiently small number.  
Take an arbitrary $m\in l\mathbb{Z}_{>0}$.
Consider the following locally closed substack $\mathcal{M}^{\mathrm{KX},\epsilon,\circ}_{V,(m)}$ of $\mathcal{M}^{\mathrm{KSBA}}_{d,(\epsilon ml^{-1})^d\cdot v,\epsilon}$ such that for any scheme $S$, the collection $\mathcal{M}^{\mathrm{KX},\epsilon,\circ}_{V,(m)}(S)$ of objects is 
$$\left\{
 (\mathcal{X},\epsilon \mathcal{H})\to S
\;\middle|
\begin{array}{l}
\text{for any geometric point $\bar{s}\in S$, there exists $(X,A)\in \mathcal{V}^{(m)}(\overline{\kappa(s)})$}\\
\text{such that $f\colon X\cong \mathcal{X}_{\bar{s}}$  and $f^*\mathcal{H}_{\bar{s}}\sim A$ for some $f$.}
\end{array}\right\}.$$
Let $\mathcal{M}^{\mathrm{KX},\epsilon}_{V,(m)}$ be the closure stack, which is an irreducible proper Deligne--Mumford stack by the result of \cite{KX}.
It is easy to see that the seminormalization $\mathcal{M}^{\mathrm{KX},\mathrm{sn}}_{V,(m)}$ of $\mathcal{M}^{\mathrm{KX},\epsilon}_{V,(m)}$ is independent of the choice of any sufficiently small $\epsilon>0$ by \cite[Theorem 2]{KX}.
Let $\pi_{\mathrm{KX}}^m\colon \mathcal{M}^{\mathrm{KX},\mathrm{sn}}_{V,(m)}\to M^{\mathrm{KX},\mathrm{sn}}_{V,(m)}$ be the coarse moduli space, which is projective by Theorem \ref{thm--ksba}.
Note that $M^{\mathrm{KX},\mathrm{sn}}_{V,(m)}$ is also seminormal by \cite[Lemma 13.7]{ABB+}.
Let $\mathcal{M}^{\mathrm{KX},\nu,\circ}_{V,(m)}$, $M^{\mathrm{KX},\nu,\circ}_{V,(m)}$, $\mathcal{M}^{\mathrm{KX},\nu}_{V,(m)}$ and $M^{\mathrm{KX},\nu}_{V,(m)}$ be the normalizations of $\mathcal{M}^{\mathrm{KX},\epsilon,\circ}_{V,(m)}$, $M^{\mathrm{KX},\epsilon,\circ}_{V,(m)}$, $\mathcal{M}^{\mathrm{KX},\mathrm{sn}}_{V,(m)}$ and $M^{\mathrm{KX},\mathrm{sn}}_{V,(m)}$, respectively.
Let $\pi^{\nu,m}_{\mathrm{KX}}\colon\mathcal{M}^{\mathrm{KX},\nu}_{V,(m)}\to M^{\mathrm{KX},\nu}_{V,(m)}$ denote the morphism of the coarse moduli space.

On the other hand, since $\mathcal{M}^{\mathrm{KX},\epsilon}_{V,(m)}$ is bounded for any sufficiently small $\epsilon>0$, we can take $\mu\in\mathbb{Z}_{>0}$ such that for any scheme and $f\colon(X,\epsilon H)\to S\in \mathcal{M}^{\mathrm{KX},\epsilon}_{V,(m)}(S)$, $f_*\omega_{X/S}^{[\mu]}$ is a line bundle over $S$ by Lemma \ref{lem--dubois}.
Using this argument, we can take a line bundle $\mathcal{L}$ on $\mathcal{M}^{\mathrm{KX},\epsilon}_{V,(m)}$ such that for any morphism $h\colon S\to \mathcal{M}^{\mathrm{KX},\epsilon}_{V,(m)}$ corresponding to $f\colon (X,\epsilon H)\to S$, $h^*\mathcal{L}\cong f_*\omega_{X/S}^{[\mu]}$.
Then, we set $\overline{\lambda}_{\mathrm{Hodge}}:=\frac{1}{\mu}\mathcal{L}$.
Let $\overline{\lambda}_{\mathrm{Hodge}}$ also denote the $\mathbb{Q}$-line bundle pullback to $\mathcal{M}^{\mathrm{KX},\mathrm{sn}}_{V,(m)}$.
We note that there exists a unique $\mathbb{Q}$-line bundle $\overline{\Lambda}_{\mathrm{Hodge}}$ on $M^{\mathrm{KX},\mathrm{sn}}_{V,(m)}$ up to $\mathbb{Q}$-linear equivalence such that $\pi_{\mathrm{KX}}^*\overline{\Lambda}_{\mathrm{Hodge}}\sim_{\mathbb{Q}}\overline{\lambda}_{\mathrm{Hodge}}$.
\end{de}

We note that $(\mathcal{M}^{\mathrm{KX},\epsilon,\circ}_{V,(m)})^{\mathrm{sn}}=\mathcal{M}^{\mathrm{KX},\mathrm{sn}}_{V,(m),\mathrm{klt}}$ by the proof of Proposition \ref{prop--v^l-stop}.

\begin{prop}
There exists a proper smooth morphism $\tau_{(m)}\colon \mathcal{M}^{\mathrm{KX},\epsilon,\circ}_{V,(m)}\to \mathcal{V}^{(m)}$ that maps $(\mathcal{X},\epsilon\mathcal{H})\in \mathcal{M}^{\mathrm{KX},\epsilon,\circ}_{V,(m)}(S)$ to $(\mathcal{X},\mathcal{A})\in \mathcal{V}^{(m)}(S)$, where $\mathcal{A}$ is the line bundle associated with $\mathcal{H}$ for any $m\in l\mathbb{Z}_{>0}$ and sufficiently small $\epsilon>0$. 

Furthermore, let $\tau_m\colon M^{\mathrm{KX},\nu,\circ}_{V,(m)}\to V^{\nu}$ be the morphism of the normalizations of the coarse moduli spaces induced by $\tau_{(m)}$.
Then, $(\tau_m)_*\mathcal{O}_{M^{\mathrm{KX},\nu,\circ}_{V,(m)}}\cong\mathcal{O}_{V^\nu}$.
\end{prop}

\begin{proof}
Let $S\to \mathcal{V}^{(m)}$ be an arbitrary morphism from a scheme that corresponds to a family $f\colon(\mathcal{X},\mathcal{A})\to S$.   
Replacing $S$ with its \'etale cover, we assume that $\mathcal{A}$ is a line bundle.
By shrinking $S$, we may assume that $f_*\mathcal{O}_{\mathcal{X}}(\mathcal{A})\cong \mathcal{O}_S^{\oplus h^0(\mathcal{X}_s,\mathcal{A}_s)}$ for any $s\in S$.
Then, $S\times_{\mathcal{V}^{(m)}}\mathcal{M}^{\mathrm{KX},\epsilon,\circ}_{V,(m)}\cong [\mathbb{P}_S(f_*\mathcal{O}_X(\mathcal{A}))/PGL_S(h^0(\mathcal{X}_s,\mathcal{A}_s))]$ for any sufficiently small $\epsilon>0$.
The right hand side is smooth over $S$.
Thus, $\tau_{(m)}$ is smooth and each geometric fiber is connected.

To show the properness, let $R$ be an arbitrary discrete valuation ring with a morphism $g\colon \mathrm{Spec}\,R\to \mathcal{V}^{(m)}$.
Let $K$ be the fractional field of $R$ and suppose that there exists a morphism $g'\colon \mathrm{Spec}\,K\to \mathcal{M}^{\mathrm{KX},\epsilon,\circ}_{V,(m)}$. 
Suppose that $g|_{\mathrm{Spec}\,K}=\tau_{(m)}\circ g'$.
Let $(\mathcal{X},\mathcal{A})$ be the family corresponding to $g$ and $\mathcal{H}_{K}$ the Cartier divisor such that $(\mathcal{X}_K,\epsilon \mathcal{H}_K)$ corresponds to $g'$.
There exist a finite cover $R'$ of $R$ with the fractional field $K'$ and the family $(\mathcal{X}',\epsilon\mathcal{H}')\to \mathrm{Spec}\,R'$ corresponding to an object of $\mathcal{M}^{\mathrm{KX},\epsilon}_{V,(m)}(R')$ such that $(\mathcal{X}'_{K'},\epsilon\mathcal{H}'_{K'})\cong (\mathcal{X}_{K'},\epsilon\mathcal{H}_{K'})$.
$\mathcal{X}'_{s'}$ is slc and $K_{\mathcal{X}'_{s'}}\sim_{\mathbb{Q}}0$ for any closed point $s'\in\mathrm{Spec}\,R'$ by the choice of $\epsilon$.
\cite[Theorem 4.2]{O3} shows that $\mathcal{X}'\cong \mathcal{X}_{R'}$ and $\mathcal{H}'\sim \mathcal{A}_{R'}$. 
This and \cite[Theorem 11.5.1]{Ols} show that $\tau_{(m)}$ is proper.

Let $(\tau_{(m)})^{\nu}\colon \mathcal{M}^{\mathrm{KX},\nu,\circ}_{V,(m)}\to (\mathcal{V}^{(m)})^{\nu}$ be the morphism induced by the normalizations of stacks.
Note that $M^{\mathrm{KX},\nu,\circ}_{V,(m)}$ and $V^{\nu}$ are the coarse moduli spaces of  $\mathcal{M}^{\mathrm{KX},\nu,\circ}_{V,(m)}$ and $(\mathcal{V}^{(m)})^{\nu}$, respectively.
By the observation in the first paragraph, it is easy to see that any geometric fiber of $\tau_m$ is normal and connected.
We note that $\tau_m$ is projective by the second paragraph of this proof since $M^{\mathrm{KX},\nu,\circ}_{V,(m)}$ is quasi-projective (cf.~Theorem \ref{thm--ksba}).
By \cite[III, Theorems 9.9 and 9.11]{Ha}, $\tau_m$ is flat.
By \cite[Exercise 9.3.11]{FGA}, we can deduce $(\tau_m)_*\mathcal{O}_{M^{\mathrm{KX},\nu,\circ}_{V,(m)}}\cong\mathcal{O}_{V^\nu}$.
\end{proof}

Therefore, we see that $\mathcal{M}^{\mathrm{KX},\nu,\circ}_{V,(m)}=\mathcal{M}^{\mathrm{KX},\epsilon,\circ}_{V,(m)}\times_{\mathcal{M}^{\mathrm{KX},\epsilon}_{V,(m)}}\mathcal{M}^{\mathrm{KX},\nu}_{V,(m)}$ is smooth and proper over $(\mathcal{V}^{(m)})^{\nu}$.

\begin{prop}\label{prop--baily-borel-V^bb}
Under Condition \ref{ass--b-ss}, there exists a positive integer $k_m\in\mathbb{Z}_{>0}$ for any $m\in l\mathbb{Z}_{>0}$ with the following property. 
Let $\nu\colon M^{\mathrm{KX},\nu}_{V,(m)}\to M^{\mathrm{KX},\mathrm{sn}}_{V,(m)}$ be the normalization. 
Then $k_m\nu^*\overline{\Lambda}_{\mathrm{Hodge}}$ is a globally generated line bundle and $|k_m\nu^*\overline{\Lambda}_{\mathrm{Hodge}}|$ defines a morphism $\theta\colon M^{\mathrm{KX},\nu}_{V,(m)}\to \mathbb{P}^{h^0(k_m\nu^*\overline{\Lambda}_{\mathrm{Hodge}})-1}$ such that
    \begin{enumerate}
\item $\overline{V}^{\mathrm{BB}}$ is a normal projective variety, where $\overline{V}^{\mathrm{BB}}$ denotes the image of $\theta$,
\item $\theta_*\mathcal{O}_{M^{\mathrm{KX},\nu}_{V,(m)}}\cong \mathcal{O}_{\overline{V}^{\mathrm{BB}}}$, and 
\item $\theta(M^{\mathrm{KX},\nu,\circ}_{V,(m)})$ is an open subset of $\overline{V}^{\mathrm{BB}}$ isomorphic to $V^{\mathrm{sn}}$ and $$\theta^{-1}\theta(M^{\mathrm{KX},\nu,\circ}_{V,(m)})=M^{\mathrm{KX},\nu,\circ}_{V,(m)}.$$
    \end{enumerate}
    Furthermore, $\overline{V}^{\mathrm{BB}}$ is independent from the choice of $m$.
    In particular, there exists a canonical open immersion $V^\nu\subset \overline{V}^{\mathrm{BB}}$.
\end{prop}

\begin{proof}
We note that $M^{\mathrm{KX},\nu}_{V,(m)}$ is a projective normal variety.   
By \cite[Proposition 2.6]{Vis}, there exists a normal variety $S$ finite over $M^{\mathrm{KX},\nu}_{V,(m)}$ such that $h\colon S\to M^{\mathrm{KX},\nu}_{V,(m)}$ factors through $\mathcal{M}^{\mathrm{KX},\nu}_{V,(m)}$.
We may assume that $h$ is a Galois cover with $G$ the Galois group.
Let $f\colon (\mathcal{X},\epsilon \mathcal{H})\to S$ be the family corresponding to the morphism $S\to \mathcal{M}^{\mathrm{KX},\epsilon}_{V,(m)}$.
Then, $f_*\omega_{\mathcal{X}/S}^{[\mu]}$ is a line bundle $\mathbb{Q}$-linear equivalent to $\mu h^*\overline{\Lambda}_{\mathrm{Hodge}}$.
Since $G$ is a finite group, there exists a positive integer $k_m\in\mu\mathbb{Z}_{>0}$ such that we have a $G$-equivariant linear equivalence $f_*\omega_{\mathcal{X}/S}^{[k_m]}\sim k_mh^*\overline{\Lambda}_{\mathrm{Hodge}}$.
Since $f_*\omega_{\mathcal{X}/S}^{[k_m]}$ is semiample by Condition \ref{ass--b-ss} and $G$ is finite, we may assume that $f_*\omega_{\mathcal{X}/S}^{[k_m]}$ is globally generated by $G$-invariant sections by replacing $k_m$ if necessary.
Thus, $k_m\overline{\Lambda}_{\mathrm{Hodge}}$ is also globally generated.
Replacing $k_m$, we may assume that $|k_m\overline{\Lambda}_{\mathrm{Hodge}}|$ defines $\theta\colon M^{\mathrm{KX},\nu}_{V,(m)}\to \mathbb{P}^{h^0(k_m\nu^*\overline{\Lambda}_{\mathrm{Hodge}})-1}$ such that if we let $\overline{V}_{(m)}^{\mathrm{BB}}$ be the image, then  $\overline{V}_{(m)}^{\mathrm{BB}}$ satisfies conditions (1) and (2)
by \cite[Theorem 2.1.27]{Laz}.

From now on, we will check condition (3).
Let $p\in M^{\mathrm{KX},\nu,\circ}_{V,(m)}$ and $q\in \theta^{-1}\theta(p)$ be closed points.
Since $\theta^{-1}\theta(p)$ is connected by (2), we may assume that there exists an irreducible proper curve $C\subset \theta^{-1}\theta(p)$ such that $p,q\in C$.
Suppose that there exists a morphism $\tilde{C}\to \mathcal{M}^{\mathrm{KX},\nu}_{V,(m)}$ from a proper smooth curve such that the image of $\tilde{C}$ to $M^{\mathrm{KX},\nu}_{V,(m)}$ coincides with $C$. 
Let $(\mathcal{X},\epsilon\mathcal{H})\to \tilde{C}$ be the corresponding family.
Then we claim that $\mathcal{X}_{\tilde{p}}\cong\mathcal{X}_{\tilde{q}}$, where $\tilde{p}$ and $\tilde{q}$ are closed points that are assigned to $p$ and $q$.
Note that the generic fiber of $\mathcal{X}\to \tilde{C}$ is klt.
We see by \cite[Theorem 3.3]{Am} that  $M_{\tilde{C}}\sim_{\mathbb{Q}}0$, where $M_{\tilde{C}}$ is the moduli $\mathbb{Q}$-divisor with respect to the klt--trivial fibration $\mathcal{X}\to \tilde{C}$.
Since any geometric fiber of $\mathcal{X}$ is slc, there exists a quasi finite morphism $U\to \tilde{C}$ such that $\mathcal{X}_U\cong U\times \mathcal{X}_{u}$ over $U$ for some $u\in U$ by \cite[Theorem 4.7]{Am}.
Let $\overline{U}$ be the compactification of $U$. 
Then, the latter can be extended to $\overline{U}\times \mathcal{X}_u$.
By \cite[Theorem 4.2 (i)]{O3}, we see that $\mathcal{X}_{\overline{U}}\cong \overline{U}\times \mathcal{X}_u$.
This shows that $\mathcal{X}_p\cong\mathcal{X}_q\cong \mathcal{X}_u$.
Therefore, $C\subset M^{\mathrm{KX},\nu,\circ}_{V,(m)}$ and hence $\theta^{-1}\theta(M^{\mathrm{KX},\nu,\circ}_{V,(m)})=M^{\mathrm{KX},\nu,\circ}_{V,(m)}$.
Since $\theta$ is a proper surjection, $\theta(M^{\mathrm{KX},\nu,\circ}_{V,(m)})$ is an open subset.

Let $W$ be an arbitrary affine open subscheme of $\theta(M^{\mathrm{KX},\nu,\circ}_{V,(m)})$.
By the above claim and \cite[Lemma 6.2]{Hat23}, it is easy to see that $\tau_m^{-1}\tau_m(\theta^{-1}(W))=\theta^{-1}(W)$.
Since $\tau_m$ is proper, we see that $\tau_m(\theta^{-1}(W))$ is open.
Furthermore, $\tau_m$ contracts the same curves as $\theta$.
Therefore, there exists a set-theoretical map $\varphi\colon V^{\nu}\to\theta(M^{\mathrm{KX},\nu,\circ}_{V,(m)})$ such that $\varphi\circ \tau_m=\theta$.
We see that $\varphi$ is continuous.
It is also easy to see that $\Gamma(W,\mathcal{O}_W)\cong \Gamma(M^{\mathrm{KX},\nu,\circ}_{V,(m)},\mathcal{O})\cong \Gamma (\tau_m(\theta^{-1}(W)),\mathcal{O})$.
By using this, we can construct a morphism $\varphi\colon V^{\nu}\to\theta(M^{\mathrm{KX},\nu,\circ}_{V,(m)})$.
Since $\varphi$ is quasi-finite and separated, we see that $V^{\nu}$ is a quasi-projective scheme.
It is easy to see that $\varphi$ is bijective and hence $\varphi$ is an isomorphism due to \cite[III, Corollary 11.4]{Ha}.

Finally, we claim that such $\overline{V}_{(m)}^{\mathrm{BB}}$ is independent of the choice of $m$.
To show this, it suffices to deal with the following case: 
pick $m_1\in l\mathbb{Z}_{>0}$ such that $m_1|m$.
Then, there exists a canonical isomorphism $\overline{V}_{(m)}^{\mathrm{BB}}\cong \overline{V}_{(m_1)}^{\mathrm{BB}}$.
Indeed, for any sufficiently small $\epsilon>0$, we have a natural morphism $\mathcal{M}^{\mathrm{KX},\frac{m}{m_1}\epsilon}_{V,(m_1)}\to \mathcal{M}^{\mathrm{KX},\epsilon}_{V,(m)}$.
By this morphism, we have canonical morphisms $\psi\colon M^{\mathrm{KX},\nu}_{V,(m_1)}\to M^{\mathrm{KX},\nu}_{V,(m)}$ and $\theta'=\theta\circ\psi \colon M^{\mathrm{KX},\nu}_{V,(m_1)}\to \overline{V}_{(m)}^{\mathrm{BB}}$.
Over $V^{\nu}$, fibers of $\theta'$ are connected.
Since $\overline{V}_{(m)}^{\mathrm{BB}}$ is normal, we have that $\theta'_*\mathcal{O}_{M^{\mathrm{KX},\nu}_{V,(m_1)}}\cong \mathcal{O}_{\overline{V}_{(m)}^{\mathrm{BB}}}$.
Let $\theta_{(m_1)}\colon M^{\mathrm{KX},\nu}_{V,(m_1)}\to \overline{V}_{(m_1)}^{\mathrm{BB}}$ be the canonical morphism.
Let $\Lambda_{(m)}$ and $\Lambda_{(m_1)}$ denote $\nu^*\overline{\Lambda}_{\mathrm{Hodge}}$ on $M^{\mathrm{KX},\nu}_{V,(m)}$ and $M^{\mathrm{KX},\nu}_{V,(m_1)}$, respectively.
It is easy to see that $\psi^*\Lambda_{(m)}\sim_{\mathbb{Q}}\Lambda_{(m_1)}$.
By construction, there are ample $\mathbb{Q}$-line bundles $L_1$ on $\overline{V}^{\mathrm{BB}}_{(m_1)}$ and $L_2$ on $\overline{V}^{\mathrm{BB}}_{(m)}$ such that $\theta_{(m_1)}^*L_2\sim_{\mathbb{Q}}\Lambda_{(m_1)}$ and $\theta'^*L_2\sim_{\mathbb{Q}}\Lambda_{(m_1)}$.
This means that $\overline{V}^{\mathrm{BB}}_{(m_1)}\cong \overline{V}^{\mathrm{BB}}_{(m)}$.
We complete the proof.
\end{proof}

\begin{proof}[Proof of Theorem \ref{thm--canonical--bundle--formula-sin}]
We use notations of Proposition \ref{prop--baily-borel-V^bb}.
First note that $|k_l\nu^*\overline{\Lambda}_{\mathrm{Hodge}}|$ defines a morphism $\theta\colon M^{\mathrm{KX},\nu}_{V,(l)}\to \mathbb{P}^{h^0(k_l\nu^*\overline{\Lambda}_{\mathrm{Hodge}})-1}$ by Proposition \ref{prop--baily-borel-V^bb}.
Let $\Lambda$ be a $\mathbb{Q}$-Cartier $\mathbb{Q}$-divisor on $\overline{V}^{\mathrm{BB}}$ such that $k_l\Lambda\sim \mathcal{O}_{\mathbb{P}^{h^0(k_l\nu^*\overline{\Lambda}_{\mathrm{Hodge}})-1}}(1)|_{\overline{V}^{\mathrm{BB}}}$.
Take $k\in k_l\cdot\mathbb{Z}_{>0}$ such that the pullback of $k\overline{\Lambda}_{\mathrm{Hodge}}$ to $\mathcal{M}^{\mathrm{KX},\nu}_{V,(l)}$ is linearly equivalent to $ k\overline{\lambda}_{\mathrm{Hodge}}$

Let $f\colon (X,\Delta)\to S$ be an lc--trivial fibration over a normal projective variety $S$.
Fix an $f$-ample line bundle $A$ on $X$.
Let $\mathcal{V}_1^{(m)}$ be the inverse image of the image of $V_1$ under $\mathcal{V}^{(m)}\to V^{(m)}$.
By assumption, there exists an open dense subset $U\subset S$ such that $\Delta_s=0$ and $(X_s,A_s)$ belongs to $\mathcal{V}_1^{(m)}(\mathbbm{k})$ for any closed point $s\in U$. 
This yields a morphism $U\to \mathcal{V}_1^{(m)}$.
By Proposition \ref{prop--v^l-stop}, there exists a finite morphism $\mathcal{V}_1^{(l)}\to\mathcal{V}_1^{(m)}$.
Thus, there exists an \'etale morphism $U'\to U$ from a variety such that $X_{U'}$ has a relatively ample line bundle $A'$ over $U'$ such that $\frac{m}{l}A'\sim_{U'}A_{U'}$.
Shrinking $U'$ if necessary, we may assume that there exists an effective relative Cartier divisor $H$ on $X_{U'}$ over $U'$ such that $H\sim A'|_{X_{U'}}$ and $(X_{\bar{s}},\epsilon H_{\bar{s}})$ belongs to $\mathcal{M}^{\mathrm{KX},\epsilon,\circ}_{V,(l)}(\overline{\kappa(s)})$ for any geometric point $\bar{s}\in U'$ and any sufficiently small $\epsilon$.
Note that there exists the canonical morphism $\mathcal{P}\colon U\to V_1$.
Let $\mu\colon\tilde{S}\to S$ be a projective birational morphism from an Ambro model and $M_{\tilde{S}}$ the moduli $\mathbb{Q}$-divisor. 
Taking a compactification and shrinking $U'$ and $U$ if necessary, we can take a projective normal variety $\hat{S}$ containing $U'$ with a morphism $\tilde{\mu}\colon \hat{S}\to \tilde{S}$ that is rationally equivalent to $U'\to U$.
By applying \cite[Lemma 2.51]{HH2}, there exist a projective generically finite morphism $\varphi\colon S'\to \hat{S}$ from a normal variety and $(X',\epsilon H')\to S'\in \mathcal{M}^{\mathrm{KX},\epsilon}_{V,(l)}(S')$ such that $(X'\times_{S}U',H'|_{X'\times_{S}U'})\cong (X_{U'}\times_{U'}\varphi^{-1}(U'),H|_{X_{U'}\times_{U'}\varphi^{-1}(U')})$.
Replacing $S'$ if necessary, we may assume that $\tilde{\mu}\circ\varphi$ is a composition of a birational morphism to a projective normal variety $g'\colon S'\to S''$ and a finite Galois cover $g\colon S''\to \tilde{S}$ with the Galois group $G$.
Then, $(X',\epsilon H')\to S'$ induces a morphism $\tau\colon S'\to M^{\mathrm{KX},\nu}_{V,(l)}$ and let $\theta\colon M^{\mathrm{KX},\nu}_{V,(l)}\to \overline{V}^{\mathrm{BB}}$ be the canonical morphism obtained in Proposition \ref{prop--baily-borel-V^bb}.
We note that $\theta\circ\tau$ can be descent to $\psi\colon S''\to \overline{V}^{\mathrm{BB}}$.
It is easy to see that $\psi$ is $G$-invariant.
Thus, $\psi$ induces a morphism $\widetilde{\mathcal{P}}\colon \tilde{S}\to \overline{V}^{\mathrm{BB}}$ such that $\widetilde{\mathcal{P}}\circ g=\psi$.

Then, by the choice of $k$, we obtain that $$\mathcal{O}_{S'}(k\tau^*\theta^*\Lambda)\cong \mathcal{O}_{S'}(kg'^*g^*M_{\tilde{S}}).$$
The restriction of this isomorphism to $g'^{-1}g^{-1}\mu^{-1}(U)$ coincides with the pullback of $$\mathcal{O}_{U}(k{\mathcal{P}}^*\Lambda)\cong(f|_U)_*\omega^{[k]}_{X_U/U}.$$ 
This means that the above isomorphism on $S'$ is $G$-equivariant and hence we obtain 
\[
\mathcal{O}_{\tilde{S}}(k\widetilde{\mathcal{P}}^*\Lambda)\cong \mathcal{O}_{\tilde{S}}(kM_{\tilde{S}}).
\]
We obtain the proof.
\end{proof}

\section{A morphism from moduli of good minimal models to moduli of quasimaps}
\label{sec--finish}

Fix $d,v\in\mathbb{Z}_{>0}$ and $u\in\mathbb{Q}_{\ne0}$.
Fix an irreducible component $\mathcal{V}\subset \mathcal{M}^{\mathrm{klt,CY}}_{d-1,v}$.
As in Proposition \ref{prop--v^l-stop}, we choose $l\in\mathbb{Z}_{>0}$.
Pick $k\in\mathbb{Z}_{>0}$ as Theorem \ref{thm--canonical--bundle--formula-sin} for $l=m$.
Take a sufficiently large $w\in\mathbb{Q}_{>0}$ such that $\mathfrak{Z}_{d, v,u,w}\to \mathfrak{Z}_{d, v,u}$ is surjective.
Consider $I\in\mathbb{Z}_{>0}$ as Setup \ref{stup--5} and $l_0\in\mathbb{Z}_{>0}$ as Lemma \ref{lem--HH2--de--yatta}.
Replacing $l$ if necessary, we may assume that $l$ is divisible by $Il_0$.
Fix a nonempty normal open subset $V_1\subset V_{\mathrm{sn}}$ as Condition \ref{ass--b-ss}.

For any $p_a\in\mathbb{Z}_{>0}$ and sufficiently divisible $r\in\mathbb{Z}_{>0}$, let $\mathcal{M}_{d,v,u,r,w,p_a}$ be an open and closed substack of $\mathcal{M}_{d,v,u,r,w}$ such that 
$$\mathcal{M}_{d,v,u,r,w,p_a}(S)=\left\{
 f\colon(\mathcal{X},\mathcal{A})\to \mathcal{C}\in \mathcal{M}_{d,v,u,r,w}(S)
\;\middle|
\begin{array}{rl}
&\text{$p_a=p_a(\mathcal{C}_{\bar{s}})$ for any}\\
&\text{geometric point $\bar{s}$.}
\end{array}\right\}$$
for any scheme $S$.
Let $f_{\mathrm{univ}}\colon(\mathcal{U},\mathscr{A})\to\mathcal{C}$ be the universal family over $\mathcal{M}_{d,I^{d-1}\cdot v,u,r,I^d\cdot w,p_a}$.
Let $\pi_{\mathcal{U}}\colon\mathcal{U}\to \mathcal{M}_{d,I^{d-1}\cdot v,u,r,I^d\cdot w,p_a}$ and $\pi_{\mathcal{C}}\colon\mathcal{C}\to \mathcal{M}_{d,I^{d-1}\cdot v,u,r,I^d\cdot w,p_a}$ be the canonical morphisms.
We note that 
\[
\mathcal{Z}=\left\{
 p\in\mathcal{C}
\;\middle|
\begin{array}{rl}
&\text{$(\mathcal{U}_{\bar{p}},\mathscr{A}_{\bar{p}})\in \mathcal{V}^{(m)}(\bar{p})$ corresponds to}\\
&\text{a geometric point of $V_1\subset V^{\mathrm{sn}}$}
\end{array}\right\}
\]
is locally closed.
Since $\pi_{\mathcal{C}}$ is proper and flat, it is easy to see that $\pi_{\mathcal{C}}(\mathcal{Z})$ is also locally closed.
Let $(\mathcal{M}_{d,v,u,w,p_a})_{\mathrm{red}}$ denote the reduced structure of $\mathcal{M}_{d,v,u,r,w,p_a}$.
Set $\mathcal{M}^W_{d,v,u,w,p_a}$ as $$(\mathcal{M}_{d,I^{d-1}\cdot v,u,I^d\cdot w,p_a})_{\mathrm{red}}\times_{\sigma,\mathbf{Pic}_{\mathcal{U}/(\mathcal{M}_{d,I^{d-1}\cdot v,u,I^d\cdot w,p_a})_{\mathrm{red}}},\mu_I}\mathbf{W}^{\mathbb{Q}}\mathbf{Pic}_{\mathcal{U}/(\mathcal{M}_{d,I^{d-1}\cdot v,u,I^d\cdot w,p_a})_{\mathrm{red}}}$$
and $\mathcal{M}^{W,I}_{d,v,u,w,p_a}$ as the reduced structure of the image of the finite morphism $\mathcal{M}^W_{d,v,u,w,p_a}\to (\mathcal{M}_{d,I^{d-1}\cdot v,u,I^d\cdot w,p_a})_{\mathrm{red}}$.
Then, let $\mathcal{M}^{W,I}_{d,v,u,w,p_a,V_1}$ be the reduced locally closed substack $\mathcal{M}^{W,I}_{d,v,u,w,p_a}\cap\pi_{\mathcal{C}}(\mathcal{Z})$.
We set $\mathcal{N}^{W,I}_{d,v,u,w,p_a}$ as the image of $\mathcal{M}^{W,I}_{d,v,u,w,p_a}$ under $\xi\colon\mathcal{M}^{W,I}_{d,v,u,w}\to\mathcal{N}^{W,I}_{d,v,u,w}$ (cf.~Definition \ref{de--Num--consruction}).
Then we see that the following holds.
\begin{lem} \label{lem--N--V_1}  
There exists a unique locally closed reduced substack $\mathcal{N}^{W,I}_{d,v,u,w,p_a,V_1}$ such that $\mathcal{M}^{W,I}_{d,v,u,w,p_a,V_1}\cong\mathcal{N}^{W,I}_{d,v,u,w,p_a,V_1}\times_{\mathcal{N}^{W,I}_{d,v,u,r,w}}\mathcal{M}^{W,I}_{d,v,u,r,w}$.
\end{lem}

\begin{proof}
The uniqueness is obvious.
It is easy to see that $\mathcal{M}^{W,I}_{d,v,u,w,p_a}\cong\mathcal{N}^{W,I}_{d,v,u,w,p_a}\times_{\mathcal{N}^{W,I}_{d,v,u,r,w}}\mathcal{M}^{W,I}_{d,v,u,r,w}$. Therefore, it suffices to find a locally closed reduced substack $\mathcal{N}^{W,I}_{d,v,u,w,p_a,V_1}$ of $\mathcal{N}^{W,I}_{d,v,u,w,p_a}$ such that $\mathcal{M}^{W,I}_{d,v,u,w,p_a,V_1}\cong\mathcal{N}^{W,I}_{d,v,u,w,p_a,V_1}\times_{\mathcal{N}^{W,I}_{d,v,u,w,p_a}}\mathcal{M}^{W,I}_{d,v,u,w,p_a}$. 

Let $\mathcal{W}$ be the scheme theoretic image of $\xi|_{\mathcal{M}^{W,I}_{d,v,u,w,p_a,V_1}}$ with the reduced stack structure by \cite[Tag 0CPU]{Stacks}.
We first show that $\xi^{-1}\xi(\mathcal{M}^{W,I}_{d,v,u,w,p_a,V_1})=\mathcal{M}^{W,I}_{d,v,u,w,p_a,V_1}$ set-theoretically.       
For any point $x\in \mathcal{M}^{W,I}_{d,v,u,w,p_a,V_1}$, we can take a morphism $\mathrm{Spec}\,\Omega\to \mathcal{M}^{W,I}_{d,v,u,w,p_a,V_1}$ whose image coincides with $\{x\}$, where $\Omega$ is an algebraically closed field.
We see that $\mathrm{Spec}\,\Omega\times_{\mathcal{N}^{W,I}_{d,v,u,w,p_a}}\mathcal{M}^{W,I}_{d,v,u,w,p_a}\to \mathcal{M}^{W,I}_{d,v,u,w,p_a}$ factors through $\mathcal{M}^{W,I}_{d,v,u,w,p_a,V_1}$ by definition.
This means that $\xi^{-1}\xi(\mathcal{M}^{W,I}_{d,v,u,w,p_a,V_1})=\mathcal{M}^{W,I}_{d,v,u,w,p_a,V_1}$.
By the properness and smoothness of $\xi$, it is easy to see that the scheme theoretic image of $\mathcal{M}^{W,I}_{d,v,u,w,p_a,V_1}\hookrightarrow \mathcal{M}^{W,I}_{d,v,u,w}$ coincides with $\xi^{-1}(\mathcal{W})$ (see \cite[Tag 0CML]{Stacks}).
Since $\mathcal{M}^{W,I}_{d,v,u,w,p_a,V_1}$ is locally closed, we see that $\xi|_{\mathcal{M}^{W,I}_{d,v,u,w,p_a,V_1}}\colon \mathcal{M}^{W,I}_{d,v,u,w,p_a,V_1}\to\mathcal{W}$ is also smooth.
Therefore, $$\mathcal{N}^{W,I}_{d,v,u,w,p_a,V_1}:=\xi|_{\mathcal{M}^{W,I}_{d,v,u,w,p_a,V_1}}(\mathcal{M}^{W,I}_{d,v,u,w,p_a,V_1})$$ is an open substack of $\mathcal{W}$.
By construction, it is easy to see that $$\mathcal{M}^{W,I}_{d,v,u,w,p_a,V_1}\cong\mathcal{N}^{W,I}_{d,v,u,w,p_a,V_1}\times_{\mathcal{N}^{W,I}_{d,v,u,w,p_a}}\mathcal{M}^{W,I}_{d,v,u,w,p_a}.$$

We complete the proof.
\end{proof}

We also write $\xi|_{\mathcal{M}^{W,I}_{d,v,u,w,p_a,V_1}}$ as $\xi$.

\begin{de}[Associated Quasimap]
Let $f\colon (X,A)\to C$ be a pair of a uniformly adiabatically K-stable klt-trivial fibration over a smooth curve and an $f$-ample line bundle that corresponds to a closed point of $\mathcal{M}_{d,v,u,r,w,p_a,V_1}$.
Choose a sufficiently large and divisible $a\in\mathbb{Z}_{>0}$.
Let $\iota\colon \overline{V}^{\mathrm{BB}}\hookrightarrow \mathbb{P}^L$ be the closed immersion induced by the complete linear system $|\Lambda_{\mathrm{Hodge}}^{\otimes a}|$ and $\mathrm{Cone}(\overline{V}^{\mathrm{BB}})$ the affine cone with respect to $\iota$ (see Definition \ref{de--hodge}).
A quasimap $q\colon C\to [\mathrm{Cone}(\overline{V}^{\mathrm{BB}})/\mathbb{G}_{m,\mathbbm{k}}]$ is called {\it the quasimap associated with $f$} if there exists an open dense subset $U\subset C$ such that $(X_p,A_p)$ corresponds to the point of $V_1$, and $q|_U$ coincides with the canonical morphism $U\to V_1\subset \overline{V}^{\mathrm{BB}}$ and $\mathrm{Bs}(q)=aB$, where $B$ is the discriminant divisor with respect to $f$.
It is easy to check by \cite[Theorem 3.1]{A} and the argument of \cite[Definition 3.22]{HH2} that $q$ is K-stable then. 
Note that if $u>0$, then $q$ is nothing but a stable quasimap by Proposition \ref{prop--K-stable-quasimaps--stability}.

Let $S$ be a normal quasi-projective variety.
Let $f\colon(\mathcal{X},\mathcal{A})\to \mathcal{C}\in \mathcal{M}_{d,v,u,r,w,p_a,V_1}(S)$ be an arbitrary object such that $\mathcal{A}$ is a genuine line bundle.
A family of K-stable quasimaps $q\in \mathcal{C}\to[\mathrm{Cone}(\overline{V}^{\mathrm{BB}})_S/\mathbb{G}_{m,S}]\in \mathcal{M}^{\mathrm{Kss,qm}}_{a(u-2p_a+2),\frac{1}{a},u,\iota}(S)$, where $\iota\colon \overline{V}^{\mathrm{BB}}\hookrightarrow \mathbb{P}^L$ is the closed immersion induced by the complete linear system $|\Lambda_{\mathrm{Hodge}}^{\otimes a}|$ for some sufficiently large and divisible $a\in\mathbb{Z}_{>0}$, is called {\it the family of stable quasimaps associated with $f$} if $q_{\bar{s}}$ is the quasimap associated with $f_{\bar{s}}$ for any geometric point $\bar{s}\in S$. 
\end{de}

The following can be shown by Theorem \ref{thm--canonical--bundle--formula-sin} and the same argument of the proof of \cite[Theorem 4.4]{HH2}.

\begin{thm}\label{thm--previous-hh}
    There exist a sufficiently divisible positive integer $k\in\mathbb{Z}_{>0}$ and a morphism $\beta_{v,w}\colon(\mathcal{M}_{d,v,u,w,p_a,V_1})^{\mathrm{sn}}\to \mathcal{M}^{\mathrm{Kss,qm}}_{k(u-2p_a+2),\frac{1}{k},u,\iota}$, where $(\mathcal{M}_{d,v,u,w,p_a,V_1})^{\mathrm{sn}}$ is the seminormalization of $(\mathcal{M}_{d,v,u,w,p_a,V_1})_{\mathrm{red}}$ and $\iota\colon \overline{V}^{\mathrm{BB}}\hookrightarrow \mathbb{P}^L$ is the closed immersion induced by the complete linear system $|\Lambda_{\mathrm{Hodge}}^{\otimes k}|$. 
\end{thm}

For this, we show the following.

\begin{thm}\label{thm--Fujino's--period--mapping--theory}
    Fix $l$ and $k$ as the first part of this section. 
    Then the following holds.
    Let $S$ be a smooth affine variety and $f\colon (\mathcal{X},\mathcal{A})\to\mathcal{C}$ an object of $\mathcal{M}_{d,v,u,r,w,p_a,V_1}(S)$, where $\mathcal{C}$ is as condition (v) of Definition \ref{defn--HH-moduli}.
    Suppose that $\mathcal{A}$ is represented by a line bundle on $\mathcal{X}$.
    Let $\mathcal{C}_{\mathrm{klt}}$ be the largest open subset of $\mathcal{C}$ such that all geometric fibers of $f$ over $\mathcal{C}_{\mathrm{klt}}$ correspond to points of $V_1$. 
    Let $\mu^\circ\colon \mathcal{C}_{\mathrm{klt}}\to V_1$ be the morphism induced by the polarized locally stable family $(\mathcal{X}_{\mathcal{C}_{\mathrm{klt}}},\mathcal{A}|_{\mathcal{C}_{\mathrm{klt}}})$. 
    Take a projective birational morphism  $\rho\colon\widetilde{\mathcal{C}}\to \mathcal{C}$ from an Ambro model with respect to $f$, and let $\widetilde{\mathcal{M}}$ and $\widetilde{\mathcal{B}}$ be the moduli $\mathbb{Q}$-divisor and the discriminant $\mathbb{Q}$-divisor on $\mathcal{C}$, respectively.
    Then, after replacing $\widetilde{\mathcal{C}}$ with a higher birational model in a suitable way if necessary, we have the following:
    \begin{enumerate}
       \item \label{thm--Fujino's--period--mapping--theory-(1)} $k\widetilde{\mathcal{M}}$ and $k\widetilde{\mathcal{B}}$ are both $\mathbb{Z}$-divisors, and
        \item \label{thm--Fujino's--period--mapping--theory-(2)} putting $L:=h^0(\overline{V}^{\mathrm{BB}},\Lambda^{\otimes k}_{\mathrm{Hodge}})-1$, then there exist $L+1$ sections $$\widetilde{\varphi}_0,\ldots,\widetilde{\varphi}_L\in H^0(\widetilde{\mathcal{C}},\mathcal{O}_{\widetilde{\mathcal{C}}}(k\widetilde{\mathcal{M}}))$$ that generate $k\widetilde{\mathcal{M}}$ and define a morphism $\widetilde{\mu}\colon\widetilde{\mathcal{C}}\to \mathbb{P}^L$ such that $\widetilde{\mu}^*\mathcal{O}(1)\sim k\widetilde{\mathcal{M}}$, $\widetilde{\mu}$ factors $\iota\colon \overline{V}^{\mathrm{BB}}\hookrightarrow \mathbb{P}^L$ and this is equivalent to $\iota \circ \mu^\circ \circ \rho$ as rational maps.  
        Here, $\iota\colon \overline{V}^{\mathrm{BB}}\to \mathbb{P}^L$ is induced by the complete linear system $|\Lambda^{\otimes k}_{\mathrm{Hodge}}|$. 
    \end{enumerate}
\end{thm}

\begin{proof}
We follow the proof of \cite[Theorem 4.2]{HH2} heavily.
Since $S$ is smooth, we note that $\mathcal{C}$ is a quasi-projective variety and smooth, and $K_{\mathcal{X}}$ is $\mathbb{Q}$-Cartier. 
Let $\mathcal{C} \hookrightarrow \mathcal{C}^{c}$ be an open immersion to a normal projective variety $\mathcal{C}^{c}$. 
Applying \cite[Corollary 2.11]{HH2}, we get a diagram
$$
\xymatrix
{
\mathcal{X} \ar[d]_{f} \ar@{^{(}->}[r] & \mathcal{X}^{c} \ar[d]^{f^{c}} 
\\
\mathcal{C} \ar@{^{(}->}[r] & \mathcal{C}^{c}  
}
$$
and an $f^{c}$-ample $\mathbb{Q}$-divisor $\mathcal{A}^{c}$ on $\mathcal{X}^{c}$ such that $f^{c} \colon (\mathcal{X}^{c},0) \to \mathcal{C}^{c}$ is a klt-trivial fibration and $\mathcal{A}^{c}|_{\mathcal{X}}=\mathcal{A}$. 
Take a sufficiently divisible $m'\in\mathbb{Z}_{>0}$ such that $m'\mathcal{A}^{c}$ is Cartier and $m'$ is divisible by $l$, and a projective birational morphism $\widetilde{\rho} \colon \widetilde{\mathcal{C}}^{c} \to \mathcal{C}^{c}$ from an Ambro model $\widetilde{\mathcal{C}}^{c}$. 
Fixing $\widetilde{\mathcal{C}}^{c}$ and replacing $\widetilde{\mathcal{C}}$ if necessary, we may assume that $\widetilde{\mathcal{C}}$ is an open subscheme of $\widetilde{\mathcal{C}}^{c}$. 
By Theorem \ref{thm--canonical--bundle--formula-sin} applied to $f^{c} \colon \mathcal{X}^{c} \to \mathcal{C}^{c}$ and $m'\mathcal{A}^{c}$, 
there exists a morphism $\mu\colon\widetilde{\mathcal{C}}^{c} \to  \overline{V}^{\mathrm{BB}}$ 
such that 
\begin{equation*}
\mu^*\Lambda^{\otimes k}_{\mathrm{Hodge}}\sim \mathcal{O}_{\widetilde{\mathcal{C}}^{c}}(k\widetilde{\mathcal{M}}^{c}),
\end{equation*}
where $\widetilde{\mathcal{M}}^{c}$ is the moduli $\mathbb{Q}$-divisor defined on $\widetilde{\mathcal{C}}^{c}$. 
Put $L:=h^0(\overline{V}^{\mathrm{BB}},\Lambda^{\otimes l}_{\mathrm{Hodge}})-1$ and $\widetilde{\varphi}_0,\ldots,\widetilde{\varphi}_L$ as the pullbacks of the coordinate functions of $\mathbb{P}^L$. 
Then (\ref{thm--Fujino's--period--mapping--theory-(2)}) of Theorem \ref{thm--Fujino's--period--mapping--theory} follows from the same argument as the proof of \cite[Theorem 4.2]{HH2}. 

(\ref{thm--Fujino's--period--mapping--theory-(1)}) follows from Lemma \ref{lem--HH2--de--yatta} and the same argument as the proof of \cite[Theorem 4.2]{HH2}.
We finish the proof.  
\end{proof}

\begin{thm}\label{thm--const--qmaps--from--CY--fib}
Fix $k$ as above. 
    Let $S$ be a quasi-projective seminormal scheme of finite type over $\mathbbm{k}$ and let $f\colon(\mathcal{X},\mathcal{A})\to\mathcal{C}$ be an object of $\mathcal{M}_{d,v,u,r,w,p_a,V_1}(S)$, where $\mathcal{A}$ is represented by a line bundle and $\mathcal{C}$ is as condition (v) of Definition \ref{defn--HH-moduli}.
Then, there exists the following object $q\colon \mathcal{C}\to [\mathrm{Cone}(\overline{V}^{\mathrm{BB}})_S/\mathbb{G}_{m,S}]$ of $\mathcal{M}^{\mathrm{Kss,qm}}_{k(u-2p_a+2),\frac{1}{k},u,\iota}(S)$ uniquely up to isomorphism satisfying the following. 
\begin{enumerate}
\item Let $\mathscr{L}$ be the associated line bundle on $\mathcal{C}$ with $q$. Then, $\frac{|u|}{u}f^*\mathscr{L}\sim k(K_{\mathcal{X}/S}-f^*K_{\mathcal{C}/S})$.
\item For any geometric point $\bar{s}\in S$, let $B_{\bar{s}}$ be the discriminant $\mathbb{Q}$-divisor of $f_{\bar{s}}$.
Then $kB_{\bar{s}}$ is the fixed part of $q_{\bar{s}}$ and $q_{\bar{s}}$ coincides with the quasimap associated to $f_{\bar{s}}$.
\end{enumerate}       
\end{thm}

\begin{proof}
 The same argument as in the proof of \cite[Theorem 4.6 and Corollary 4.7]{HH2} also works in this case due to Theorem \ref{thm--Fujino's--period--mapping--theory}. 
 We leave the details to the reader.
\end{proof}

\begin{proof}[Proof of Theorem \ref{thm--previous-hh}]
By applying Theorem \ref{thm--const--qmaps--from--CY--fib}, we obtain this theorem in the same way as \cite[Theorem 4.4]{HH2}.    
\end{proof}

Set $\beta:=\beta_{I^{d-1}\cdot v,I^d\cdot w}|_{\mathcal{M}^{W,I}_{d,v,u,w,p_a,V_1}}$.
Then the following holds.

\begin{thm}\label{thm--quasifinite}
Fix $k$, $w$ and $I$ as Theorem \ref{thm--previous-hh}.
    Then there exists a morphism $$\alpha\colon(\mathcal{N}^{W,I}_{d,v,u,w,p_a,V_1})^{\mathrm{sn}}\to \mathcal{M}^{\mathrm{Kss,qm}}_{k(u-2p_a+2),\frac{1}{k},u,\iota}$$ such that $\alpha\circ\xi^{\mathrm{sn}}\cong\beta$, where $(\mathcal{N}^{W,I}_{d,v,u,w,p_a,V_1})^{\mathrm{sn}}$ is the seminormalization of $\mathcal{N}^{W,I}_{d,v,u,w,p_a,V_1}$ and $\xi^{\mathrm{sn}}\colon (\mathcal{M}^{W,I}_{d,v,u,w,p_a,V_1})^{\mathrm{sn}}\to(\mathcal{N}^{W,I}_{d,v,u,w,p_a,V_1})^{\mathrm{sn}}$ is the morphism induced by $\xi$.
\end{thm}

Before showing Theorem \ref{thm--quasifinite}, we check that $q$ depends only on the algebraic equivalence of $\mathcal{A}$ in Theorem \ref{thm--const--qmaps--from--CY--fib}.

\begin{lem}\label{lem--invariant--linear--equivalence}
In the situation of Theorem \ref{thm--const--qmaps--from--CY--fib}, suppose that there exists another line bundle $\mathcal{A}'$ on $\mathcal{X}$ such that $f\colon(\mathcal{X},\mathcal{A}')\to\mathcal{C}$ belongs to $\mathcal{M}_{d,v,u,r,w,p_a,V_1}(S)$.
Suppose that $\mathcal{A}^{\otimes-1}\otimes\mathcal{A}'\in\mathbf{Pic}^0_{\mathcal{X}/S}$.
Let $q'$ be the family of quasimaps obtained by applying Theorem \ref{thm--const--qmaps--from--CY--fib} to $f\colon(\mathcal{X},\mathcal{A})\to\mathcal{C}$.
Then $q=q'$.
\end{lem}

\begin{proof}
 By the uniqueness of $q'$, it suffices to show that $q$ also satisfies the conditions in Theorem \ref{thm--quasifinite} for $f\colon(\mathcal{X},\mathcal{A}')\to\mathcal{C}$.
 It is easy to see that $q$ satisfies condition (1) of Theorem \ref{thm--quasifinite} for $f\colon(\mathcal{X},\mathcal{A}')\to\mathcal{C}$.
 To check that condition (2) is satisfied by $q$ for $f\colon(\mathcal{X},\mathcal{A}')\to\mathcal{C}$, we may assume that $S=\mathrm{Spec}\,\mathbbm{k}$ and it is enough to show that $q=q'$.
 Since $\mathrm{Bs}(q)=kB=\mathrm{Bs}(q')$, where $B$ is the discriminant $\mathbb{Q}$-divisor of $f$, and the morphisms $\mathcal{C}\to \overline{V}^{\mathrm{BB}}$ induced by $(\mathcal{X},\mathcal{A})$ and $(\mathcal{X},\mathcal{A}')$ are the same by \cite[Lemma 6.2]{Hat23}, we see that $q=q'$ in this case.
 We complete the proof.
\end{proof}

\begin{proof}[Proof of Theorem \ref{thm--quasifinite}]
Let $f\colon S\to (\mathcal{M}^{W,I}_{d,v,u,w,p_a,V_1})^{\mathrm{sn}}$ be an arbitrary \'etale surjective morphism from a scheme.
Then $S$ is seminormal.
Since $\xi$ is smooth, we see that $(\mathcal{M}^{W,I}_{d,v,u,w,p_a,V_1})^{\mathrm{sn}}=\mathcal{M}^{W,I}_{d,v,u,w,p_a,V_1}\times_{\mathcal{N}^{W,I}_{d,v,u,w,p_a,V_1}}(\mathcal{N}^{W,I}_{d,v,u,w,p_a,V_1})^{\mathrm{sn}}$ and the natural morphism $\xi^{\mathrm{sn}}\circ f$ is smooth.
Let $p_1, p_2\colon S\times_{(\mathcal{N}^{W,I}_{d,v,u,w,p_a,V})^{\mathrm{sn}}} S\to S$ be the first and second projections.
Let $R:=S\times_{(\mathcal{N}^{W,I}_{d,v,u,w,p_a,V_1})^{\mathrm{sn}}} S$, then we see that $R$ is naturally an \'etale groupoid and $(\mathcal{N}^{W,I}_{d,v,u,w,p_a,V_1})^{\mathrm{sn}}\cong [S/R]$.
We note that if $\beta\circ f\circ p_1$ and $\beta\circ f\circ p_2$ are canonically isomorphic with compatibility conditions as \cite[Tag 044U]{Stacks}, then there exists a unique morphism $\alpha\colon (\mathcal{N}^{W,I}_{d,v,u,w,p_a,V_1})^{\mathrm{sn}}\to \mathcal{M}^{\mathrm{Kss,qm}}_{l(u-2p_a+2),\frac{1}{l},u,\iota}$ up to isomorphism such that $\alpha\circ \xi^{\mathrm{sn}} $ is isomorphic to $\beta$.
We can see that $\beta\circ f\circ p_1$ and $\beta\circ f\circ p_2$ are canonically isomorphic as follows.
Consider $S\times_{(\mathcal{N}^{W,I}_{d,v,u,w,p_a,V_1})^\mathrm{sn}}(\mathcal{M}^{W,I}_{d,v,u,w,p_a,V_1})^{\mathrm{sn}}$ and let $q_1\colon S\times_{(\mathcal{N}^{W,I}_{d,v,u,w,p_a,V_1})^{\mathrm{sn}}}(\mathcal{M}^{W,I}_{d,v,u,w,p_a,V_1})^{\mathrm{sn}} \to S$ and $q_2\colon S\times_{(\mathcal{N}^{W,I}_{d,v,u,w,p_a,V_1})^{\mathrm{sn}}}(\mathcal{M}^{W,I}_{d,v,u,w,p_a,V})^{\mathrm{sn}}\to(\mathcal{M}^{W,I}_{d,v,u,w,p_a,V_1})^{\mathrm{sn}}$ be the projections.
Then, it suffices to show that $\beta\circ f\circ q_1$ and $\beta \circ q_2$ are canonically isomorphic.
Let $(\mathcal{U},\mathcal{A})$ be a family over $(\mathcal{M}^{W,I}_{d,v,u,w,p_a,V_1})^{\mathrm{sn}}$ that is the pullback of the universal polarized family over $\mathcal{M}_{d,I^{d-1}\cdot v,u,I^d\cdot w,p_a,V_1}$.
Let $(\mathcal{U}_1,\mathcal{A}_1)$ and $(\mathcal{U}_2,\mathcal{A}_2)$ be the pullback of $(\mathcal{U},\mathcal{A})$ under $f\circ q_1$ and $q_2$. 
By replacing $S$ with its \'etale cover, we may assume that $\mathcal{A}$ is a line bundle and then $\mathcal{A}_1$ and $\mathcal{A}_2$ are line bundles.
By construction of $\mathcal{N}^{W,I}_{d,v,u,w,p_a,V}$, we see that $\mathcal{U}_1$ and $\mathcal{U}_2$ are canonically isomorphic, and $\mathcal{A}_1-\mathcal{A}_2\in\mathbf{Pic}^0_{\mathcal{U}_1/S}$ via this identification. 
Thus, the corresponding families of quasimaps by $\beta$ are canonically isomorphic by Lemma \ref{lem--invariant--linear--equivalence}.
We can easily check the condition (3) of \cite[Tag 044U]{Stacks} for the isomorphism between $\beta\circ f\circ p_1$ and $\beta\circ f\circ p_2$.
By this, we can see that there exists the desired $\alpha$.
\end{proof}

The following derives from \cite[Proposition 4.10]{HH2}.

\begin{lem}\label{lem--isotrivial}
Let $f\colon (\mathcal{X},\mathcal{A})\to D\times C\in \mathcal{M}_{d,I^{d-1}\cdot v,u,r,I^d\cdot w,p_a}(C)$ be an object such that $\mathcal{A}$ is a line bundle, where $C$ is an affine smooth curve.    
For any closed point $p\in C$, suppose that $f_p$ defines a closed point of $(\mathcal{M}^{W,I}_{d,v,u,w,p_a,V_1})^{\mathrm{sn}}$.
Let $q\colon D\times C\to [\mathrm{Cone}(\overline{V}^{\mathrm{BB}})_C/\mathbb{G}_{m,C}]$ be the associated family of quasimaps.
Suppose further that $q$ is isotrivial.
Let $f$ also denote the morphism $C\to(\mathcal{M}^{W,I}_{d,v,u,w,p_a,V_    1})^{\mathrm{sn}}$ induced by $f$.
Then, the morphism $C\to (N^{W,I}_{d,v,u,w,p_a,V_1})^{\mathrm{sn}}$ induced by $\xi^{\mathrm{sn}}\circ f$ is constant.
\end{lem}

\begin{proof}
As \cite[Proposition 4.11]{HH2}, we can deduce that there exists a nonempty open subset $W\subset C$ such that $\mathcal{X}\times_CW\cong \mathcal{X}_w\times W$ for some closed point $w\in W$.
By Corollary \ref{cor--N--closedpoint}, we obtain the proof.
\end{proof}

By Lemma \ref{lem--isotrivial}, we deduce the quasi-finiteness as \cite[Theorem 4.8]{HH2}.

\begin{prop}\label{prop-q-proj--N}
Let $N^{W,I}_{d,v,u,w,p_a,V_1}$ be the coarse moduli space of $\mathcal{N}^{W,I}_{d,v,u,w,p_a,V_1}$ and $$\overline{\alpha}\colon (N^{W,I}_{d,v,u,w,p_a,V_1})^{\mathrm{sn}}\to M^{\mathrm{Kps,qm}}_{k(u-2p_a+2),\frac{1}{k},u,\iota}$$ the morphism induced by $\alpha$.
Here, we note that $(N^{W,I}_{d,v,u,w,p_a,V_1})^{\mathrm{sn}}$ is the seminormalization of $N^{W,I}_{d,v,u,w,p_a,V_1}$ and the coarse moduli space of $(\mathcal{N}^{W,I}_{d,v,u,w,p_a,V_1})^{\mathrm{sn}}$. 
    Then $\overline{\alpha}$ is a quasi-finite morphism and $(N^{W,I}_{d,v,u,w,p_a,V_1})^{\mathrm{sn}}$ is a quasi-projective scheme.
\end{prop}

\begin{proof}
    Recall that $N^{W,I}_{d,v,u,w,p_a,V_1}$ exists as a separated algebraic space of finite type by \cite{KeM}.
    Therefore, $\overline{\alpha}$ is separated.
    Note that if $\overline{\alpha}$ is quasi-finite, then $(N^{W,I}_{d,v,u,w,p_a,V_1})^{\mathrm{sn}}$ is quasi-affine over a projective scheme $M^{\mathrm{Kps,qm}}_{k(u-2p_a+2),\frac{1}{k},u,\iota}$ (cf.~Theorem \ref{thm--qm--cm--positive}) by \cite[Theorem 7.2.10]{Ols}.
    By this, it suffices to show that $\overline{\alpha}$ is quasi-finite.
    Assume the contrary.
    Then, there exists a morphism $\varphi\colon C\to (\mathcal{N}^{W,I}_{d,v,u,w,p_a,V_1})^{\mathrm{sn}}$ from a smooth affine curve such that the induced morphism $\overline{\varphi}\colon C\to \mathcal{M}^{\mathrm{Kss,qm}}_{k(u-2p_a+2),\frac{1}{k},u,\iota}$ is constant by the same argument of the proof of \cite[Theorem 4.8]{HH2}.
    Let $q\colon \mathcal{D}\to [\mathrm{Cone}(\overline{V}^{\mathrm{BB}})\times C/\mathbb{G}_{m,C}]$ be the family of quasimaps corresponding to $\overline{\varphi}$.
    Note that $\mathcal{D}\to C$ is isotrivial.
    If any geometric fiber of $\mathcal{D}\to C$ is isomorphic to $\mathbb{P}^1$, then $\mathcal{D}\to C$ is a ruled surface and then we may assume that $\mathcal{D}\cong \mathbb{P}^1\times C$ by replacing $C$ with its \'etale cover.
If any geometric fiber of $\mathcal{D}\to C$ is elliptic, we see that $\mathcal{D}\to C$ has a horizontal divisor.
Taking the Stein factorization of this divisor, we may assume that $\mathcal{D}\to C$ has a section $\mathcal{S}$.
Then $(\mathcal{D},\mathcal{S})$ is a stable family.
Therefore, replacing $C$ with its \'etale cover, we see that $\mathcal{D}\cong C\times \mathcal{D}_c $ for some $c\in C$ because the automorphism group of $(\mathcal{D}_c,\mathcal{S}_c)$ is finite.
In the case where $g(\mathcal{D}_c)\ge2$ for some $c\in C$, replacing $C$ with its \'etale cover, we may assume that $\mathcal{D}\cong C\times \mathcal{D}_c $ for some $c\in C$ as the above argument.
In all the cases, we can apply Lemma \ref{lem--isotrivial} to obtain the quasi-finiteness of $\overline{\alpha}$.
Therefore, $(N^{W,I}_{d,v,u,w,p_a,V_1})^{\mathrm{sn}}$ is quasi-projective.
We obtain the proof.
\end{proof}

\begin{setup}\label{stup--112}
Suppose that $u>0$.
    Let $f\colon (\mathcal{X},\mathcal{A})\to \mathcal{C}$ be an object of $\mathcal{M}_{d,I^{d-1}\cdot v,u,r,I^d\cdot w}(S)$.
    Suppose that $\mathcal{A}$ is a line bundle and $S$ is a normal quasi-projective variety.
Let $\pi_{\mathcal{C}}\colon \mathcal{C}\to S$ and $\pi_{\mathcal{X}}\colon \mathcal{X}\to S$ be the canonical morphisms.
By the choice of $l_0$, there exists $\mathcal{L}$ a $\pi_{\mathcal{C}}$-ample line bundle on $\mathcal{C}$ such that $f^*\mathcal{L}\sim\omega_{\mathcal{X}/S}^{[l_0]}$ independent of the choice of $f$.
Then, there exist $\mathbb{Q}$-divisors $D_0,D_1,D_2,D_3$ and $D_4$ on $S$ such that 
\begin{equation}\label{eq--decomp}
\lambda_{\mathrm{CM},\pi_{\mathcal{X}},(\mathcal{X},\mathcal{A}+q'f^*\mathcal{L})}^{\otimes (\mathcal{A}_t^d+dq'\mathcal{A}_t^{d-1}\cdot \mathcal{L}_t)^2}=q'^4D_4+q'^3D_3+q'^2D_2+q'D_1+D_0,     
\end{equation}
where $t\in S$ is an arbitrary point for any sufficiently large $q'\in\mathbb{Q}_{>0}$.
By the proof of \cite[Proposition 2.63]{HH2}, we note that $D_4\sim_{\mathbb{Q}}D_3\sim_{\mathbb{Q}}0$.

Suppose further that a semisimple linear algebraic group $G$ acts on $(\mathcal{X},\mathcal{A})$ and $S$ equivariantly.
Then $\lambda_{\mathrm{CM},\pi_{\mathcal{X}},(\mathcal{X},\mathcal{A}+q'f^*\mathcal{L})}$ admits a natural $G$-action for any $q'$ (see \cite[Section 2.5.1]{HH2}).
By the uniqueness of the decomposition \eqref{eq--decomp} as $\mathbb{Q}$-divisors, note that $D_0$, $D_1$, and $D_2$ are $\mathbb{Q}$-Cartier $\mathbb{Q}$-divisors with natural $G$-action.
\end{setup}

\begin{prop}\label{prop--approx--CM}
Under the situation of Setup \ref{stup--112}, suppose that a semisimple linear algebraic group $G$ acts on $(\mathcal{X},\mathcal{A})$ and $S$ equivariantly.
Then, we obtain a $G$-equivariant $\mathbb{Q}$-linear equivalence $D_2\sim_{\mathbb{Q}}\frac{dl_0 v}{2Iw}D_1$.

Furthermore, let $q\colon \mathcal{C}\to [\mathrm{Cone}(\overline{V}^{\mathrm{BB}})_S/\mathbb{G}_{m,S}]$ be the object of $\mathcal{M}^{\mathrm{Kss,qm}}_{k(u-2p_a+2),\frac{1}{k},u,\iota}(S)$ defined by the morphism $\beta_{I^{d-1}\cdot v,I^d\cdot w}$.
Then there exists a $G$-equivariant $\mathbb{Q}$-linear equivalence $\gamma D_2\sim_{\mathbb{Q}}\lambda_{\mathrm{CM},q}$ for some $\gamma\in\mathbb{Q}_{>0}$.
\end{prop}

\begin{proof}
By the proof of \cite[Proposition 2.63]{HH2}, we can see that 
\begin{align*}
    D_1&\sim_{\mathbb{Q}}\frac{d\mathcal{A}_s^d}{l_0}(\pi_{\mathcal{X}})_*(\mathcal{A}^{d-1}\cdot\mathcal{L}^2),\\
    D_2&\sim_{\mathbb{Q}}\frac{d^2\mathcal{A}_s^{d-1}\cdot\mathcal{L}_s}{2l_0}(\pi_{\mathcal{X}})_*(\mathcal{A}^{d-1}\cdot\mathcal{L}^2)
\end{align*}
by a standard calculation, where $s\in S$ is a closed point.
Then, we obtain $D_2\sim_{\mathbb{Q}}\frac{dl_0 v}{2Iw}D_1$. By \cite[Proposition 1.4]{GIT}, the $\mathbb{Q}$-linear equivalence preserves the $G$-actions.

The latter assertion follows from the same argument of \cite[Lemma 4.12]{HH2}.
\end{proof}
Under Setup \ref{stup--112}, let $\Lambda_{\mathrm{CM},t}$ also denote the pullback of $\Lambda_{\mathrm{CM},t}$ on $M_{d,I^{d-1}\cdot v,u,r,I^d\cdot w}$ to $(M^{W,I}_{d,v,u,w,p_a,V_1})^\nu$.
By the same argument as the proof of \cite[Proposition 2.63]{HH2}, we see that $$\frac{1}{(dl_0uvI^{d-1})^2}D_2\sim_{\mathbb{Q}}\Lambda_{\mathrm{CM},\infty}|_{(M^{W,I}_{d,v,u,w,p_a,V_1})^\nu},$$ where $\Lambda_{\mathrm{CM},\infty}$ is defined on $(M_{d,I^{d-1}\cdot v,u,I^d\cdot w})^\nu$ as Definition \ref{de--CM--level}.
\begin{thm}\label{thm--ampleness-CM}
$(M^{W,I}_{d,v,u,w,p_a,V_1})^{\mathrm{sn}}$ is a quasi-projective scheme.
Furthermore, $\Lambda_{\mathrm{CM},t}$ is ample on $(M^{W,I}_{d,v,u,w,p_a,V_1})^\nu$ for any sufficiently large $t\in\mathbb{Q}_{>0}$ and $\Lambda_{\mathrm{CM},\infty}|_{(M^{W,I}_{d,v,u,w,p_a,V_1})^\nu}$ is the pullback of some ample $\mathbb{Q}$-line bundle on $(N^{W,I}_{d,v,u,w,p_a,V_1})^\nu$.
\end{thm}

\begin{proof}
First, we deal with the case where $u>0$.
Since $(M^{W,I}_{d,v,u,w,p_a,V_1})^{\mathrm{sn}}$ is projective over $(N^{W,I}_{d,v,u,w,p_a,V_1})^{\mathrm{sn}}$ by Theorem \ref{thm--rel--ample} and $(N^{W,I}_{d,v,u,w,p_a,V_1})^{\mathrm{sn}}$ is quasi-projective by Proposition \ref{prop-q-proj--N}, we have the first assertion.
For the second assertion, we note that $\Lambda_{\mathrm{CM},t}$ is $\xi^\nu$-ample by Theorem \ref{thm--rel--ample} for any sufficiently large $t>0$.
Let $\overline{\beta}^\nu\colon (M^{W,I}_{d,v,u,w,p_a,V_1})^{\nu}\to M^{\mathrm{Kps,qm}}_{k(u-2p_a+2),\frac{1}{k},u,\iota}$ be the canonical morphism induced by $\beta$.
Let $\bar{\xi}^\nu\colon (M^{W,I}_{d,v,u,w,p_a,V_1})^{\nu}\to(N^{W,I}_{d,v,u,w,p_a,V_1})^{\nu}$ be the morphism induced by $\xi$.
Furthermore, we set the morphism $\overline{\alpha}^\nu\colon (N^{W,I}_{d,v,u,w,p_a,V_1})^{\nu}\to M^{\mathrm{Kps,qm}}_{k(u-2p_a+2),\frac{1}{k},u,\iota}$ induced by $\alpha$.
Note that the normalization morphisms of algebraic spaces of finite type over a field are finite morphisms by \cite[Tag 07U4]{Stacks} and the fact that schemes of finite type over a field are Nagata (cf.~\cite[Section 12]{Mat}).
By Theorem \ref{thm--rel--ample}, we see that $\Lambda_{\mathrm{CM},t}$ is $\bar{\xi}^\nu$-ample for any sufficiently large $t\in\mathbb{Q}_{>0}$.

On the other hand, by the proof of \cite[Theorem 5.1]{HH}, there exist a connected semisimple linear algebraic group $G$ and a quasi-projective normal scheme $S$ such that $G$ acts on $S$ and $[S/G]\cong (\mathcal{M}^{W,I}_{d,v,u,w,p_a,V_1})^{\nu}$.
Let $f\colon (\mathcal{X},\mathcal{A})\to \mathcal{C}$ be the family of uniformly adiabatically K-stable polarized klt--trivial fibrations over curves over a normal quasi-projective variety $S$ corresponding to the composition of the natural morphism $S\to \mathcal{M}^{W,I}_{d,v,u,w,p_a,V_1}$ and $\mathcal{M}^{W,I}_{d,v,u,w,p_a,V_1}\to \mathcal{M}_{d,I^{d-1}\cdot v,u,I^d\cdot w}$.
Note that $\mathcal{A}$ is a line bundle in this case.
As Setup \ref{stup--112}, there exist $\mathbb{Q}$-Cartier $\mathbb{Q}$-divisors $D_0,D_1$ and $D_2$ on $S$ such that $$\lambda_{\mathrm{CM},\pi_{\mathcal{X}},(\mathcal{X},\mathcal{A}+q'f^*\mathcal{L})}^{\otimes (\mathcal{A}_t^d+dq'\mathcal{A}_t^{d-1}\cdot \mathcal{L}_t)^2}=q'^2D_2+q'D_1+D_0,$$ where $t\in S$ is an arbitrary point for any sufficiently large $q'\in\mathbb{Q}_{>0}$.
Note that $D_0$, $D_1$ and $D_2$ admit $G$-linearlizations in a natural way as $\mathbb{Q}$-line bundles.
Since $G$ is semisimple, we have a $\mathbb{Q}$-linear equivalence $D_2\sim_{\mathbb{Q}}\frac{d l_0v}{2Iw}D_1$ that preserves the $G$-linearizations by Proposition \ref{prop--approx--CM}.
Therefore, there exist unique $\mathbb{Q}$-line bundles $L_0$, $L_1$ and $L_2$ on $(M^{W,I}_{d,v,u,w,p_a,V_1})^{\nu}$ such that (cf.~\cite[Theorem 10.3]{alper})
\[
\Lambda_{\mathrm{CM},q'l_0}^{\otimes (\mathcal{A}_t^d+dq'\mathcal{A}_t^{d-1}\cdot \mathcal{L}_t)^2}=q'^2L_2+q'L_1+L_0,
\]
and $L_2\sim_{\mathbb{Q}}\frac{d l_0v}{2Iw}L_1$.
Then, $L_2,L_1,\Lambda_{\mathrm{CM},\infty}|_{(M^{W,I}_{d,v,u,w,p_a,V_1})^\nu}\in\mathbb{Q}_{>0}\cdot(\overline{\beta}^\nu)^*\Lambda^{\mathrm{qmaps}}_{\mathrm{CM}}$ as $\mathbb{Q}$-linear equivalence classes by Proposition \ref{prop--approx--CM} and what we have seen just before this theorem, and $(\overline{\alpha}^\nu)^*\Lambda^{\mathrm{qmaps}}_{\mathrm{CM}}$ is ample by the quasi-finiteness.
By Theorem \ref{thm--rel--ample}, we see that $L_0$ is $\xi^\nu$-ample because $\Lambda_{\mathrm{CM},t}$ is $\bar{\xi}^\nu$-ample for any sufficiently large $t\in\mathbb{Q}_{>0}$.
By \cite[Tag 0892]{Stacks} and \cite[II, Exercise 7.5]{Ha}, we obtain the proof in the case where $u>0$.

Finally, we deal with the case where $u<0$.
By Propositions \ref{prop--unnec--u<0} and \ref{prop-q-proj--N}, we see that $\overline{\beta}$ is quasi-finite.
By this fact and Theorems \ref{thm--logfanoqmaps} and \ref{thm--previous-hh}, the assertions in the case where $u<0$ follow from the same argument as \cite[Theorem 4.8 and Corollary 4.13]{HH2}.
We complete the proof.
\end{proof}



\begin{thebibliography}{BCHM10}

\bibitem[AV02]{AV} D.~Abramovich, A.~Vistoli, Compactifying the space of stable maps. J.~Amer.~Math.~Soc. {\bf 15}, no.~1 (2002), 27--75.

\bibitem[AEHK25]{AEHK}H.~Akaike, M.~Enokizono, M.~Hattori, Y.~Koto, Normal stable degenerations of Noether-Horikawa surfaces. arXiv:2507.17633




\bibitem[Ale94]{Al}
V.~Alexeev, 
Boundedness and $K^2$ for log surfaces, 
Internat. J. Math., {\bf5} (6), (1994), 779–810.

\bibitem[Alp13]{alper} J.~Alper, Good moduli spaces for Artin stacks, Annales de l'Institut Fourier, {\bf63} (2013), 2349--2402. 

\bibitem[AHLH23]{AHLH} J. Alper, D. Halpern-Leistner, and J.~Heinloth, Existence of moduli spaces for algebraic stacks. Invent. Math. {\bf 234} (2023), 949--1038.

\bibitem[Amb04]{A} F.~Ambro, Shokurov's boundary property. J. Differential Geom. {\bf 67} (2004), no. 2, 229--255.

\bibitem[Amb05]{Am} F.~Ambro, The moduli b-divisor of an lc-trivial fibration. Compos. Math. {\bf 141}(2) (2005), 385--403.

\bibitem[ABB+23]{ABB+} K.~Ascher, D.~Bejleri, H.~Blum, K.~DeVleming, G.~Inchiostro,
Y.~Liu, and X.~Wang. Moduli of boundary polarized Calabi-Yau pairs.
arXiv:2307.06522.

\bibitem[Bai61]{Ba} W.L.~Baily, On the automorphism group of a generic curve of genus $>2$, J. Math. Kyoto Univ. {\bf1} (1961), 101-108.


\bibitem[BFMT25]{BFMT}  B.~Bakker, S.~Filipazzi, M.~Mauri, J.~Tsimerman, Baily--Borel compactifications of period images and the b-semiampleness conjecture, arXiv:2508.19215.

\bibitem[BHPV04]{BHPV} W. Barth, K. Hulek, C. Peters, and A. Van de Ven, {\it Compact complex surfaces}, second edition, Springer-Verlag, Berlin, 2004.

\bibitem[Bir19]{Bir-bab} C. Birkar, Anti-pluricanonical systems on Fano varieties, Ann. of Math., {\bf190} (2019), no. 2, 345-463.

\bibitem[Bir21]{birkar-nefpart}
C.~Birkar, 
Boundedness and volume of generalised pairs, preprint (2021), arXiv:2103.14935v2. 

\bibitem[Bir23]{Bir} C.~Birkar, Geometry of polarised varieties. Publ.~Math.~Inst.~Hautes~\'Etudes Sci. {\bf137} (2023), 47–105.




\bibitem[BCHM10]{BCHM} C. Birkar, P. Cascini, C. D. Hacon, and J. M$^\mathrm{c}$Kernan, Existence of minimal models for varieties of log general type, J. Amer. Math. Soc. {\bf23} (2010), no. 2, 405--468.


\bibitem[BJ20]{BlJ} H.~Blum, M.~Jonsson. Thresholds, valuations, and K-stability, Adv. Math. {\bf365} (2020).

\bibitem[BL22]{BL} H.~Blum, Y.~Liu. Openness of uniform K-stability in families of $\mathbb{Q}$-Fano varieties, Ann. Sci. \'{E}cole. Norm. Sup., (4) {\bf55} (2022), no. 1, 1--41.





\bibitem[BDPP13]{BDPP}  S.~Boucksom, J.~P.~Demailly, M.~P\u{a}un and T.~Peternell, The pseudo-effective cone of a compact Kähler manifold and varieties of negative Kodaira dimension. J.~Algebraic Geom. {\bf22} (2013), 201--248 

\bibitem[BHJ17]{BHJ} S. Boucksom, T. Hisamoto, M. Jonsson. Uniform K-stability, Duistermaat-Heckman measures and singularities of pairs, Ann. Inst. Fourier (Grenoble) {\bf 67}(2), (2017) 743-841.







\bibitem[CPZ15]{CPZ}X.X. Chen, M. P\u{a}un, Y. Zeng. On deformation of extremal metrics. arXiv preprint arXiv:1506.01290 (2015).

\bibitem[CP21]{CP} G. Codogni, Zs. Patakfalvi. Positivity of the CM line bundle for families of K-stable klt Fano varieties. Invent. Math. {\bf223}  (2021), 811-894.

\bibitem[CP23]{CP2} G. Codogni, Zs. Patakfalvi. A Note on Families of K-Semistable Log-Fano Pairs. Birational Geometry, K\"{a}hler–Einstein Metrics and Degenerations. (2023), 195--203.





\bibitem[Der16]{De2} R. Dervan, Uniform stability of twisted constant scalar curvature K\"{a}hler metrics, Int. Math. Res. Not. IMRN {\bf 15} (2016), 4728-4783.

\bibitem[Der26]{Dervan} Corrigendum to ``Moduli of polarised manifolds via canonical K\"{a}hler metrics'', Available on R. Dervan's website.

\bibitem[DN18]{DN} R.~Dervan, P.~Naumann. Moduli of polarised manifolds via canonical K\"{a}hler metrics, to appear in Ann. Inst. Fourier.






\bibitem[Don02]{Dn2} S. K. Donaldson, Scalar curvature and stability of toric varieties, J. Differential Geom. {\bf62} (2002), no. 2, 289--349.



\bibitem[FC90]{FC} G. Faltings, C.-L. Chai, {\it Degeneration of abelian varieties}, Ergeb. Math. Grenzgeb. (3), {\bf22},
Springer-Verlag, Berlin, 1990.

\bibitem[F+05]{FGA} B. Fantechi, et. al. {\it Fundamental Algebraic Geometry: Grothendieck's FGA Explained}. Amer. Math. Soc. 2005.



\bibitem[FR06]{FR06}J.~Fine and J.~Ross, A note on positivity of the CM line bundle, Int. Math. Res. Not.
(2006), Art. ID 95875

\bibitem[FS90]{FS}A.~Fujiki, G.~Schumacher. The moduli space of extremal compact K\"{a}hler manifolds and generalized Weil-Petersson metrics, Publ.~Res.~Inst.~Math.~Sci. {\bf26} (1990), 101--183.



\bibitem[F12]{fujino-bpf} O.~Fujino, Basepoint-free theorems: saturation, $b$-divisors, and canonical bundle formula, Algebra Number Theory {\textbf{6}} (2012), no.~4, 797--823.

\bibitem[F17a]{fujino-eff-slc} O.~Fujino, Effective basepoint-free theorem for semi-log canonical surfaces, Publ. Res. Inst. Math. Sci. {\textbf{53}} (2017), no. 3, 349--370. 

\bibitem[F18]{fujino-semi-positivity} O.~Fujino, Semipositivity theorems for moduli problems, Ann. of Math., {\textbf{187}} (2018), no.~3, 639--665. 


\bibitem[F22]{fujino-slc-trivial} 
O.~Fujino, Fundamental properties of basic slc-trivial fibrations I, Publ. Res. Inst. Math. Sci. {\textbf{58}} (2022), no.~3, 473--526. 

\bibitem[FG14]{FG} O. Fujino, Y. Gongyo, On the moduli b-divisors of lc-trivial fibrations, Ann. Inst. Fourier (Grenoble) {\bf64}(4), 1721-1735, 2014. 

\bibitem[FM23]{FM3} O.~Fujino, K.~Miyamoto, Nakai--Moishezon ampleness criterion for real line bundles. Math. Ann. {\bf385} (2023), 459–470. 






\bibitem[Fuj19]{Fjtop} K.~Fujita, Openness results for uniform K-stability, Math.~Ann. {\bf 373}(2019), 1529--1548.


\bibitem[FO18]{FO} K. Fujita, Y. Odaka, On the K-stability of Fano varieties and anticanonical divisors. Tohoku Math. J. {\bf70} (2018), 511--521.







\bibitem[Giv98]{givental} A.~Givental, A mirror theorem for toric complete intersections, in “Topological
field theory, primitive forms and related topics (Kyoto, 1996)”, Progr. Math.,
{\bf160}, Birkhuser Boston, Boston, MA, 1998, 141–175.

\bibitem[Gon13]{gongyo}Y.~Gongyo, Abundance theorem for numerically trivial log canonical divisors of semi-log canonical pairs. J. Algebraic Geom. {\bf22} (2013), 549--564

\bibitem[Gro96]{Gro} M.~Gross, The deformation theory of Calabi-Yau $n$-folds with canonical singularities can be obstructed, Essays on Mirror Manifolds, II, (1996) 401–411.

\bibitem[GHJ03]{GHJ} M. Gross, D. Huybrechts, and D. Joyce. Calabi-Yau manifolds and related geometries. Universitext. Springer-Verlag, Berlin, 2003. Lectures from the Summer School held in Nordfjordeid, June 2001.

\bibitem[SGA1]{SGA} A.~Grothendieck, Rev\^etement \'etales et groupe fondamental (SGA1). Lecture Notes in Math. {\bf288} (1971).




\bibitem[HP10]{HP}P.~Hacking, Y.~Prokhorov, Smoothable del Pezzo surfaces with quotient singularities, Compositio Math.~{\bf146} (2010), 169--192


\bibitem[HMX18]{hmx-boundgentype}C.~D.~Hacon, J.~M\textsuperscript{c}Kernan, C.~Xu, Boundedness of moduli of varieties of general type, J. Eur. Math. Soc. {\textbf{20}} (2018), no.~4, 865--901. 



\bibitem[Har77]{Ha} R. Hartshorne. {\it Algebraic Geometry}, Springer-Verlag, 1977.





\bibitem[HH25a]{HH} K. Hashizume, M. Hattori, On boundedness and moduli spaces of K-stable Calabi-Yau fibrations over curves, Geom.~Topol. (3), {\bf29} (2025), 1619--1691.

\bibitem[HH25b]{HH2} K. Hashizume, M. Hattori, K-moduli of quasimaps and on quasi-projectivity of K-moduli of Calabi-Yau fibrations over curves, arXiv:2504.21519





\bibitem[Hat21]{Hat2} M. Hattori,  A decomposition formula for J-stability and its applications, arXiv:2103.04603, to appear in Michigan Math. J.



\bibitem[Hat24]{CM} M. Hattori, Minimizing CM degree and special K-stable varieties, Int. Math. Res. Not. IMRN, {\bf 2024} (2024), 5728--5772.

\bibitem[Hat25a]{Hat} M. Hattori, On K-stability of Calabi-Yau fibrations,  J.~Inst.~Math.~Jussieu. (3), {\bf24} (2025) 961--1019. 






\bibitem[Hat25b]{Hat23} M.~Hattori, Special K-stability and positivity of CM line bundles. Publ.~Res.~Inst. Math.~Sci. {\bf61} (2025), 447--486.

\bibitem[HL10]{HL}
D.~Huybrechts, M.~Lehn,
  The geometry of moduli spaces of sheaves. Cambridge University Press, 2010






\bibitem[KeMo97]{KeM} S.~Keel, S.~Mori, Quotients by Groupoids. Ann. of Math. {\bf145}, no. 1 (1997), 193--213. 

\bibitem[KM76]{KnMu} F. Knudsen and D. Mumford. The projectivity of the moduli space of stable curves I: Preliminaries on “det” and “Div”. Math. Scand., 39(1), 1976.

\bibitem[Ko85]{kollar-toward} J.~Koll\'ar, Toward moduli of singular varieties, Compositio Math., {\bf56} (1985) no. 3, 369--398.

\bibitem[Ko90]{kollar-moduli-stable-surface-proj} J.~Koll\'ar, Projectivity of complete moduli. J. Differential Geom. {\textbf{32}} (1990) 235--268.

\bibitem[Ko93]{kollar-eff-basepoint-free} J.~Koll\'ar, Effective base point freeness, Math. Ann. {\textbf{296}} (1993), no.~4, 595--605.

\bibitem[Ko11]{Koex} J.~Koll\'ar, Two examples of surfaces with normal crossing singularities. Sci. China Math. {\bf54} (2011), 1707--1712.



\bibitem[Ko13]{kollar-mmp} J. Koll\'{a}r. {\it Singularities of the Minimal Model Program.} Cambridge Tracts in Mathematics {\bf 200}. Cambridge University Press, Cambridge, 2013.

\bibitem[Ko23]{kollar-moduli} J. Koll\'{a}r. {\it Families of varieties of general type}. 
Cambridge Tracts in Mathematics {\bf231}. Cambridge University Press, Cambridge, 2023


\bibitem[KM98]{KM} J. Koll\'{a}r and S. Mori. {\it Birational geometry of algebraic varieties.} Cambridge Tracts in Mathematics {\bf 134}. Cambridge University Press, Cambridge, 1998

\bibitem[KSB88]{KSB}
J.\ Koll\'ar and N.\ Shepherd-Barron, 
Threefolds and deformations of surface singularities,
Invent.\ Math.\ \textbf{91} (1988), 299--338.

\bibitem[KX20]{KX} J.~Koll\'ar,  C.~Xu, Moduli of Polarized Calabi--Yau Pairs. Acta. Math. Sin.-English Ser. {\bf36}, 631--637 (2020).

\bibitem[KP17]{KP} S.~Kov\'acs, Zs.~Patakfalvi, Projectivity of the moduli space of stable log-varieties and subadditivity of log-Kodaira dimension. J. Amer. Math. Soc. {\bf30} (2017), 959--1021.



\bibitem[L04]{Laz} R. Lazarsfeld. {\it Positivity in algebraic geometry.} I. Ergebnisse der Mathematik und ihrer Grenzgebiete. 3. Folge, Springer. (2004).





\bibitem[LXZ22]{LXZ} Y. Liu, C. Xu, and Z. Zhuang. Finite generation for valuations computing stability thresholds and applications to K-stability. Ann. of Math. {\bf196}(2) (2022), 507--566.



\bibitem[MOP11]{MOP}A.~Marian, D.~Oprea, R.~Pandharipande, The moduli space of stable quotients, Geom. Topol. {\bf15} (2011), 1651--1706.



\bibitem[Mat80]{Mat} H. Matsumura, {\it Commutative algebra}. Benjamin. 2nd. edition. (1980).






\bibitem[MP95]{DP2}D.~Morison and R.~Plesser, Summing the instantons: quantum cohomology and mirror symmetry in toric varieties, Nuclear Phys. B {\bf440} (1995), no. 1-2, 279–354.


\bibitem[M74]{Ab} D. Mumford, {\it Abelian varieties}, Oxford university press, (1974).



\bibitem[MFK94]{GIT} D. Mumford, J. Fogarty, F. Kirwan, {\it Geometric Invariant Theory}, 3rd. edition, Ergebnisse der Mathematik und ihrer Grenzgebiete, {\bf34}, Springer-Verlag (1994).




\bibitem[Oda10]{O2} Y. Odaka, On the GIT stability of polarized varieties: a survey, Proceeding of Kinosaki algebraic geometry symposium 2010.
\bibitem[Oda13a]{Odaka} Y.~Odaka, A generalization of Ross-Thomas slope theory, Osaka J. Math. {\bf50} (2013), 171--185.
\bibitem[Oda13b]{O3} Y. Odaka, {\it On the moduli of K\"{a}hler-Einstein Fano manifolds}, Proceeding of Kinosaki algebraic geometry symposium 2013, available at arXiv:1211.4833.
\bibitem[Ols16]{Ols} M. Olsson, {\it Algebraic spaces and stacks}. Colloquium Publications {\bf62}, American Mathematical Society, 2016.



\bibitem[PX17]{PX} Zs. Patakfalvi, C.~Xu, Ampleness of the CM line bundle on the moduli space of canonically polarized varieties. Algebraic Geometry {\bf4} (1) (2017) 29--39.

\bibitem[PT09]{PT} S. Paul, G. Tian. CM stability and the generalized Futaki invariant II. Ast\'{e}risque No. {\bf328} (2009), 339-354.

\bibitem[Pos22]{P} Q.~Posva, Positivity of the CM line bundle for K-stable log Fanos, Trans.~Amer.~Math.~Soc. {\bf375} (2022), 4943--4978.





\bibitem[Sta]{Stacks} The authors of stacks project, Stacks project.

\bibitem[Tod11]{Toda} Y.~Toda, Moduli spaces of stable quotients and wall-crossing phenomena. Compositio Math.~(5), {\bf147} (2011), 1479-1518.

\bibitem[Tia97]{T} G. Tian, K\"{a}hler-Einstein metrics with positive scalar curvature, Invent. Math. {\bf130} (1997), no. 1, 1--37.



\bibitem[Vie91]{viehweg91} E.~Viehweg, Quasi-projective quotients by compact equivalence relations, Math.~Ann.~{\bf 289} (1991), 297--314.


\bibitem[Vie95]{viehweg95} E.~Viehweg, {\it Quasi-Projective Moduli for Polarized Manifolds}, Ergebnisse der Mathematik und ihrer Grenzgebiete (3), vol.~{\bf30}, Springer-Verlag, Berlin, 1995.

\bibitem[Vis89]{Vis}A. Vistoli, Intersection theory on algebraic stacks and on their moduli spaces, Invent. Math. {\bf97} (1989), 613–670.

\bibitem[WX14]{WX} X. Wang, C. Xu, Non existence of asymptotic GIT compactification, Duke Math. J. {\bf163} (2014), 2217--2241.

\bibitem[Xu25]{Xu} C.~Xu, {\it K-stability of Fano varieties}, New Mathematical Monographs {\bf50}, Cambridge University
Press, Cambridge, 2025.


\bibitem[XZ20]{XZ} C.~Xu, Z.~Zhuang, On positivity of the CM line bundle on K-moduli spaces, Ann.~of Math.~{\bf192} (2020), 1005--1068. 

\bibitem[XZ21]{XZuniqueness} C.~Xu, Z.~Zhuang, Uniqueness of the minimizer of the normalized volume function, Cambridge J. Math. {\bf9} (2021), 149--176. 



\end{thebibliography}
\end{document}